\documentclass[10pt,a4paper]{amsart}

\usepackage{fullpage}

\usepackage{setspace}
\usepackage{color}
\usepackage{quoting}
\usepackage{dutchcal}
\usepackage{amsmath}
\usepackage{leftidx}
\usepackage{amsthm}
\usepackage{amssymb}
\usepackage{algorithm}
\usepackage{algpseudocode}
\usepackage{mathtools}
\usepackage[breaklinks]{hyperref}
\usepackage{cleveref}
\usepackage{xstring}
\usepackage{pdfsync}
\usepackage{caption}
\usepackage{subcaption}
\usepackage{lipsum}
\usepackage{amsfonts}
\usepackage{graphicx}

\usepackage{etoolbox}
\usepackage[normalem]{ulem}

\usepackage[backend=biber, style=numeric, sorting=nyt]{biblatex}

\DeclarePairedDelimiter\floor{\lfloor}{\rfloor}%
\DeclarePairedDelimiter\abs{\lvert}{\rvert}%
\makeatletter
\let\oldabs\abs
\def\abs{\@ifstar{\oldabs}{\oldabs*}}

\newtheorem{theorem}{Theorem}
\newtheorem{definition}{Definition}
\newtheorem{proposition}[definition]{Proposition}
\newtheorem{lemma}{Lemma}

\crefname{lemma}{Lemma}{Lemmas}
\crefname{proposition}{Proposition}{Propositions}

\numberwithin{equation}{section}
\numberwithin{lemma}{section}

\newcommand{\subtitle}[1]{%
  \posttitle{%
    \vskip3em
    \par\end{center}
    \begin{center}\LARGE#1\end{center}
    \vskip0.5em}%
}

\newcommand\newpar{\vspace{5mm}\newline}
\newcommand\normalizer{\frac{1}{2\pi}}
\newcommand\oyInZ{\oy\in\mathbb{Z}}
\newcommand\rgamgam{\mathcal{R}_{\Gamma|\gamma}}
\newcommand\h{\mathcal{H}}
\newcommand{\hl}[1][l]{\mathcal{H}_{#1}}
\newcommand\oy{\omega_{y}}
\newcommand\dxi{d_{\xi}}
\newcommand\doy{d_{\omega_{y}}}
\newcommand\psiOy{\psi_{\oy}}

\newcommand\phiOy{\phi_{\oy}}
\newcommand\gammaOy{\gamma_{\oy}}
\newcommand\mxi{M_{\xi}}

\newcommand\ckPsiOy{c_{k}(\psi_{\oy})}

\newcommand\Aoyl{A_{l}^{\psi}(\oy)}
\newcommand{\Aoylarg}[1][l]{A_{#1}^{\psi}(\oy)}
\newcommand\Vpsi{V_{\psi, l}}
\newcommand\PCoy{PC^{(d+1, 1)}_{\mathbb{T}}}

\newcommand{\gkma}[3][k]{g_{#1,#2,#3}(x)}
\newcommand\ox{\omega_{x}}
\newcommand\C{C^{(d+1)}_{\mathbb{T}}}

\newcommand\PC{PC^{(d+1, 1)}_{\mathbb{T}}}
\newcommand\Fx{F_{x}}
\newcommand\Admxi{A_{d,\mxi}}
\newcommand{\Al}[1][l]{A_{#1}(x)}
\newcommand\moyx{m_{\oy}(x)}
\newcommand\xix{\xi(x)}
\newcommand\yx{y_{x}}
\newcommand\epsOyx{\varepsilon_{\oy}(x)}
\newcommand\kappax{\kappa_{x}}
\newcommand\gammax{\Gamma_{x}}
\newcommand\Phidx{\Phi_{d,x}}
\newcommand\Moy{M_{\oy}}

\newcommand\cOyFx{c_{\oy}(\Fx)}

\newcommand\ddx{\frac{d}{dx}}
\newcommand\ppx{\frac{\partial}{\partial x}}
\newcommand{\ddxk}[1][k]{\frac{d^{#1}}{dx^{#1}}}
\newcommand{\ddyk}[1][k]{\frac{d^{#1}}{dy^{#1}}}
\newcommand\dkdxk{\frac{d^{k}}{dx^{k}}}
\newcommand\pkpxk{\frac{\partial^{k}}{\partial x^{k}}}

\newcommand\delOyx{\delta_{\oy}(x)}

\newcommand\mOxOy{m_{\ox}(\oy)}
\newcommand\delOxOy{\delta_{\ox}(\oy)}

\newcommand\cosoyxi{\cos(\oy\xix)}
\newcommand\sinoyxi{\sin(\oy\xix)}
\newcommand\xii[1][i]{\xi_{#1}}
\newcommand\oypm{\omega_{\pm,y}}
\newcommand\oyapm[1][a]{\omega_{\pm,y}^{#1}}

\newcommand\trivec[1][k]{\vec{\triangle}_{#1}}
\newcommand\omegavec[1][k]{\vec{\Omega}_{#1}}
\newcommand{\trilm}[2][l]{\triangle_{#1}^{(#2)}}

\newcommand{\ppxk}[1][k]{%
    \IfEqCase{#1}{%
            {1}{\ppx F(x,y) dy}%
            {k}{\pkpxk F(x,y) dy}}%
    [\frac{\partial^{#1}}{\partial x^{#1}}F(x,y) dy]}

\newcommand{\ik}[1][k]{I_{\oy, #1}(x)}
\newcommand\fcs{f_{c\vert s}(x)}
\newcommand\fcsk[1][k]{f_{c\vert s}^{(#1)}(x)}
\newcommand{\torus}[1][2]{%
    \IfEqCase{#1}{%
            {1}{\mathbb{T}}}%
    [\mathbb{T}^{#1}]}
\newcommand{\LIK}[1][k]{%
    \IfEqCase{#1}{%
            {0}{\int_{-\pi}^{\xi(x)}{e^{-\imath y \oy}F(x,y) dy}}}%
    [\int_{-\pi}^{\xi(x)}{e^{-\imath y \oy}\ppxk[#1]}]}
\newcommand{\RIK}[1][k]{%
    \IfEqCase{#1}{%
            {0}{\int_{\xi(x)}^{\pi}{e^{-\imath y \oy}F(x,y) dy}}}%
    [\int_{\xi(x)}^{\pi}{e^{-\imath y \oy}\ppxk[#1]}]}

\newcounter{parnum}

\begin{document}

\title{Algebraic Reconstruction of Piecewise-Smooth Functions of Two Variables from Fourier Data}

\author{Michael Levinov}
\address{Department of Applied Mathematics, Tel Aviv University, Tel Aviv, Israel.}

\author{Yosef Yomdin}
\address{Department of Mathematics, Weizmann Institute of Science, Rehovot, Israel.}

\author{Dmitry Batenkov}
\address{Department of Applied Mathematics, Tel Aviv University, Tel Aviv, Israel.}
\curraddr{Basis Research Institute, NYC, USA.}
\email{dima.batenkov@gmail.com}
\thanks{Corresponding author: D.~Batenkov}

\keywords{Fourier analysis, piecewise smooth functions, edge detection, super-resolution, inverse problems}

\subjclass[2010]{Primary 65T40; Secondary 65D15}

\begin{abstract}
    We investigate the problem of reconstructing a 2D piecewise smooth function from its bandlimited Fourier measurements.
    This is a well known and well studied problem with many real world implications, in particular in medical imaging.  While many techniques have been proposed over the years to solve the problem, very few consider the accurate reconstruction of the discontinuities themselves.

    In this work we develop an algebraic reconstruction technique for two-dimensional functions consisting of two continuity pieces with a smooth discontinuity curve. By extending our earlier one-dimensional method, we show that both the discontinuity curve and the function itself can be reconstructed with high accuracy from a finite number of Fourier measurements. The accuracy is commensurate with the smoothness of the pieces and the discontinuity curve. We also provide a numerical implementation of the method and demonstrate its performance on synthetic data.
\end{abstract}

\maketitle

\section{Introduction}\label{sec:introduction alt}
In this work we revisit the classical problem of approximating a piecewise regular function
$F:\torus\to\mathbb{R}$ from Fourier data:
\begin{gather}\label{eq:fourier-data}
    \widehat{F}(\ox,\oy)\coloneqq \frac{1}{4\pi^{2}}\int_{-\pi}^{\pi} \int_{-\pi}^{\pi} F(x,y) e^{-\imath \ox x} e^{-\imath \oy y}\,dx\,dy\quad, \abs{\ox}\leq M,\;\abs{\oy}\leq N,
\end{gather}
where $\torus[2]$ is the 2D periodic torus $[-\pi,\pi)^2$.
The presence of discontinuities hinders the otherwise fast convergence of the Fourier series of $F$, while also introducing the \textit{Gibbs phenomenon}~\cite{gibbs1898fourier}. Both issues have very serious implications, e.g. when using spectral methods to calculate solutions of PDEs with
shocks~\cite{gottlieb1977numerical} or when  approximating sharp edges, e.g. boundaries of tissues using Magnetic
Resonance (MR) images
obtained via the inverse Fourier transformation~\cite{veraart2016gibbs}.
Therefore, an important questions arises:
\begin{quoting}
    \textit{Can piecewise smooth functions be reconstructed from their Fourier measurements with accuracy comparable to
    their smooth counterparts?}
\end{quoting}

In the rest of the introduction, we first establish some notation (\cref{subsec:problem setting}), then present our main findings (\cref{subsec:summary of our results}), and finally compare our results with prior works (\cref{subsec:closely related work}).

\subsection{Problem setting}\label{subsec:problem setting}

In our model we will develop a super-resolving technique (from Fourier measurements) for a piecewise-smooth 2D function
$F:\torus\to\mathbb{R}$ with only two continuity pieces (as described in~\cref{fig:graph_for_algebraic reconstruction method_02})
such that the boundary curve can be represented by the graph of a smooth function $y = \xi(x)$
(as described in~\cref{fig:graph_for_algebraic reconstruction method_01}).
\begin{figure}[ht]
    \centering
    \begin{subfigure}{0.55\linewidth}
        \centering
        \includegraphics[width=.99\linewidth]{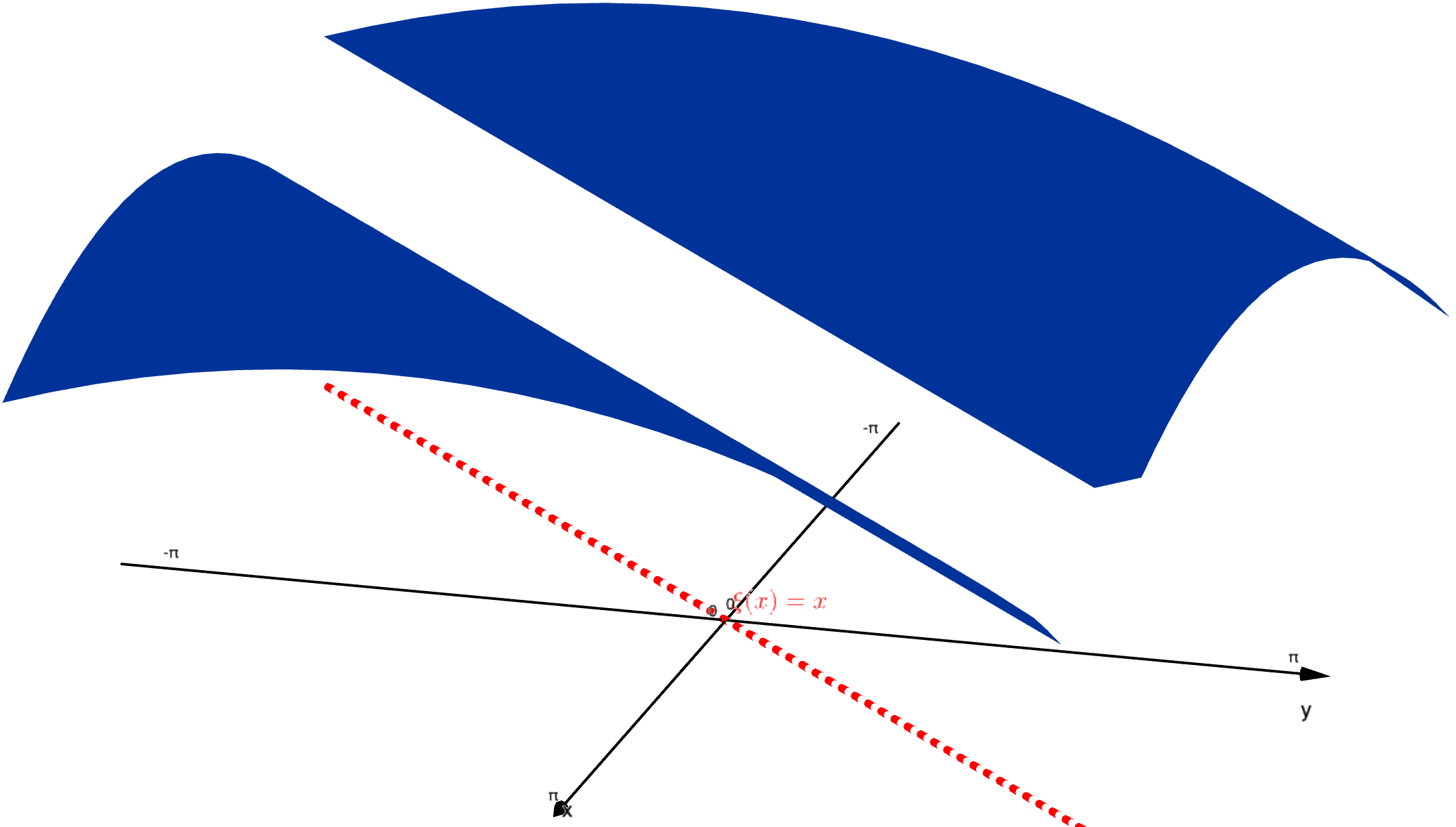}
        \caption{}
        \label{fig:graph_for_algebraic reconstruction method_02}
    \end{subfigure}
    \begin{subfigure}{.4\linewidth}
        \centering
        \includegraphics[width=.99\linewidth]{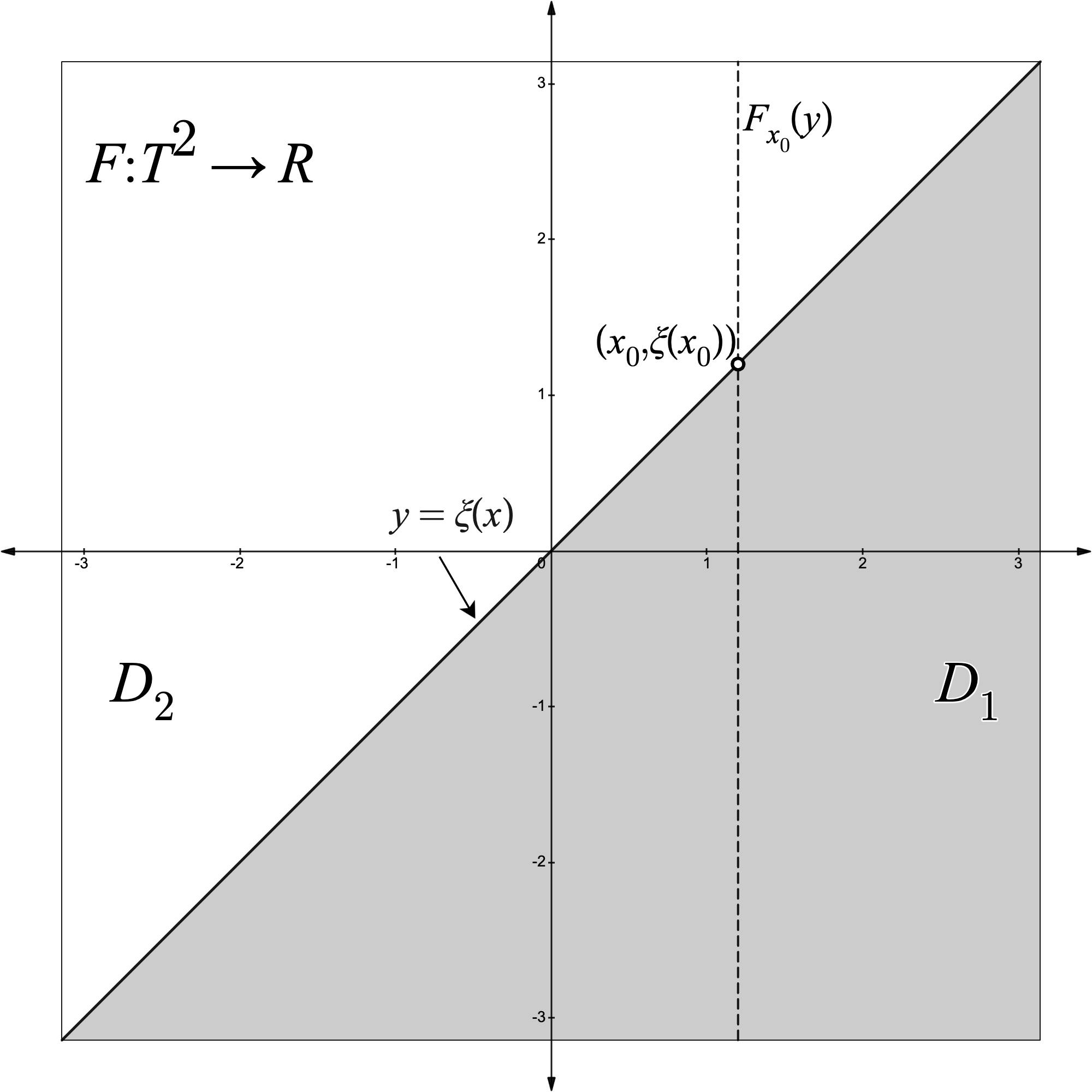}
        \caption{}
        \label{fig:graph_for_algebraic reconstruction method_01}
    \end{subfigure}
    \caption{Our model of the piecewise-smooth 2D function $F:\torus\to\mathbb{R}$ with two continuity pieces. \eqref{fig:graph_for_algebraic reconstruction method_02} The graph of $F$. \eqref{fig:graph_for_algebraic reconstruction method_01} The continuity domains $D_1,D_2$ and and the jump curve $\Sigma$.}
    \label{fig:graph_for_algebraic reconstruction method}
\end{figure}

\begin{definition}
    \label{C-continuity-def}
    Let $C^{(d+1)}_{I}$ denote the class of continuous functions having $d$-continuous derivatives in $I\subseteq\torus[1]$, such that $f^{(d+1)}$
    is piecewise-continuous and piecewise-differentiable in $I$.
\end{definition}

\begin{definition}
    \label{PC-continuity-def}
    Let $PC^{(d+1, K)}_{I}$ denote the class of functions $f$ with $K$ discontinuity points,
    $-\pi\leq\xi_{1} < \xi_{2} < \ldots < \xi_{K} < \pi$ s.t. $f \in C^{(d+1)}_{[\xi_{i}, \xi_{
        i+1}]}$ for $1 \leq i \leq K-1$.
\end{definition}

\begin{definition}
    \label{description-of-F-and-xi}
    Let $F:\torus \rightarrow \mathbb{R}$. For each $x\in\mathbb{T}$ let $\Fx: \mathbb{T} \to\mathbb{R}$ denote the ``slice'' $F_x(y) \coloneqq F(x,y)$,
    and assume that for each $x\in\mathbb{T}$ we have $F_x \in\PC$.
    Also denote the smooth boundary curve, $\Sigma$, as a function of $x$ by:
    \begin{equation}\label{eq:boundary-curve}
       \Sigma = \{\left(x,y\right): y = \xi(x),\; x\in\mathbb{T} \}
    \end{equation}
    and assume that $\xi(x) \in C_{(-\pi,\pi)}^{\dxi}$ for some $\dxi\in\mathbb{N}$, $\dxi\geq d$.\newline

    Now we define the two continuity pieces of $F$ and their domains:
    \begin{equation} \label{eq: continuity-pieces-F}
        \begin{split}
            F_{1} &:D_{1}\rightarrow\mathbb{R},\text{ where } D_{1} \coloneqq \left\{ (x,y)\;\Big|\;-\pi\leq y\leq \xix \right\}\\
            F_{2} &:D_{2}\rightarrow\mathbb{R},\text{ where } D_{2} \coloneqq \left\{ (x,y)\;\Big|\; \xix < y < \pi \right\}.
        \end{split}
    \end{equation}
    Our function $F$ can therefore be written as follows:
    \begin{equation}
        \label{eq:2.3}
        F(x,y) =
        \begin{cases}
            F_1(x,y) & \text{if } \left( x, y \right) \in D_1 \\
            F_2(x,y) & \text{if } \left( x, y \right) \in \torus \setminus D_1.
        \end{cases}
    \end{equation}
\end{definition}

\noindent{\bf Remark:} note that the assumptions above imply that each slice $F_x$ has a single jump at $y_x=\xi(x)$, and is otherwise periodic and smooth. This assumption can be relaxed so that $F_x$ has also a jump at the endpoints, but we will not consider this case here. 

In this paper, we present a technique that, given any $x\in\torus[1]$, reconstructs the slice $F_{x}(y)$ by approximating the jump location $y_x=\xi(x)$ and the jump magnitudes of $\frac{\partial^k}{\partial y^k} F_x(y)$ at $y_x$
(see~\cref{Fx-jump-mag-def} below) for $0\leq k \leq
d_r$ where $d_r\leq d$ is the chosen reconstruction order.
Our proposed method utilizes and combines the two methods described in~\cite{batenkov2012algebraic, batenkov2015complete} (which will be presented and analyzed going forward).

\subsection{Summary of our results}\label{subsec:summary of our results}

To state our results, let us first establish some additional notation related to our model.

\begin{definition}
    \label{definition-of-Fx-and-jump-location-and-jump-magnitudes-and-assumptions}
    Let $F:\torus \rightarrow \mathbb{R}$ with $F_x$ as described in~\cref{description-of-F-and-xi} where
    $x\in\mathbb{T}$, and let $A_{l}(x)$ denote the jump magnitude of $\ddyk[l]\left( \Fx \right)$
    at $y_x\coloneqq \xi(x)$:
    \begin{gather}\label{Fx-jump-mag-def}
        A_{l}(x) \coloneqq \lim_{y \to y_{x}^{+}} {\frac{\partial^l}{\partial y^l}\Fx(y)}
        - \lim_{y \to y_{x}^{-}}{\frac{\partial^l}{\partial y^l}\Fx(y)}.
    \end{gather}
    We further denote:
    \begin{itemize}
        \item $M_{\xi} := \sup\limits_{x\in\torus[1]} \left\{ \abs{\frac{d^l}{dx^{l}}\xi(x)}\Bigg|
        \label{bound-for-xi} l=0,\ldots,d \right\}<\infty$.
        \item $B_{F} := \sup\limits_{y\in\torus[1]\setminus\left\{ \xix \right\}}\left\{\abs{\frac{\partial^{k+1}}
        {\partial x^{k+1}}F_x(y)} \Bigg| k=0, \dots, d+1 \right\}<\infty$, for all $x\in\torus[1]$.
        \label{upper-bound-partial-derivatives-fx}
        \item $A_{x} := \max\limits_{l=0, \ldots, d+1 }\left\{\abs{A_{l}(x)}\right\} < \infty$.
        \label{upper-bound-jump-magnitudes-fx}
    \end{itemize}
\end{definition}

We next define the $\psiOy$ function, which is an essential ingredient in our work. 

\begin{definition}\label{def:psi-oy}
    Let $F:\torus[2]\to\mathbb{R}$ be an unknown 2D piecewise-smooth function as defined in~\cref{description-of-F-and-xi}.
    Assume that the given data are the $\left(2M+1 \right) \times \left( 2N+1 \right)$ Fourier coefficients of $F$. For each $\abs{\oy}\leq N$ we define
    \begin{align}\label{def-of-psi-omega-y}
        \begin{split}
        \psiOy(x) &\coloneqq \widehat{F_x}(\oy) = \frac{1}{2\pi} \int_{-\pi}^{\pi} e^{-\imath y\oy} F(x,y)\, dy \\ 
        &= \frac{1}{2\pi} \left(\int_{-\pi}^{\xi(x)} e^{-\imath y\oy} F_1(x,y)\, dy +
        \int_{\xi(x)}^{\pi} e^{-\imath y\oy} F_2(x,y)\,dy \right)
        \end{split}
    \end{align}
    where $x\in \torus[1]$.
\end{definition}

\begin{lemma}\label{psi-oy-properties}
Let $\xi$ as in~\cref{eq:boundary-curve} s.t. $\xi\in C_{(-\pi,\pi)}^{d_{\xi}}$ where $d_{\xi}\geq d$,
then for each $\oyInZ$, $\psiOy$ satisfies the following properties:
    \begin{itemize}
        \item $\psiOy\in \PCoy$.
        \item $\psiOy$ and its first $d$ derivatives have a (single) jump discontinuity at $x=-\pi$.
    \end{itemize}  
\end{lemma}

The proof of \cref{psi-oy-properties} is presented in the appendix. For further developments, we denote by $\Aoyl$  the jump magnitude of
    $\ddxk[l]\left( \psiOy \right)$ at $x=-\pi$:
    \begin{gather}\label{psi-oy-jump-mag-def}
        \Aoyl \coloneqq \psiOy^{(l)}(-\pi) - \psiOy^{(l)}(\pi).
    \end{gather}

The crux of our method is the following.
In~\cite{batenkov2012algebraic, batenkov2015complete} it was shown that a
piecewise-smooth 1D function $f\in PC^{(d+1, K)}_{I}$ can be reconstructued from its Fourier coefficients with high
accuracy by a certain ``algebraic'' reconstruction procedure (\cref{alg:full-order-algorithm}). Our idea is to use
this 1D method for each $x\in \torus[1]$ to approximate $F_x \in\PC$, and for doing so, we need an accurate approximation of
$\widehat{F}_x(\oy)\equiv \psiOy(x)$.
For this latter task, \emph{we can again use the 1D method}, provided sufficiently many Fourier coefficients of
$\psiOy\in \PCoy$.
However, those are given by our data
\begin{equation}\label{psi-oy-fourier-coefficients}
    \widehat{F}(\omega_{x}, \omega_{y})\equiv \widehat{\psi_{\omega_{y}}}(\omega_{x}).
\end{equation}    

To see why \eqref{psi-oy-fourier-coefficients} is true, consider the definition
\begin{gather*}
    \widehat{F}\left( \ox, \oy \right)=
        \frac{1}{4\pi^{2}}\int_{-\pi}^{\pi}{\int_{-\pi}^{\pi}}{F(x,y)e^{-\imath x \ox}e^{-\imath y \oy}}\,dy\,dx.
\end{gather*}
Using Fubini's theorem \cite{fubini1907sugli} and \eqref{def-of-psi-omega-y} we have
\begin{align*}
    \widehat{F}\left( \ox, \oy \right)&= \frac{1}{2\pi}\int_{-\pi}^{\pi}{e^{-\imath x \ox}}\frac{1}{2\pi} \int_{\pi}^{\pi} e^{-\imath y \oy} F(x,y) dy dx\\
    &= \frac{1}{2\pi}\int_{-\pi}^{\pi}{\psiOy(x) e^{-\imath x \ox}}\,dx
    =\widehat{\psi_{\omega_{y}}}(\omega_{x}).
\end{align*}

\vskip 1em
\noindent{\bf Outline of method:}
For a function $F:\torus[2]\to\mathbb{R}$ as in \cref{description-of-F-and-xi} we require its $\left( 2M+1 \right) \times \left( 2N+1\right)$
Fourier coefficients as the input data.
Then for every $\abs{\oy}\leq N$ we approximate $\psiOy$
(see~\cref{def-of-psi-omega-y}) using $\left\{ \widehat{F}(\ox,\oy) \right\}_{\abs{\ox}\leq{M}}$ as the input
for our adaptation of the algorithms described in~\cite{batenkov2012algebraic,batenkov2015complete} which will
yield an approximation $\widetilde{\psiOy}$ for $\Fx$'s Fourier coefficients at $x\in\torus[1]$.
Next we take this approximation of $\left\{ \widehat{\Fx}(\oy) \right\}_{\abs{\oy}\leq N}$ and
use it as an input for our second part of the two-part algorithm to obtain an approximation $\widetilde{F_x}$ to the \textit{slice} $\Fx$.  The entire algorithm is presented in~\cref{alg:2D reconstruction} below.


Our main theoretical result provides upper bounds on the accuracy of recovering each $F_x$, including the position of the
discontinuity curve $\Sigma$ and the magnitudes $A_{l}(x)$ of the jump discontinuities of $\frac{\partial^l}{\partial y^l}F_x$ at $y=\xi(x)$,
as well as the pointwise accuracy of approximating $F_x$ for each $x\in\torus[1]$, using the method outlined above.

\begin{theorem}
    \label{thm:general-theorem-for-fx}
    Let $F:\mathbb{T}^2\to\mathbb{R}$ as in~\cref{description-of-F-and-xi} with $\left(2M+1 \right)\times \left(2N + 1 \right)$ Fourier
    coefficients as in \cref{eq:fourier-data} such that $\forall x\in\torus[1]$ we have $\Fx\in\PC$ as
    in~\cref{definition-of-Fx-and-jump-location-and-jump-magnitudes-and-assumptions}, and assume the following
    relation between $M$ and $N$:
    \begin{equation}\label{relation-between-Mox-and-Moy}
        N^2 \leq M.
    \end{equation}
    Further assume that $\inf_{x\in\torus[1]} |A_0(x)| = A_L > 0$. There exist constants $\rgamgam$, $C_{2,d}$, $C_{6,d-l}$, $C_{10,d}$
    (as defined in~\cref{r-gamma-gamma,bound-on-yx,Ax-tilde-bound}) such that for every $r>0$ there exists $N'(r)$ such that for all $N>N'$
    and $M\geq N'{^2}$ the following bounds hold for all $x\in\mathbb{T}$:
    \begin{equation} \label{main-theorem-for-fx}
        \begin{split}
            \abs{\widetilde{\xi}(x) - \xix} &\leq C_{2,d} \frac{A_x}{\abs{A_{0}(x)}} \cdot \rgamgam \cdot N^{-d-2} \\
            \abs{\widetilde{A_{l}}(x) - A_{l}(x)} &\leq C_{6,d-l} \cdot \frac{A_x}{\abs{A_{0}(x)}} \rgamgam \cdot N^{l-d-1} \\
            \abs{\widetilde{\Fx}(y) - \Fx(y)} &\leq C_{10,d}\frac{A_{x}(1+A_x)}{\abs{A_{0}(x)}} \rgamgam N^{-d-1},\quad \forall y\in\torus[1]\setminus B_r(\xi(x)),
        \end{split}
    \end{equation}
    where $A_x$ and $A_{0}(x)$ are as in~\cref{definition-of-Fx-and-jump-location-and-jump-magnitudes-and-assumptions}, and $B_r(t)$ is a ball of radius $r$ centered at $t$.
\end{theorem}

The key estimate used to prove \cref{thm:general-theorem-for-fx} is the following reconstruction bound on each $\psiOy$ (recall \cref{def:psi-oy}):
\begin{theorem}
    \label{thm:main-result-for-psi}
    Let there be given $\left( 2M + 1 \right)\times \left( 2N + 1 \right)$  Fourier coefficients of $F$ as in \cref{eq:fourier-data}.
    Let $|\oy|\leq N$ be fixed. Then there exist $\h$, $\hl$, $\hl[T]^{*}$ and $R_{d}$
    (as defined respectively in~\cref{ht-star,r-gamma-oy-star-func,h-func,hl-func}) s.t.:
    \begin{gather}\label{eq:bound-for-thm-2}
        \begin{split}
            &\abs{\widetilde{\Aoyl}-\Aoyl} \leq \hl [d-l] A_{x} \abs{\oy}^{d} M^{l-d-1},\quad
            l=0,\ldots, d,\\
            &\abs{\widetilde{\psiOy}(x) - \psiOy(x)} \leq \left( \hl[T]^{*} + R_{d}\h \right) A_{x}
            \abs{\oy}^{d} M^{-d-1}.
        \end{split}
    \end{gather}
\end{theorem}

As one can see, the decay rate in the bound for the error in approximating the jump magnitudes or in the pointwise
accuracy for approximating $\psiOy(x)$ is effected by the number of Fourier coefficients that is used as input and
the order of reconstruction.
Also we notice the effect of a higher order reconstruction, $l$,  diminishes the accuracy of the jump magnitudes
approximation, but all of that is to be expected and seen in other work on this subject (see~\cite{
    batenkov2012algebraic,batenkov2015complete}).
A new phenomenon appears in the 2D case: as $N$ increases,  the error bounds in\eqref{eq:bound-for-thm-2} increase as $N^d$
for a fixed $M$. Therefore we impose the additional assumption in~\eqref{relation-between-Mox-and-Moy} which ensures sufficiently
fast decay in~\eqref{main-theorem-for-fx}. This is most certainly an artifact of our method as we are reconstructing the slices in
a fixed $x$ direction.

\subsection{Related work}\label{subsec:closely related work}

In this section we provide some background and review related literature on the problem of high-accuracy reconstruction of piecewise-smooth functions.

It is well-known that linear trigonometric approximation of smooth periodic functions is minimax optimal on the torus $\mathbb{T}^q$ in the sense of $n$-widths \cite{mhaskar2013applications}. In particular, for functions with $r$ continuous derivatives (more precisely, for which $\Delta^{r/2} f$ is bounded in $\ell^p$ where $\Delta$ is the Laplacian), the approximation error by trigonometric polynomials of total degree $n$ (a space of dimension $O(n^q)$) decays as $n^{-r}$, with the optimal rate achieved by a particular class of summability kernels. This result is valid for $p=\infty$ as well. For functions with singularities, the approximation error decays very fast ``away from the discontinuities''; however, the discontinuities themselves can be localized with resolution at most $1/n$.  

Reconstruction of piecewise-smooth functions in one variable from Fourier data has been investigated extensively. The reader is referred to \cite{batenkov2012algebraic, batenkov2015complete} for a detailed discussion of the problem, its history, and a solution based on parametric reconstruction of the piecewise-polynomial approximation through algebraic techniques. This method, in turn, is based on an earlier work by K.Eckhoff \cite{eckhoff1995accurate}, introducing an essential modification: the algebraic system is solved for a small subset of the entire Fourier data, to attain maximal accuracy. The present work extends these ideas into the two-dimensional case (see also related work \autocite{batenkov2013} on reconstruction from moments). The one-dimensional algebraic reconstruction technique in \cite{batenkov2012algebraic, batenkov2015complete} is closely related to exponential fitting and the inverse problem of recovering sparse measures and distributions from low-frequency data, cf. \autocite{akinshin2021, candes2014, batenkov2013c, batenkov2013b,batenkov2018,batenkov2014b,batenkov2017c, batenkov2023,batenkov2020, batenkov2021, batenkov2021b,katz2023,katz2024b,katz2024c,filbir2012,mhaskar2000a} and some multi-dimensional generalizations thereof \cite{kunis2016, mhaskar2019, poon2019, diab2024, cuyt2018, diederichs2022, sauer2017, sauer2018}; providing robustness estimates for such problems is a nontrivial task. Fourier reconstruction of non-periodic smooth functions (i.e. with endpoint singularities, as our $\psiOy$, recall \cref{psi-oy-properties}) is a classical topic in numerical analysis of spectral methods \cite{gottlieb1977, gottlieb1996,gottlieb1997, chen2014d}, see also \cite{kvernadze2004, kvernadze2010, adcock2014b, barkhudaryan2007, poghosyan2021}.

In \cite{gelb2016} the spectral edge detection method by Gelb\&Tadmor \cite{gelb2007a,tadmor2007} utilizing annihilating filters is extended to two dimensions and nonuniform Fourier samples (see also \cite{adcock2017a}), employing ``line-by-line'' reconstruction similar to our approach, although without explicit accuracy estimates. Similar techniques are applied for functions on a sphere in \cite{blakely2007}. A somewhat related approach is taken in \cite{wasserman2015} utilizing a variational technique with a $\ell_1$ penalty term, based on the annihilating filter method.

A recent work by D.Levin \cite{levin2020} is most closely related to ours. It employs a similar setting of a discontinuity curve and two smooth pieces, and reconstructs the distance function to the boundary curve $\Gamma$ by a 2D spline approximation utilizing an iterative method. An initial mesh for approximating the singularity curve is required, and also a global condition for approximating $\Gamma$ as a zero set of some function from a linear space. The accuracy bounds depend on a certain Lipschitz constant, which is not estimated explicitly.

Several works investigated reconstruction from point-wise data, such as \cite{lipman2010, amir2018}, or cell-averages \cite{amat2023,cohen2024}, where the piecewise-smooth nature of the function is exploited.

Finally we would like to mention the classical computational harmonic analysis techniques based on sparse representations in overcomplete frames such as wavelets \cite{daubechies1992}, curvelets \cite{candes2004new}, shearlets \cite{guo2012optimally, kutyniok2012,schober2021}, wedgelets \cite{donoho1999}, or adaptive approaches such as bandlets and grouplets  \cite{mallat1999wavelet}. While some of these methods provide optimal reconstruction of cartoon-like images, these typically require the representation coefficients to be available, while also not targeting reconstruction of discontinuity curves directly.

\subsection{Code availability}\label{subsec:code}
The code implementing the algorithms described in this work is available at \url{https://github.com/mlevinov/algebraic-fourier-2d}.

\subsection{Acknowledgements}
The authors would like to thank Shai Dekel, Hrushikesh Mhaskar, J\"{u}rgen Prestin and David Levin for helpful comments and discussions. M.L. and D.B. were partially supported by Israel Science Foundation Grant 1793/20 and a collaborative grant from the Volkswagen Foundation.

\subsection{Organization}\label{subsec:organization}
The rest of the manuscript is organized as follows.
In~\cref{sec:The algebraic reconstruction method} we describe our method in detail and prove the main results.
In~\cref{sec:numerical-experiments} we present numerical experiments that demonstrate the accuracy of our method. In~\cref{sec:future-work} we discuss possible extensions and future work. Some more technical details are deferred to the appendix.

\section{The algebraic reconstruction method}\label{sec:The algebraic reconstruction method}
Recall from \cref{psi-oy-properties} that $\psiOy$ is in the general case only piecewise-smooth in $x$. Therefore, to obtain an accurate approximation of $\psi0y$, we apply a combination of the methods developed in ~\cite{batenkov2012algebraic,batenkov2015complete} for
recovering $\widetilde{\psiOy}$. By reconstructing $\psiOy$ we get an approximation for $\widehat{\Fx}(\oy)$ and under the assumption that
$\Fx(y)$ is a piecewise-smooth 1D function we apply the method described
in~\cite{batenkov2012algebraic,batenkov2015complete}
once again to recover the jump locations and jump magnituides of $\Fx$ with sufficient accuracy. Thus we can in effect recover
$F$ \emph{slice by slice}.

The rest of this section is organized as follows. In~\cref{subsec:reconstruction-of-psi} we describe the decomposition of $\psiOy$ and provide the full details for its recovery.
Then in~\cref{subsec:accuracy-of-psi} we present the analysis of the approximation of $\widetilde{\psiOy}$, proving~\cref{thm:main-result-for-psi}. Afterwards, we proceed to the second step which is tasked with the approximation of the unknown 2D piecewise-smooth
function $F$. Similarly to the first step, we will begin by describing the decomposition of each \textit{slice}
$\Fx$ (see~\cref{description-of-F-and-xi}), also presenting the needed definitions and assumptions for the
reconstruction of $\Fx$. The accuracy of the second step, i.e. the approximation error $\left|\widetilde{\Fx}-\Fx\right|$, is analyzed in ~\cref{subsec:accuracy-of-F-at-x}, proving~\cref{thm:general-theorem-for-fx}.

The full algorithm for our technique is finally presented in~\cref{subsec:algorithm}.

\subsection{The first stage}\label{subsec:reconstruction-of-psi}

We use the following notation: for $x_0\in[-\pi,\pi]$ we define
\begin{equation}
    \label{v-psi-def}
    V_{{l}}(x;x_0) = -\frac{\left(2\pi\right)^l}{\left( l+1 \right)!}B_{l+1}\left( \frac{x-x_0}{2\pi} \right),\quad \xi\leq x\leq\xi+2\pi,
\end{equation}
 where $V_{{l}}(x)$ is understood to be periodically extended to $\left[ -\pi, \pi \right]$ and $B_{l}(x)$ is the $l$-th Bernoulli
polynomial~\cite{olver2010nist}.
Since $\psiOy$ has a jump at $x_0=-\pi$, we denote for simplicity $\Vpsi(x):=V_{l}(x;-\pi)$.

We also use the notation $c_k(f) := \widehat{f}(k)$ to denote the Fourier coefficients of a 1D function $f$.

We begin with a decomposition of $\psiOy$ as defined in~\cref{def-of-psi-omega-y} into a sum of two functions:
\begin{equation} \label{eq:decomposition-psi-omega-y}
    \psi_{\omega_{y}}(x) = \gamma_{\omega_{y}}(x)+\phi_{\omega_{y}}(x),\quad x\in\mathbb{T},
\end{equation}
where $\gamma_{\oy} \in C_{\mathbb{T}}^{\left( \doy + 1 \right)}$ and $\ddxk[l]\gammaOy$ is periodic for $l=0,\ldots,\doy$ and $\phi_{\omega_{y}}(x)$ is a piecewise polynomial of degree
$d_{\oy}$ with a discontinuity at $x=-\pi$ which is uniquely determined by $\left\{ \Aoyl \right\}_{0 \leq l \leq d_{\oy}}$ (recall \cref{psi-oy-jump-mag-def})
such that it ``absorbs'' all the discontinuities of $\psi_{\omega_{y}}$ and its derivatives.
\noindent Eckhoff~\cite{eckhoff1993accurate,eckhoff1995accurate,eckhoff1998high} derives the following explicit
representation of $\phi_{\omega_{y}}(x)$:
\begin{equation}
    \label{phi-omega-y}
    \phi_{\omega_{y}}(x)=\sum_{l=0}^{d} \Aoyl \Vpsi(x).
\end{equation}

%
\begin{proposition}
    \label{prop:phi-oy-coefficients}
    Let $\phi_{\omega_{y}}(x)$ be given by~\cref{phi-omega-y}, then for $\ox\in\mathbb{Z}$:
    \begin{gather*}
        c_{\ox}(\phiOy) =
        \begin{cases}
            0 &\text{if } \ox=0 \\
            \frac{\left( -1 \right)^{\ox}}{2\pi} \sum_{l=0}^{d} \frac{\Aoyl}{\left( \imath \ox \right)
            ^{\left( l+1 \right)}} &\text{if } \ox \neq 0
        \end{cases}
    \end{gather*}
\end{proposition}
\begin{proof}
    See~\cref{proof-phi-oy-coefficients}.
\end{proof}

\noindent From~\cref{eq:decomposition-psi-omega-y} we get:
\begin{gather}
    \label{decomposition-corfficient-psi-oy}
    c_{k}(\psiOy) = c_{k}(\gammaOy) + c_{k}(\phiOy).
\end{gather}
Eckhoff observed that if $\gamma_{\omega_{y}}$ is sufficiently smooth, then the contribution of $c_{k}\left( \gamma_{\omega_{y}} \right)$ to
$\ckPsiOy$ becomes negligible, i.e. for large $\ox$ we have
\begin{gather}\label{c_ox-aprrox-c_phiox}
    c_{\ox}(\psiOy)\approx c_{\ox}(\phiOy)=\frac{(-1)^{\ox}}{2\pi}\sum_{l=0}^{d}\frac{\Aoyl}{(\imath \ox)^{l+1}}
    ,\quad |\ox|\gg 1.
\end{gather}

As will be explained below, the last approximate equality (cf.~\cref{decimated-system-psi} below) will be used in order to obtain approximate values for the jump magnitudes $\Aoyl$ from the Fourier coefficients of $\psiOy$. Let us denote those by $\widetilde{\Aoyl}$. Then we further define
\begin{align}
    \label{phi-oy-tilde-coefficient}
    c_{\ox}(\widetilde{\phiOy}) &\coloneqq \frac{(-1)^{\ox}}{2\pi} \sum_{l=0}^{d} \frac{\widetilde{\Aoyl}}{(\imath \ox)^{l+1}},\quad
    \abs{\ox} \leq M\\
    \label{gamma-oy-tilde-coefficient}
    c_{\ox}(\widetilde{\gammaOy}) &\coloneqq c_{\ox}(\psiOy) - c_{\ox}(\widetilde{\phiOy}),\quad \abs{\ox} \leq M\\
    \label{gamma-oy-tilde}
    \widetilde{\gammaOy}(x) &\coloneqq \sum_{|\ox| \leq M} c_{\ox}(\widetilde{\gammaOy}) e^{\imath x\ox}
\end{align}
and take the final approximation:
\begin{align}
    \begin{split}
    \label{final-approximation}
    \widetilde{\psiOy}(x) &\coloneqq \widetilde{\gammaOy}(x) + \widetilde{\phiOy}(x)\\
    &\;= \sum_{|\ox| \leq M} c_{\ox}(\widetilde{\gammaOy})
    e^{\imath x\ox} + \sum_{l=0}^{d} \widetilde{\Aoyl} \Vpsi(x).
    \end{split}
\end{align}

The recovery of the approximate jump magnitudes of $\ddxk[l]\psiOy$ for $0 \leq l\leq d$ at $x=-\pi$ is performed
using the so-called \emph{decimation} ~\cite{batenkov2015complete}.
In details, we take $\oyInZ,\;\; d\in\mathbb{N}$, $M \gg 1$ and choose the indices $\ox$ to be evenly distributed across the range
$\left\{ 0,1,\ldots,M \right\}$.
So by denoting
\begin{gather} \label{def-for-Nx}
    M_{1} \coloneqq \floor*{\frac{M}{\left( d + 1 \right)}}
\end{gather}
and by~\cref{c_ox-aprrox-c_phiox} we solve the following linear system in order to extract $\Aoyl$
\begin{gather}\label{decimated-system-psi}
    \begin{split}
        c_{\ox}(\psiOy) \approx \frac{(-1)^{\ox}}{2\pi} \sum_{l=0}^{d} \frac{\Aoyl}{(\imath\ox)^{l+1}}\left( =c_{\ox}(\phiOy) \right) \\
        \text{where }\ox = M_{1}, 2M_{1}, \ldots, (d + 1)M_{1}.
    \end{split}
\end{gather}

The following result bounds the error in~\cref{decimated-system-psi}.
\begin{lemma} \label{main-result-for-coefficients-of-gamma-omega-y}
    Let $\oyInZ,\; k\in\mathbb{N}$ and let there be given $(2M+1) \times (2N+1)$ Fourier coefficients of
    function $F:\torus[2]\to\mathbb{R}$ as in~\cref{eq:fourier-data}.
    Let $\gammaOy \in C_{\mathbb{T}}^{(d + 2)}$ as defined in~\cref{eq:decomposition-psi-omega-y}.
    \newline
    Then there exist constants $\Admxi$ and $B_{F} \geq 0$ (see~\cref{Admxi,definition-of-Fx-and-jump-location-and-jump-magnitudes-and-assumptions}) such that:
    \begin{gather*}
        \abs{c_{\ox}(\gammaOy)} \leq \left( 2d+5 \right) \Big(B_{F} + \Admxi A_{x} \abs{\oy}^{d} \Big)\abs{\ox}^{-d-2}.
    \end{gather*}
\end{lemma}
\begin{proof}
    Immediately follows from~\cref{bound-for-gamma-omega-y-derivatives} (formulated and proved in~\cref{subsec:proof-of-prop-15}) and from \cite[Vol. I, chapter 3, Theorem 13.6]{zygmund2002trigonometric}.
\end{proof}

\noindent Now, in order to recover $\left\{ \Aoyl \right\}_{l=0}^{\doy}$ from the system
in~\cref{decimated-system-psi} we define
\begin{equation}\label{alpha-omega-y-def}
     \alpha_{l}(\oy) \coloneqq \imath^{l} A_{\doy-l}^{\psi}(\oy)
\end{equation}
and
\begin{equation}
    \label{def-of-moy-mox}
    \begin{split}
        \mOxOy &\coloneqq (-1)^{\ox} \sum_{l=0}^{d} \alpha_{l}(\oy) \ox^{l} \\
        &\Leftrightarrow \\
        (-1)^{-\ox}\mOxOy &\coloneqq \sum_{l=0}^{d} \alpha_{l}(\oy) \ox^{l}
    \end{split}
\end{equation}

\noindent and by multiplying each side of~\eqref{decimated-system-psi} by $2\pi(\imath \ox)^{d + 1}$ we get a new definition:
\begin{gather}
    \label{new-def-moy-ox-and-delta}
    \begin{split}
        \widetilde{\mOxOy} &\coloneqq 2\pi(\imath \ox)^{d + 1} \cdot \widetilde{c_{\ox}(\psiOy)} = \mOxOy + \delOxOy,\\
        \ox &= M_{1}, 2M_{1},\ldots, (d + 1)M_{1},  \\
        &\text{where by \cref{main-result-for-coefficients-of-gamma-omega-y} we have} \\
        \big\lvert\delOxOy\big\rvert &\leq 2\pi \left( 2d+5 \right) \Big(B_{F} + \Admxi A_{x} \abs{\oy}^{d} \Big)\abs{\ox}^{-1}.
    \end{split}
\end{gather}

We are solving this system for $\ox = M_{1}, 2M_{1},\ldots ,(d + 1)M_{1}$ which leads us to a linear system of order $(d+1) \times (d+1)$
described here:
\begin{equation}
    \label{def-vnd}
    V_{N}^{d} \coloneqq
    \begin{bmatrix}
        1      & N      & N^{2}        & \cdots & N^{d}        \\
        1      & 2N     & (2N)^{2}     & \cdots & (2N)^{d}     \\
        \vdots & \vdots & \vdots       & \vdots & \vdots       \\
        1      & (d+1)N & ((d+1)N)^{2} & \cdots & ((d+1)N)^{d}
    \end{bmatrix}
\end{equation}
\indent Note that $V_{N}^{d}$ is the Vandermonde matrix on $\left\{ N, 2N, \ldots, (d+1)N \right\}$ and thus it is
nondegenerate for all $N\gg 1$.
\newline\newline
Presenting the linear system in~\cref{def-of-moy-mox}:
\begin{gather}
    \label{linear-system-psi}
    \begin{bmatrix}
    (-1)
        ^{M_{1}} m(M_{1}, \oy)\\
        (-1)^{2M_{1}} m(2M_{1}, \oy)               \\
        \vdots\\
        (-1)^{(d+1)M_{1}} m((d+1)M_{1}, \oy) \\
    \end{bmatrix} =
    V_{M_{1}}^{d} \cdot
    \begin{bmatrix}
        \alpha_{0}(\oy)    \\
        \alpha_{1}(\oy)    \\
        \alpha_{2}(\oy)    \\
        \vdots             \\
        \alpha_{d}(\oy) \\
    \end{bmatrix},
\end{gather}
the solution contains the jump magnitudes of $\frac{d^l}{dx^l}\psiOy$ where $l=0,\ldots,d$, but unfortunately $
\mOxOy$ are unknown so we use $\widetilde{m_{\ox}}(\oy)$ and solve the perturbed linear system:
\begin{gather}
    \label{pertrubed-system-psi}
    \begin{bmatrix}
    (-1)^{M_{1}}\widetilde{m}_{M_{1}}(\oy)\\
        (-1)^{2M_{1}}\widetilde{m}_{2M_{1}}(\oy)                \\
        \vdots\\
        (-1)^{(d+1)M_{1}}\widetilde{m}_{(d+1)M_{1}}(\oy) \\
    \end{bmatrix} =
    V_{M_{1}}^{d} \cdot
    \begin{bmatrix}
        \widetilde{\alpha_{0}}(\oy)    \\
        \widetilde{\alpha_{1}}(\oy)    \\
        \widetilde{\alpha_{2}}(\oy)    \\
        \vdots                         \\
        \widetilde{\alpha_{d}}(\oy) \\
    \end{bmatrix}
\end{gather}
\noindent to get the approximation for the jump magnitudes, $\left\{ \widetilde{\Aoyl} \right\}_{l=0,\ldots, d}$, at $x=-\pi$.
\newline
\\
The next and final step is to recover $\psiOy$ by~\cref{final-approximation}:
\begin{gather*}
    \label{recovery-of-psi}
    \widetilde{\psiOy}(x)
    = \sum_{\abs{\ox}\leq M} c_{\ox}(\widetilde{\gammaOy}) e^{\imath\ox x} + \sum_{l=0}^{d} \widetilde{\Aoyl}V_{\psi,l}(x).
\end{gather*}

\subsection{Proof of \cref{thm:main-result-for-psi}}\label{subsec:accuracy-of-psi}
%
This section is dedicated to proving \cref{thm:main-result-for-psi}.

The jump location of $\psiOy$ is known to be at $x=-\pi$ (or at $\pi$), and we define the jump
magnitude of $\ddxk[l](\psiOy)$ as described in~\cref{psi-oy-jump-mag-def}.

Since we assume that for each $\abs{\oy}\leq \Moy$ the jump location is known to be at $\xi=-\pi$, the first step
would be to analyze the accuracy of the jump magnitudes approximation, $\widetilde{\Aoyl}$. Towards this goal, we have the following result.
\begin{lemma}
    \label{moy-tilde-lemma}
    Let $\oyInZ$, $N\in\mathbb{N}$, $0<j\in\mathbb{N}$ and $\mOxOy$ as defined in~\cref{new-def-moy-ox-and-delta}.
    Then there exist constants $\Admxi$ and $B_{F} \geq 0$
    (see~\cref{Admxi,definition-of-Fx-and-jump-location-and-jump-magnitudes-and-assumptions}) such that:
    \begin{gather}
        \abs{\widetilde{m}_{jN}(\oy) - m_{jN}(\oy)} \leq 2\pi \abs{jN}^{-1}
        \left( 2d+5 \right) \Big(B_{F} + \Admxi A_{x} \abs{\oy}^{d} \Big).
    \end{gather}
\end{lemma}
\begin{proof}
    Immediately follows from~\cref{new-def-moy-ox-and-delta}.
\end{proof}

\begin{proof}[Proof of~\cref{thm:main-result-for-psi}, first estimate]
    Denoting
    $$
    \zeta_{j}\coloneqq (-1)^{j N}(\widetilde{m}_{j N}(\oy) -  m_{j N}(\oy)),
    $$
   \cref{pertrubed-system-psi} implies that  for $N = M_{1}$ the error vector satisfies:
    \begin{gather}
        \label{inverse-perturbed-system-psi}
        \begin{bmatrix}
            \widetilde{\alpha_{0}}(\oy) - \alpha_{0}(\oy)       \\
            \widetilde{\alpha_{1}}(\oy) - \alpha_{1}(\oy)       \\
            \widetilde{\alpha_{2}}(\oy) - \alpha_{2}(\oy)       \\
            \vdots                                              \\
            \widetilde{\alpha_{d}}(\oy) - \alpha_{d}(\oy) \\
        \end{bmatrix} =
        \left( V_{M_{1}}^{d} \right)^{-1} \cdot
        \begin{bmatrix}
            \zeta_{1}      \\
            \zeta_{2}      \\
            \vdots         \\
            \zeta_{d+1} \\
        \end{bmatrix}.
    \end{gather}

    Using~\cref{moy-tilde-lemma} we have that:
    \begin{equation}
        \label{eq:bound-on-zeta-j}
        \begin{split}
            \abs{\zeta_{j}} &= \abs{\widetilde{m}_{j M_{1} }(\oy) -  m_{j M_{1} }(\oy)}\\
            &\leq 2\pi \abs{j M_{1}}^{-1} \leq 2\pi \abs{j M_{1}}^{-1} \left( 2d+5 \right) \Big(B_{F} + \Admxi A_{x} \abs{\oy}^{d} \Big)\\
            &\overset{j\geq 1}{\leq} 2\pi (M_{1})^{-1} \left( 2d+5 \right) \Big(B_{F} + \Admxi A_{x} \abs{\oy}^{d} \Big).
        \end{split}
    \end{equation}

    Going back to~\eqref{def-vnd}, we rewrite $V_{M_{1}}^{d}$ in order to estimate
    $\left( V_{M_{1}}^{d} \right)^{-1}$ as follows:
    \begin{equation*}
        \begin{split}
            V_{M_{1}}^{d} &=
            \begin{bmatrix}
                1      & M_{1}         & \cdots & (M_{1})^{d}           \\
                1      & 2M_{1}        & \cdots & (2M_{1})^{d}        \\
                \vdots & \vdots          & \vdots & \vdots                   \\
                1      & (d+1)M_{1} & \cdots & ((+1)M_{1})^{d}
            \end{bmatrix} \\
            \\
            &=
            \begin{bmatrix}
                1      & 1        & \cdots & 1               \\
                1      & 2        & \cdots & 2^{d}        \\
                \vdots & \vdots   & \vdots & \vdots          \\
                1      & (d+1) & \cdots & (d+1)^{d}
            \end{bmatrix} \cdot diag\left( 1, M_{1}, \ldots, (M_{1})^{d} \right) \\
            \\
            &= V_{1}^{d} \cdot diag\left( 1, M_{1}, \ldots, (M_{1})^{d} \right)
        \end{split}
    \end{equation*}

    \begin{gather}
        \begin{split}
            \Rightarrow \left( V_{M_{1}}^{d} \right)^{-1} =
            \begin{bmatrix}
                1 \\
                & (M_{1})^{-1} \\
                & & \ddots &                 \\
                & &        & (M_{1})^{-d}
            \end{bmatrix}
            \left( V_{1}^{d} \right)^{-1}
        \end{split}.
    \end{gather}
    Inserting the above into~\cref{inverse-perturbed-system-psi} we obtain
    \begin{gather}
        \label{eq:2.39}
        \begin{bmatrix}
            \widetilde{\alpha_{0}}(\oy) - \alpha_{0}(\oy)       \\
            \widetilde{\alpha_{1}}(\oy) - \alpha_{1}(\oy)       \\
            \widetilde{\alpha_{2}}(\oy) - \alpha_{2}(\oy)       \\
            \vdots                                              \\
            \widetilde{\alpha_{d}}(\oy) - \alpha_{d}(\oy) \\
        \end{bmatrix} =
        \begin{bmatrix}
            1 \\
            & (M_{1})^{-1} \\
            & & \ddots &                 \\
            & &        & (M_{1})^{-d}
        \end{bmatrix}
        \left( V_{1}^{d} \right)^{-1}
        \begin{bmatrix}
            \zeta_{1}      \\
            \zeta_{2}      \\
            \vdots         \\
            \zeta_{d+1} \\
        \end{bmatrix}.
    \end{gather}
    Presently, we consider the $l^{th}$ row of $\left( V_{1}^{d} \right)^{-1} = \left( a_{l,1},\ldots, a_{l,d+1} \right)$, and obtain
    \begin{equation*}
        \widetilde{\alpha_{l}}(\oy) - \alpha_{l}(\oy) = (M_{1})^{-l}\sum_{k=1}^{d+1} a_{l,k}\zeta_{k}.
    \end{equation*}
    Denote $C_{l}$ as the upper bound on the absolute sum of the entries of the $l^{th}$ row of $\left( V_{1}^{d} \right)$:
    \begin{equation}
        \label{eq:bound-on-inverse-of-V-1-doy}
        \sum_{k=1}^{d+1} \abs{a_{l,k}} \leq C_{l} \leq C^{*} \coloneqq \max\limits_{0\leq l\leq d+1} \left\{ C_l \right\}.
    \end{equation}
    Using this bound with~\cref{eq:bound-on-zeta-j} we proceed as follows:
    \begin{align*}
        &\abs{\widetilde{\alpha_{l}}(\oy) - \alpha_{l}(\oy)} = \abs{M_{1}^{-l}\sum_{k=1}^{d+1} a_{l,k}\zeta_{k}} \leq M_{1}^{-l}
        \sum_{k=1}^{d+1} \abs{a_{l,k}}\abs{\zeta_{k}} \\
        &\leq M_{1}^{-l} \sum_{k=1}^{d+1} \frac{2\pi}{M_{1}} \left( 2d+5 \right) \Big(B_{F} + \Admxi A_{x} \abs{\oy}^{d} \Big) \abs{a_{l,k}} \\
        &= \frac{2\pi}{(M_{1})^{l+1}} \left( 2d+5 \right) \Big(B_{F} + \Admxi A_{x} \abs{\oy}^{d} \Big) \sum_{k=1}^{d+1} \abs{a_{l,k}} \\
        &\leq 2\pi C_{l} \left( 2d+5 \right) \Big(B_{F} + \Admxi A_{x} \abs{\oy}^{d} \Big) M_{1}^{-l-1} \\
        &\leq \frac{2\pi C_{l} \left( 2d+5 \right) \Big(B_{F} + \Admxi A_{x} \abs{\oy}^{d} \Big)}{{\floor{\frac{
            M}{d+2}}^{l+1}}}.
    \end{align*}
    Assuming $\frac{M}{d+2} - 1 > \frac{M}{2(d+2)} \Leftrightarrow M > 2(d+2)$ we further have:
    \begin{align*}
        \abs{\widetilde{\alpha_{l}}(\oy) - \alpha_{l}(\oy)} &\leq 2\pi C_{l} (2d+4)^{l+1} \left( 2d+5 \right) \Big(B_{F} + \Admxi A_{x} \abs{\oy}^{d} \Big) M^{-l-1} \\
        &\leq 2\pi C_{l} (2d+5)^{l+2}\Big(B_{F} + \Admxi A_{x} \abs{\oy}^{d} \Big) M^{-l-1}.
    \end{align*}
    Now let us define the constants
    \begin{align}
        \label{h-func}
            \h &\coloneqq \h(d, B_{F}, \mxi) \text{ s.t. } B_{F} + \Admxi A_{x} \abs{\oy}^{d} \leq \h A_{x} \abs{\oy}^{d},\\
        \label{hl-func}
            \hl &\coloneqq \hl (d,B_{F}, \mxi) = 2\pi C_{l} (2d+5)^{l+2} \h.
    \end{align}
%
    This implies
    \begin{gather*}
        \abs{\widetilde{\alpha_{l}}(\oy) - \alpha_{l}(\oy)} \leq \hl A_{x} \abs{\oy}^{d} M^{-l-1}
    \end{gather*}
    and
    \begin{align}
        \abs{\widetilde{\alpha_{l}(\oy)} - \alpha_{l}(\oy)} &= \abs{{\imath}^{d+1-l}
            \widetilde{A_{d-l}^{\psi}(\oy)}-{\imath}^{d+1-l}A_{d-l}^{\psi}(\oy)}
        \notag\\
        &= \abs{\widetilde{A_{d-l}^{\psi}(\oy)}-A_{d-l}^{\psi}(\oy)}. \notag
    \end{align}
    Finally we have:
    \begin{align}
        \begin{split}
            \abs{\widetilde{A_{d-l}^{\psi}(\oy)}-A_{d-l}^{\psi}(\oy)} &\leq \hl A_{x} \abs{\oy}^{d} M^{-l-1}, \\
            \text{for }l=0,\ldots,&d
        \end{split}
    \end{align}
    and by reindexing, we have
    \begin{gather}
        \begin{split}
            \abs{\widetilde{\Aoyl}-\Aoyl} &\leq \hl[d-l] A_{x} \abs{\oy}^{d} M^{l-d-1}, \\
            \text{for }l=0,\ldots,&d.
        \end{split}
    \end{gather}
    completing the proof.
\end{proof}

\begin{proof}[Proof of~\cref{thm:main-result-for-psi}, second estimate]
Recall that by ~\cref{eq:decomposition-psi-omega-y} we have $\gammaOy=\psiOy-\phiOy$ with $\gammaOy$ being $d$-times everywhere continuously differentiable  function. We denote
\begin{equation}
    \label{eq:psi-oy-mox}
    \psi_{\oy,M}(x) \coloneqq \sum_{\abs{\ox}\leq M} \left( c_{\ox}(\psiOy) - c_{\ox}(\phiOy) \right)e^{\imath x\ox} + \phiOy(x).
\end{equation}
We will estimate the pointwise approximation error of $\psiOy$ by bounding each term in the decomposition:
\begin{equation}
    \label{eq:err-final-approx-psi}
    \abs{\widetilde{\psiOy}(x)-\psiOy(x)} \leq \abs{\widetilde{\psiOy}(x)-\psi_{\oy,M}(x)} + \abs{\psi_{\oy, M}(x)-\psiOy(x)}.
\end{equation}

\noindent We commence by estimating the first term in the right-hand side of~\cref{eq:err-final-approx-psi}. Starting with a definition:
\begin{gather*}
    \Theta_{d}^{*}(\oy) \coloneqq \widetilde{\phiOy} - \phiOy,
\end{gather*}
we have
    \begin{equation*}
        \begin{split}
            \abs{\widetilde{\psiOy}(x)-\psi_{\oy, M}(x)} &= \abs{\sum_{\abs{\ox}\leq M}\left( c_{\ox}(\phiOy)-c_{\ox}(\widetilde{\phiOy})  \right)
                e^{\imath x\ox} + \left( \widetilde{\phiOy(x)} - \phiOy(x) \right)} \\
            &= \abs{\Theta_{d}^{*}(\oy)(x) - \sum_{\abs{\ox}\leq M}c_{\ox}(\Theta_{d}^{*}(\oy))e^{\imath x\ox}} \\
            &= \abs{\sum_{\abs{\ox}> M}c_{\ox}(\Theta_{d}^{*}(\oy))e^{\imath x\ox} }
        \end{split}
    \end{equation*}

    Using~\cref{phi-omega-y}
    we also have
    \begin{gather*}
        \Theta_{d}^{*}(\oy)(x) = \sum_{l=0}^{d} \left( \widetilde{A_{l}^{\psi}}(\oy) - \Aoyl \right)\Vpsi(x),
    \end{gather*}
    \indent thus:
    \begin{align*}
        c_{\ox}(\Theta_{d}^{*}(\oy)) &= c_{\ox} \left( \sum_{l=0}^{d} \left( \widetilde{A_{l}^{\psi}}(\oy) -
        \Aoyl \right) \Vpsi(x) \right) \\
        &= \left(\sum_{l=0}^{d} \left( \widetilde{A_{l}^{\psi}}(\oy) -\Aoyl \right)c_{\oy}(\Vpsi) \right).
    \end{align*}
    Using $\Vpsi \in C_{\mathbb{T}}^{l}$ and a well-known estimate (see~\cite[Section.~3,p.~27]{gottlieb1977numerical})
    we claim that there exist positive constants $\{T_{l}\}_{l=0}^{d}$ such that:
    \begin{gather}\label{T-star}
        \abs{\sum_{\abs{\ox}> M}c_{\ox}(\Vpsi)e^{\imath x \ox}} \leq T_{l} M^{-l}
        \leq T^{*} M^{-l}
    \end{gather}
    where $T^{*} \coloneqq \max\limits_{0\leq l\leq d} \left\{ T_l \right\}$
    and therefore together with~\cref{hl-func} we get:
    \begin{align*}
        \abs{\sum_{\abs{\ox}> M} c_{\ox} (\Theta_{d}^{*}(\oy)) e^{\imath x \ox}} &\leq
        2\pi T_{l} C_{d-l} \left( 2d+5 \right) \Big(B_{F} + \Admxi A_{x} \abs{\oy}^{d} \Big) (M_{1})^{l-d-1} M^{-l} \\
        &\leq \frac{2\pi T_{l} C_{d-l} \left( 2d+5 \right) \Big(B_{F} + \Admxi A_{x} \abs{\oy}^{d} \Big)}{(M_{1})^{d+1-l} M^{l}} \\
        &\leq \frac{2\pi T_{l} C_{d-l} \left( 2d+5 \right) \Big(B_{F} + \Admxi A_{x} \abs{\oy}^{d} \Big)}{(\frac{M}{d+2}-1)^{d+1-l} M^{l}} \\
        &\leq \frac{2\pi T_{l} C_{d-l} \left( 2d+5 \right) \Big(B_{F} + \Admxi A_{x} \abs{\oy}^{d} \Big)}{(\frac{M}{2(d+2)})^{d+1-l} M^{l}} \\
        = \Bigg( 2\pi T_{l} C_{d-l} &\left( 2d+5 \right) \Big(B_{F} + \Admxi A_{x} \abs{\oy}^{d} \Big) (2d+4)^{d+1-l} \Bigg)M^{-d-1} \\
        \leq \Bigg( 2\pi T_{l} C_{d-l} &\Big(B_{F} + \Admxi A_{x} \abs{\oy}^{d} \Big) (2d+5)^{d+2-l} \Bigg) M^{-d-1} \\
        &\leq T_{l} \hl[d-l] A_{x} \abs{\oy}^{d} M^{-d-1} \leq T^{*} \hl[d-l] A_{x} \abs{\oy}^{d} M^{-d-1}.
    \end{align*}

    Denoting
    \begin{gather}\label{ht-star}
        \hl[T]^{*} \coloneqq T^{*} \cdot \max\limits_{0\leq l\leq d}\left\{ \hl \right\}
    \end{gather}
    we get
    \begin{equation} \label{eq:first-part-of-final-error-psi}
        \abs{\widetilde{\psiOy}(x)-\psi_{\oy, M}(x)} \leq \hl[T]^{*} A_{x} \abs{\oy}^{d} M^{-d-1}.
    \end{equation}

    \noindent Now we estimate the second term in the right-hand side of~\cref{eq:err-final-approx-psi}. Combining~\cref{eq:decomposition-psi-omega-y} and~\cref{eq:psi-oy-mox} we have:
    \begin{equation}\label{eq:second-part-of-final-error-psi}
            \abs{\psi_{\oy, M}(x) - \psiOy(x)} 
            = \abs{\sum_{\abs{\ox} > M} c_{\ox}(\gammaOy) e^{\imath x\ox}}.
    \end{equation}
    Using the same analysis as in~\cite[p.301]{batenkov2012algebraic} together with~\cref{main-result-for-coefficients-of-gamma-omega-y,eq:decomposition-psi-omega-y} implies that there exists a constant $R$ s.t.:
    \begin{gather}\label{R-gamma-oy}
        \begin{split}
            \abs{\psi_{\oy, M}(x) - \psiOy(x)} &\leq R \left( 2d+5 \right) \Big(B_{F} + \Admxi A_{x} \abs{\oy}^{d} \Big) M^{-d-1} \\
            &\leq R \left( 2d+5 \right) \h A_{x} M^{-d-1}.
        \end{split}
    \end{gather}
    Combining~\cref{eq:first-part-of-final-error-psi,eq:second-part-of-final-error-psi} we get:
    \begin{align*}
        \abs{\widetilde{\psiOy}(x) - \psi_{\oy, M}(x)} &\leq \hl[T]^{*} A_{x} \abs{\oy}^{d} M^{-d-1} = \\
        &=\left( \hl[T]^{*} + R \left( 2d+5 \right) \h \right) A_{x} \abs{\oy}^{d} M^{-d-1} + R \left( 2d+5 \right) \h A_{x} M^{-d-1}.
    \end{align*}
    Denoting:
    \begin{gather}\label{r-gamma-oy-star-func}
        R_{d} \coloneqq (2d+5) R
    \end{gather}
    we can conclude that for $\oyInZ$ and $x\in\torus[1]$ we have:
    \begin{gather} \label{psi-tilde-minus-psi}
        \abs{\widetilde{\psiOy}(x) - \psiOy(x)} \leq \left( \hl[T]^{*} + R_{d}\h \right) A_{x} \abs{\oy}^{d} M^{-d-1}.
    \end{gather}

    This completes the proof of \cref{thm:main-result-for-psi}.
\end{proof}

\subsection{The second stage}\label{subsec:reconstruction-of-f_at_x}

So far, we have described the reconstruction process of $\left\{ \widetilde{\psiOy}(x) \right\}_{\abs{\oy}\leq N}$ for
a given $x\in\torus[1]$ by applying \cref{alg:full-order-algorithm} to the given data
$\left\{ \widehat{F}(\ox,\oy) \right\}_{\abs{\ox}\leq M}$ for each $\abs{\oy} \leq N$.
In the second stage, we will use these approximated values of $\widehat{\Fx}(\oy)$ (recall that $\psiOy(x) = \widehat{\Fx}(\oy)$)
to recover $\Fx(y)$ for each $y\in\torus[1]$.
This way, we eventually recover $F_x(y)$ for each $x,y\in\torus[1]$, and thus recover $F$ itself.
(In practice, this can be done for a finite number of $x$ and $y$).
Now we present the decomposition for a \textit{slice} of $F$ at $x\in\torus[1]$:
\begin{proposition}
    \label{decomposition-of-Fx}
    Let $x\in\torus[1]$, then
    \begin{gather}
        \Fx(y) =\gammax(y) + \Phidx(y),\quad \forall y \in \torus[1],
    \end{gather}
where $\gammax \in \C$, and $\Phidx$ is a piecewise polynomial of degree $d$ “absorbing” all the discontinuities of $\Fx$ and its $d + 1$ derivatives.
\end{proposition}
In particular, $\Phidx$ has a single discontinuity at $y_{x}=\xix$ and therefore it is uniquely determined by $\left\{ A_{l}(x) \right\}_{0\leq l \leq d}$ and $y_{x}$ as follows:
\begin{align} \label{eq:phidx}
    \Phidx(y) &= \sum_{l=0}^{d}{A_{l}(x)V_{l}\left( y;\yx \right)}.
\end{align}
Here, again, $V_l(y; \yx)$ is the periodic Bernoulli polynomial, as in~\cref{v-psi-def}.
Furthermore, we have that $\gammax\in\C$ which in turn by~\cite[Vol.~I, Chapter 3, Theorem 13.6]{zygmund2002trigonometric}
gives us, for $x\in\torus[1]$, a constant $R_{\gammax}$ s.t.:
\begin{equation} \label{R-gammax}
    \abs{c_{\oy}(\gammax)} \leq R_{\gammax}\abs{\oy}^{-d-2}.
\end{equation}
We assume a uniform bound over the Fourier coefficients of $\gammax$ as follows:
\begin{equation} \label{eq:absolut-bound-for-gamma}
    |c_{\oy}(\gammax)| \leq R_{\gammax}\abs{\oy}^{-d-2}\leq R_{\Gamma_{\mathbb{T}}}\abs{\oy}^{-d-2},\quad\forall x\in\torus[1].
\end{equation}
\\
\noindent Now we denote:
\begin{gather}
    \label{kappa-x0}
    \kappax \coloneqq e^{-\imath \xix}
\end{gather}

\noindent and get:
\begin{proposition}
    \label{prop:phidx-order-d-coefficients}
    Let $x\in\torus[1]\text{ and }\Phidx$ as given by~\eqref{eq:phidx}, then for $\oyInZ$:
    \begin{equation} \label{eq:phidx-coefficients}
        c_{\oy}(\Phidx) =
        \begin{cases}
            0 &\textrm{if } \oy=0, \\
            \frac{\left( \kappax \right)^{\ox}}{2\pi} \sum_{l=0}^{d} \frac{A_{l}(x)}{\left( \imath \oy \right)
            ^{\left( l+1 \right)}} &\textrm{ otherwise.}
        \end{cases}
    \end{equation}
\end{proposition}

\begin{proof}
    Repeat the proof for~\cref{prop:phi-oy-coefficients}.
\end{proof}

\noindent Using~\cref{def-of-psi-omega-y} for each $x\in\torus[1]\text{ and }\oyInZ$ we denote:
\begin{gather}
    \label{c-omega_y-equals-psi-omega-y}
    c_{\oy}(\Fx) \coloneqq \psiOy(x)
\end{gather}

\noindent and from~\cref{psi-tilde-minus-psi} we get:
\begin{gather} \label{psi-omega-y-tilde-with-delta}
    \begin{split}
        \widetilde{\psiOy}(x)&\coloneqq \psiOy(x) + \delOyx,\quad \textrm{where} \\
        \abs{\delOyx} &\leq \left(\hl[T]^{*} + R_{d} \h \right) A_{x} \abs{\oy}^{d} M^{-d-1}.
    \end{split}
\end{gather}

\noindent In order to recover the parameters $\yx=\xix \text{ and } \left\{ A_{l}(x) \right\}_{l=0}^{d}$ of $\Fx$
and to approximate $\Fx(y)\text{ at } y\in\torus[1]$ we would ideally need to have access to $\left\{ c_{\oy}(\Fx) \right\}$ for
$\abs{\oy}\leq N$
which are not available.
Instead we use the approximation of $c_{\oy}(\Fx)$ which is given in~\cref{psi-omega-y-tilde-with-delta} by $\widetilde{\psiOy(x)}$.
To apply decimation once again, we denote:
\begin{gather*}
    N_{1} \coloneqq \floor*{\frac{N}{d+2}} \notag.
\end{gather*}

\noindent Using~\cref{decomposition-of-Fx} we further denote:
\begin{gather}
    \label{eq:2.56}
    \cOyFx \coloneqq c_{\oy}(\gammax) + c_{\oy}(\Phidx)
\end{gather}
and by~\cref{c-omega_y-equals-psi-omega-y,psi-omega-y-tilde-with-delta} we have:
\begin{align*}
    \cOyFx &= \psiOy(x) = c_{\oy}(\gammax) + c_{\oy}(\Phidx)
\end{align*}
which gives us:
\begin{align*}
    \cOyFx &= \widetilde{\psiOy}(x) - \delOyx = c_{\oy}(\gammax) + c_{\oy}(\Phidx),
\end{align*}
therefore
\begin{gather} \label{psi-oy-at-x0-tilde}
    \widetilde{\psiOy}(x) = c_{\oy}(\gammax) + c_{\oy}(\Phidx) + \delOyx
\end{gather}
where $\delOyx$ is given by~\cref{psi-omega-y-tilde-with-delta}

Now for every $x\in\torus[1]$ we denote:
\begin{align}
    \alpha_{l}(x) &= {\imath}^{l} A_{d-l}(x) \label{alpha-l-x}\\
    \moyx &\coloneqq {\kappax}^{\oy}\sum_{l=0}^{d} \alpha_{l}(x){\oy}^{l} \label{moy-at-x}
\end{align}
Multiplying both sides of~\cref{psi-oy-at-x0-tilde} by $2\pi \left( \imath \oy \right)^{d+1}$ and using \cref{alpha-l-x,moy-at-x} together with~\cref{prop:phidx-order-d-coefficients} we get:
\begin{align*}
    2\pi(\imath\oy)^{d+1} \widetilde{\psiOy}(x) &= 2\pi(\imath\oy)^{d+1}c_{\oy}(\Phidx) + 2\pi(\imath\oy)^{d+1}
    \left( c_{\oy}(\gammax) + \delOyx \right) \\
    & = \kappax^{\oy}\sum_{l=0}^{d} (\imath \oy)^{d-l}A_{l}(x_0) + 2\pi(\imath\oy)^{d+1} \left( c_{\oy}(\gammax)
    + \delOyx \right) \\
    &= \moyx + 2\pi(\imath\oy)^{d+1} \left( c_{\oy}(\gammax) + \delOyx \right).
\end{align*}

Next we define:
\begin{gather} \label{moy-tiled-at-x}
    \widetilde{m}_{\oy}(x) \coloneqq \moyx + \epsOyx
\end{gather}
where $\epsOyx = 2\pi(\imath\oy)^{d+1} \left( c_{\oy}(\gammax) + \delOyx \right)$ and:
\begin{equation} \label{eq:eps-at-x-of-oy-bound}
    \begin{split}
        \abs{\epsOyx} &= 2\pi\abs{\oy}^{d+1}\abs{ \left( c_{\oy}(\gammax) + \delOyx \right)}\\
        &\leq 2\pi\abs{\oy}^{d+1} \left( \abs{c_{\oy}(\gammax)} + \abs{\delOyx} \right)\\
        &\text{using~\cref{eq:absolut-bound-for-gamma,psi-omega-y-tilde-with-delta} we continue}\\
        &\leq 2\pi \left( R_{\Gamma_{\mathbb{T}}} \abs{\oy}^{-1} + \left(\hl[T]^{*} + R_{d} \h \right) A_{x} \abs{\oy}^{2d+1} M^{-d-1} \right).
    \end{split}
\end{equation}

\noindent The decimated system in~\cref{moy-at-x} is solved in two steps.
First, a polynomial equation $q_{N}^{d}(u)=0$ is constructed from the perturbed
coefficients $\left\{ \widetilde{\psiOy}(x) \right\}_{\oy = N_{1}, 2N_{1},\ldots, (d+2)N_{1}}$ and denoted by:
\begin{gather}
    \label{qnd-for-fx}
    q_{N_{1}}^{d}(u) = \sum_{j=0}^{d+1} (-1)^{j} \binom{d+1}{j} \widetilde{m}_{(j+1)N_{1}}(x) u^{d+1-j}.
\end{gather}

\begin{definition}
    \label{pnd-for-fx}
    Let $x\in\torus[1]$, then
    \begin{equation}
        \label{eq:2.61}
        p_{N_1}^{d}(u) \coloneqq \sum_{j=0}^{d+1} (-1)^{j} \binom{d+1}{j} m_{(j+1)N_1}(x) u^{d+1-j}.
    \end{equation}
\end{definition}

Therefore, $q_{N_{1}}^{d}$ is a perturbation of $p_{N_{1}}^{d}(u)$ which is constructed from the unperturbed and unknown values of
$\left\{ \cOyFx \right\}$.

\begin{proposition}
    \label{root-of-pnd}
    The point $u = \kappax^{N_{1}} = e^{-\imath\xix N_{1}}$ is a root of $p_{N_1}^{d}.$
\end{proposition}

\begin{proof}
    Denoting $z = \kappax^{N_{1}}$ we get:
    \begin{align*}
        p_{N_1}^{d}(z) &= \sum_{j=0}^{d+1} (-1)^{j} \binom{d+1}{j} m_{(j+1)N_{1}}(x) z^{d+1-j} \\
        &= \sum_{j=0}^{d+1} (-1)^{j} \binom{d+1}{j} \left( z^{j+1}\sum_{l=0}^{d} \alpha_{l}(x){(j+1)N_{1}}^{l} \right) z^{d+1-j} \\
        &= \sum_{j=0}^{d+1} (-1)^{j} \binom{d+1}{j} z^{d+2} \left( \sum_{l=0}^{d} \alpha_{l}(x) (j+1)^{l} N_{1}^{l} \right) \\
        &= z^{d+2}\sum_{l=0}^{d} \alpha_{l}(x)N_{1}^{l} \left\{ \sum_{j=0}^{d+1} (-1)^{j}\binom{d+1}{j} (j+1)^{l}\right\}
    \end{align*}\\
    Now we notice that if we apply the forward difference operator $d+1$ times over the polynomial $f(x) = x^{l}$ we get:
    \begin{gather*}
        \delta^{d+1}f(x) = \sum_{j=0}^{d+1} (-1)^{j} \binom{d+1}{j} \left( j+1 \right)^{l}
    \end{gather*}
    but because $l < d+1 $ using~\cite[p.~6]{boole1872treatise} we conclude that the above equals 0, implying that $p_{N_1}^{d}(u)=0$.
\end{proof}

Now we understand that one of the roots of $p_{N_{1}}^{d}(u)$ is $u=\kappa_{x}^{N_{1}}$ and thus, by solving the
perturbed equation $q_{N_{1}}^{d}(u)=0$ we will
recover the approximation of $\kappax^{N_{1}}$, which is $\widetilde{\kappax}^{N_{1}} = e^{-i\widetilde{\xix}N_{1}}$ and by
extracting the $N_{1}^{th}$ root and subsequently taking logarithm we obtain the approximation of the jump location
of $\Fx$, which is $\widetilde{\yx} = \widetilde{\xix}$.
The operation of taking root generally results in a multi-valued solution, therefore to ensure correct
reconstruction, we additionally assume that the jump $\xix$ must be known with \textit{a priori} accuracy of
order $o\left((N_{1})^{-1}\right)$, which is valid because it is always possible to apply the half order algorithm such as
in~\cite[p.~4]{batenkov2015complete} in order to achieve the approximation of the jump location with the assumed accuracy.

Now the second step is to recover the jump magnitudes $\left\{ A_{l}(x) \right\}_{l=0}^d$ by solving the linear system of equations in ~\cref{vnd-system-Fx}
using the approximated jump location $\widetilde{\kappax}$.

\begin{definition}
    \label{vnd-Fx}
    Let $V_{N}^{d}$ denote the $(d+1)\times (d+1)$ matrix
    \begin{equation}
        \label{eq:2.62}
        V_{N}^{d} \coloneqq
        \begin{bmatrix}
            1      & N      & N^{2}        & \cdots & N^{d}        \\
            1      & 2N     & (2N)^{2}     & \cdots & (2N)^{d}     \\
            \vdots & \vdots & \vdots       & \vdots & \vdots       \\
            1      & (d+1)N & ((d+1)N)^{2} & \cdots & ((d+1)N)^{d}.
        \end{bmatrix}
    \end{equation}
\end{definition}
\indent Note that $V_{N}^{d}$ is the Vandermonde matrix on $\left\{ N, 2N, \ldots, (d+1)N \right\}$ and thus it is nondegenerate for all $0\neq N\in\mathbb{N}$.

\begin{proposition}
    \label{vnd-system-Fx}
    The vector of exact magnitudes $\left( \alpha_{0}(x),\ldots, \alpha_{d}(x) \right)^{T}$ satisfies:
    \begin{equation}
        \label{exact-vnd-system-Fx}
        \begin{bmatrix}
            m_{N_{1}}(x)\cdot(\kappax)^{-N_{1}}   \\
            m_{2N_{1}}(x)\cdot(\kappax)^{-2N_{1}} \\
            \vdots                          \\
            m_{(d+1)N_{1}}(x)\cdot(\kappax)^{-(d+1)N_{1}}
        \end{bmatrix}
        = V_{N_{1}}^{d} \cdot
        \begin{bmatrix}
            \alpha_{0}(x) \\
            \alpha_{1}(x) \\
            \vdots          \\
            \alpha_{d}(x)
        \end{bmatrix}.
    \end{equation}
\end{proposition}

\begin{proof}
    Immediately follows from~\cref{moy-at-x}
\end{proof}

The solution contains the jump madnitudes of $\frac{d^l}{dx^l}\Fx$ where $l=0,\ldots,d$, but unfortunately $m_{\oy}(x)$ isn't known so we use
$\widetilde{m}_{\oy}(x)$ and solve the perturbed linear system:
\begin{equation} \label{eq:pertrubed-system-Fx}
    \begin{bmatrix}
        \widetilde{m}_{N_{1}}(x)\widetilde{\kappax}^{-N_{1}}   \\
        \widetilde{m}_{2N_{1}}(x)\widetilde{\kappax}^{-2N_{1}} \\
        \vdots                                             \\
        \widetilde{m}_{(d+1)N_{1}}(x)\widetilde{\kappax}^{-(d+1)N_{1}}
    \end{bmatrix}
    = V_{N_{1}}^{d} \cdot
    \begin{bmatrix}
        \widetilde{\alpha}_{0}(x) \\
        \widetilde{\alpha}_{1}(x) \\
        \vdots          \\
        \widetilde{\alpha}_{d}(x)
    \end{bmatrix},
\end{equation}
to get the approximation for the jump magnitudes,$\left\{ \widetilde{A_{l}}(x) \right\}_{l=0}^{d}$,
where $\widetilde{\alpha}_{l}(x) = {\imath}^{l} \widetilde{A}_{d-l}(x)$.
\newline
\\
The next and final step is to recover $\Fx$ by:
\begin{equation} \label{eq:recovery-of-Fx}
    \begin{split}
        \widetilde{\Fx}(y) &= \sum_{\abs{\oy}\leq N} c_{\oy}(\widetilde{\gammax}) e^{\imath\oy y}+\sum_{l=0}^{d} \widetilde{A_{l}}(x_0)V_{l}(y;\widetilde{\yx}) \\
        &= \sum_{\abs{\oy}\leq N} \Bigl( \widetilde{\psiOy}(x) - c_{\oy}(\widetilde{\Phidx}) \Bigr) e^{\imath\oy y} + \sum_{l=0}^{d} \widetilde{A_{l}}
            (x_0)V_{l}(y;\widetilde{\yx}).
\end{split}
\end{equation}

\subsection{Proof of \cref{thm:general-theorem-for-fx}}\label{subsec:accuracy-of-F-at-x}
Here we present the proof for~\cref{thm:general-theorem-for-fx}. In contrast
to~\cref{subsec:accuracy-of-psi} the jump location is unknown and we begin our calculations with the error
in approximating the jump location.
The second step would be to analyze the error in approximating the jump magnitudes of $\Fx$ up to the chosen
reconstruction order of $\Fx$; in the third (and last) step we incorporate the former two steps and
analyze the error of a pointwise approximation of $\Fx$.

\begin{lemma} \label{mk-tiled}
    Let $x\in\torus[1]$ and $\Fx$ as in~\cref{description-of-F-and-xi}, $\Fx\in\PC$ and let $\psiOy\in\PCoy$ as
    in~\cref{def-of-psi-omega-y,psi-oy-properties}.
    Assume that
    \begin{gather*}
        N^{2} \leq M.
    \end{gather*}
    Then there exists $\rgamgam \coloneqq \rgamgam\left(d, B_{F},\mxi, R_{\Gamma_{\mathbb{T}}}, R \right)$ as defined
    in~\cref{r-gamma-gamma} below such that:
    \begin{equation}\label{eq: bound-for-err-in-m-oy-x}
        \abs{\widetilde{m}_{\oy}(x) - \moyx} \leq \rgamgam A_{x} \abs{\oy}^{-1}.
    \end{equation}
\end{lemma}

\begin{proof}
    Using~\cref{eq:eps-at-x-of-oy-bound} we write:
    \begin{align*}
        \abs{\widetilde{m}_{\oy}(x) - \moyx} \leq
        &2\pi \left( R_{\Gamma_{\mathbb{T}}} \abs{\oy}^{-1} + \left(\hl[T]^{*} + R_{d} \h \right) A_{x} \abs{\oy}^{2d+1} M^{-d-1} \right) \\
        =&2\pi\left( \frac{ R_{\Gamma_{\mathbb{T}}}}{\abs{\oy}} + \frac{\left(\hl[T]^{*} + R_{d} \h \right) A_{x}}{\abs{\oy}}
        \cdot \frac{\abs{\oy}^{2d+2}}{M^{d+1}} \right) \\
        \leftidx{_{\abs{\oy}\leq N}}\leq &2\pi\left( \frac{ R_{\Gamma_{\mathbb{T}}}}{\abs{\oy}} +
        \frac{\left(\hl[T]^{*} + R_{d} \h \right) A_{x}}{\abs{\oy}} \cdot \frac{N^{2d+2}}{M^{d+1}}\right) \\
        =&2\pi\left( \frac{ R_{\Gamma_{\mathbb{T}}}}{\abs{\oy}} +
        \frac{\left(\hl[T]^{*} + R_{d} \h \right) A_{x}}{\abs{\oy}} \cdot \left( \frac{N^{2}}{M} \right)^{{d+1}}\right) \\
        \leftidx{_{N^{2}\leq M}}\leq &2\pi\left(\frac{R_{\Gamma_{\mathbb{T}}} + \left(\hl[T]^{*} + R_{d} \h \right) A_{x}}{\abs{\oy}}\right)
    \end{align*}
    and we get:
    \begin{gather*}
        \big| \widetilde{m}_{\oy}(x) - \moyx \big| \leq
        2\pi \left( R_{\Gamma_{\mathbb{T}}} + \left(\hl[T]^{*} + R_{d} \h \right) A_{x} \right) \abs{\oy}^{-1}.
    \end{gather*}
    Because $R_{\Gamma_{\mathbb{T}}}$, $R_{d}$, $\hl[T]^{*}$ and $\h$ are constants which are dependent only on
     $d,\; B_{F},\; \mxi$ and $R$ (recall \cref{R-gamma-oy}) we will define a function that will replace this complicated structure for the sake of clarity
    in the following way:
    \begin{gather}\label{r-gamma-gamma}
        \begin{split}
            \mathcal{R}_{\Gamma|\gamma} \coloneqq
            \mathcal{R}_{\Gamma|\gamma} (d, B_{F}, \mxi, R_{\Gamma_{\mathbb{T}}}, R) \text{ s.t.:}\\
            2\pi \left( R_{\Gamma_{\mathbb{T}}} + \left(\hl[T]^{*} + R_{d} \h \right) A_{x} \right) \leq \mathcal{R}_{\Gamma|\gamma} A_{x}
        \end{split}
    \end{gather}
    and therefore we can finish the proof:
    \begin{equation*}
        \abs{\widetilde{m}_{\oy}(x) - \moyx} \leq \rgamgam A_{x} \abs{\oy}^{-1}.\qedhere
    \end{equation*}
\end{proof}

\subsubsection{Jump location}
Recall~\cref{definition-of-Fx-and-jump-location-and-jump-magnitudes-and-assumptions}. The main approximation result for the jump location is the following.

\begin{proposition}
    \label{main-result-on-jump-location-Fx}
    Let $N_{1} \coloneqq \floor*{\frac{N}{d+2}}$, $p_{N_{1}}^{d}$ and $q_{N_{1}}^{d}$ be given by~\cref{pnd-for-fx,qnd-for-fx}, respectively. Then there exists $C_{2,d}$
    and $\rgamgam$ (see~\cref{bound-on-yx,r-gamma-gamma}) such that
    \begin{gather}
        \abs{\widetilde{\yx} - \yx} \leq C_{2,d} \frac{A_{x}}{\abs{A_{0}(x)}} \rgamgam N^{-d-2}.
    \end{gather}
\end{proposition}
\begin{proof}
    Let $\left\{ y_{1}^{(N_{1})},\ldots,y_{d}^{(N_{1})} \right\}$ be the roots of $q_{N_{1}}^{d}$ and
    $\left\{ u_{1}^{(N_{1})},\ldots,u_{d}^{(N_{1})} \right\}$ be the roots of $p_{N_{1}}^{d}$.
    Using $Lemma\;18$ from~\cite{batenkov2015complete} and~\cref{mk-tiled} implies that for a large
    $N_{1} = \Big\lfloor \frac{N}{d+2} \Big\rfloor$ there exists $C_{1,d}\coloneqq C_{1}(d)$ such that:
    \begin{gather*}
        \abs{y_{i}^{(N_{1})} - u_{i}^{(N_{1})}} \leq C_{1,d} \frac{A_{x}}{\abs{A_{0}(x)}} \rgamgam \left( N_{1} \right)^{-d-1}.
    \end{gather*}
    Keeping in mind that one of the roots of $p_{N_{1}}^{d}$ is $\kappax^{N_{1}} = \left( e^{-\imath \yx} \right)^{N_{1}}$,
    so by using the bound found above we write:
    \begin{equation*}
        \widetilde{\kappax}^{N_{1}} = \kappax^{N_{1}} + \frac{C_{1}^{*}(N_{1})}{(N_{1})^{d+1}},
        \text{ where } \abs{C_{1}^{*}(N_{1})} \leq C_{1,d} \frac{A_{x}}{\abs{A_{0}(x)}} \rgamgam.
    \end{equation*}
    Extraction of the $(N_{1})^{th}$ root further decreases the error by the factor of $\frac{1}{N_{1}}$:
    \begin{align*}
        \abs{\widetilde{\kappax} - \kappax} =& \abs{1-\frac{\widetilde{\kappax}}{\kappax}} =
        \abs{1-\left( \frac{\widetilde{\kappax}^{N_{1}}}{\kappax^{N_{1}}} \right)^{\frac{1}{N_{1}}}} \\
        \leftidx{_{\left( {\abs{C_{1}^{**}(N_{1})}\leq C_{1,d} \frac{A_{x}}{A_{0}(x)} \rgamgam} \right)}}=&
        \abs{1-\left( 1 + \frac{C_{1}^{**}(N_{1})}{(N_{1})^{d+1}} \right)^{\frac{1}{N_{1}}}} \\
        \leftidx{_\text{Bernoulli's inequality}} \leq & C_{1,d} \frac{A_{x}}{\abs{A_{0}(x)}} \rgamgam (N_{1})^{-d-2}.
    \end{align*}
    The final step would be to recover $\yx$ from $\kappax = e^{-\imath\yx}$. Using the above bound
    we write $\widetilde{\kappax} = \kappax + C_{1}^{\#}(N_{1}) \frac{A_{x}}{\abs{A_{0}(x)}} \rgamgam (N_{1})^{-d-2}$,\quad
    $\abs{C_{1}^{\#}(N)} \leq C_{1,d}$ and get:
    \begin{gather*}
        \abs{\widetilde{y_{x_{0}}} - y_{x_{0}}} = \abs{\log\left( \frac{\widetilde{y_{x_{0}}}}{y_{x_{0}}} \right)}
        = \abs{\log\left( 1 + C_{1}^{\#\#}(N_{1}) \frac{A_{x}}{\abs{A_{0}(x)}} \rgamgam (N_{1})^{-d-2} \right)},\\
        \text{where } \abs{C_{1}^{\#\#}(N_{1}) } \leq C_{1,d}.
    \end{gather*}
    Using the estimation of $\abs{\log\left( 1+\varepsilon\right)} < 2\abs{\varepsilon}$ for $\abs{\varepsilon}\ll 1$
    we get:
    \begin{equation*}
        \abs{\widetilde{\yx} - \yx} \leq 2C_{1,d} \frac{A_{x}}{\abs{A_{0}(x)}} \rgamgam N_{1}^{-d-2}.
    \end{equation*}
    Since $N_{1} = \biggl\lfloor{\frac{N}{d+2}}\biggr\rfloor$ and by assuming that
    $N > 2(d+2)$ we have that for $k\in\mathbb{N}$:
    \begin{equation}
        \label{bound-on-N}
        N^{-k} \leq \left( 2(d+2) \right)^{k} N^{-k}.
    \end{equation}
    Denoting:
    \begin{gather*}
        C_{2,d} \coloneqq 2C_{1,d}\left( 2d+4 \right)^{d+2},
    \end{gather*}
    we finally have:
    \begin{gather} \label{bound-on-yx}
        \begin{split}
            \abs{\widetilde{\yx} - \yx} &\leq C_{2,d} \frac{A_{x}}{\abs{A_{0}(x)}} \rgamgam N^{-d-2}
        \end{split}
    \end{gather}
    completing the proof of~\cref{main-result-on-jump-location-Fx}.
\end{proof}

\subsubsection{Jump magnitudes}
\noindent Now we proceed to the second part of~\cref{thm:general-theorem-for-fx} and investigate the accuracy of
recovering the jump magnitudes of $\ddxk[l]\Fx$ at $\yx$. Subtracting the unperturbed system \cref{exact-vnd-system-Fx} from the
perturbed system \cref{eq:pertrubed-system-Fx} we get that the error vector satisfies:
\begin{equation}
    \label{eq:error-vnd-system}
    \begin{bmatrix}
        \widetilde{\alpha_{0}}(x) - \alpha_{0}(x_0) \\
        \widetilde{\alpha_{1}}(x) - \alpha_{1}(x_0) \\
        \vdots                                        \\
        \widetilde{\alpha_{d}}(x) - \alpha_{d}(x_0)
    \end{bmatrix}
    = \left( V_{N_{1}}^{d} \right)^{-1}
    \begin{bmatrix}
        \widetilde{m}_{N_{1}}(x)\widetilde{\kappax}^{-N_{1}} - m_{x_{0}}(N_{1})\kappax^{-N_{1}}    \\
        \widetilde{m}_{2N_{1}}(x)\widetilde{\kappax}^{-2N_{1}} - m_{x_{0}}(2N_{1})\kappax^{-2N_{1}} \\
        \vdots                                                                         \\
        \widetilde{m}_{(d+1)N_{1}}(x)\widetilde{\kappax}^{-(d+1)N_{1}} - m_{x_{0}}((d+1)N_{1})
        \kappax^{-(d+1)N_{1}}
    \end{bmatrix}.
\end{equation}

Going back to the bound $A_x$ from~\cref{definition-of-Fx-and-jump-location-and-jump-magnitudes-and-assumptions} and to~\cref{moy-at-x} we have:
\begin{align*}
    \abs{m_{(j+1)N_{1}}(x)} &= \abs{\kappax^{(j+1)N_{1}} \sum_{l=0}^{d} \alpha_{l}(x)(j+1)^{l}
    N_{1}^{l}} \\
    &\leq \abs{\kappax^{(j+1)N_{1}}} \sum_{l=0}^{d} \abs{\alpha_{l}(x)} (j+1)^{l} N_{1}^{l} \\
    &= \sum_{l=0}^{d} \abs{\alpha_{l}(x)} (j+1)^{l} N_{1}^{l} \leq A_{x}\sum_{l=0}^{d} (j+1)^{l}
    N_{1}^{l} \leq A_{x}\sum_{l=0}^{d} (j+1)^{d} N_{1}^{d} \\
    &= A_{x} (d+1) (j+1)^{d} N_{1}^{d} \\
    &\leq A_{x} (d+1) (d+1)^{d} N_{1}^{d} = A_{x} (d+1)^{d+1} N_{1}^{d}.
\end{align*}

Denoting $C_{3,d} = (d+1)^{d+1}$ we get:
\begin{gather*}
    \abs{m_{(j+1)N_{1}}(x)} \leq C_{3,d} A_{x} N_{1}^{d}.
\end{gather*}
Applying~\cref{mk-tiled} we have that there exists $\rgamgam(N_{1})$ which depends only on $N_{1}$ such that:
\begin{gather}
    \widetilde{m}_{(j+1)N_{1}}(x) = m_{(j+1)N_{1}}(x) + \rgamgam(N_{1})N_{1}^{-1}, \quad \rgamgam(N_{1}) \leq \rgamgam
\end{gather}
and from the proof of~\cref{main-result-on-jump-location-Fx} there exists $C_{1}^{\#}(N_{1})$, which depends only on $N_{1}$
such that:
\begin{equation*}
    \widetilde{\kappax} = \kappax + C_{1}^{\#}(N_{1}) \frac{A_{x}}{\abs{A_{0}(x)}} \rgamgam N_{1}^{-d-2},\quad C_{1}^{\#}(N_{1})
    \leq C_{2,d}.
\end{equation*}
Applying Taylor majorization, e.g. \cite[Proposition A.7]{batenkov2012algebraic},
gives the following estimate:
\begin{align*}
    \widetilde{\kappax}^{-(j+1)N_{1}} &= \left( \kappax + C_{1}^{\#}(N_{1}) \frac{A_{x}}{\abs{A_{0}(x)}} \rgamgam N_{1}^{-d-2} \right)
    ^{-(j+1)N_{1}}\\
    &= \kappax^{-(j+1)N_{1}}\left( 1 + \frac{C_{1}^{\#}(N_{1})}{\kappax}\frac{A_{x} }{\abs{A_{0}(x)}}\rgamgam N_{1}^{-d-2} \right)^{-(j+1)N_{1}} \\
    &= \kappax^{-(j+1)N_{1}}\left( 1 - C_{1}^{\#\#}(N_{1}) \frac{A_{x}}{\abs{A_{0}(x)}} \rgamgam N_{1}^{-d-1} \right),\quad
    \abs{C_{1}^{\#\#}(N_{1})} \leq C_{2,d}.
\end{align*}
Consequently,
\begin{equation}\label{eq:estimate-for-zeta}
    \begin{split}
        &\begin{split}
            &\abs{\widetilde{m}_{(j+1)N_{1}}(x)\widetilde{\kappax}^{-(j+1)N_{1}} - m_{(j+1)N_{1}}(x)\kappax^{-(j+1)N_{1}}} = \\
            &=\Big| \left( m_{(j+1)N_{1}}(x) + \rgamgam(N_{1})N_{1}^{-1}\right) \kappax^{-(j+1)N_{1}} \\
            &\qquad \qquad \left( 1 - C_{1}^{\#\#}(N_{1}) \frac{A_{x}}{\abs{A_{0}(x)}} \rgamgam N_{1}^{-d-1} \right) - m_{(j+1)N_{1}}(x)\kappax^{-(j+1)N_{1}} \Big| \\
        \end{split}\\
        &\begin{split}
            &\leq \frac{\rgamgam}{N_{1}}\abs{\frac{C_{2,d}C_{3,d}A_{x}}{\abs{A_{0}(x)}}+1} + O\left( N_{1}^{-d-2} \right) \\
            &\leq C_{4,d} \left( 1 + \frac{A_{x}}{\abs{A_{0}(x)}} \right) \rgamgam N_{1}^{-1},\quad
            C_{4,d}=\max\left\{ 1,C_{2,d}, C_{3,d}\right\}.
        \end{split}
    \end{split}
\end{equation}
Denote $\zeta_{j}(x) = \widetilde{m}_{(j+1)N_{1}}(x)\widetilde{\kappax}^{-(j+1)N_{1}} - m_{(j+1)N_{1}}(x)\kappax^{-(j+1)N_{1}}$
and let $\mathcal{v^d}_l$ be an upper bound on the sum of absolute values of the entries in the $l^{th}$ row of
$\left( V_{1}^{d} \right)^{-1}$. We obtain:
\begin{gather*}
    V_{N_{1}}^{d} = V_{1}^{d}\cdot \operatorname{diag}\left\{ 1,N_{1},\ldots,N_{1}^{d} \right\} \\
    \Leftrightarrow \\
    \left( V_{N_{1}}^{d} \right)^{-1} = \operatorname{diag}\left\{ 1,N_{1}^{-1},\ldots,N_{1}^{-d} \right\} \cdot \left( V_{1}^{d} \right)^{-1},
\end{gather*}
therefore,
\begin{gather}
    \begin{bmatrix}
        \widetilde{\alpha_{0}}(x) - \alpha_{0}(x) \\
        \widetilde{\alpha_{1}}(x) - \alpha_{1}(x) \\
        \vdots                                        \\
        \widetilde{\alpha_{d}}(x) - \alpha_{d}(x)
    \end{bmatrix}
    =
    \begin{bmatrix}
        1 &        &        & &        \\
        & N_{1}^{-1} &        & &        \\
        &        & \ddots & &        \\
        &        &        & & N_{1}^{-d}
    \end{bmatrix}
    \cdot \left( V_{1}^{d} \right)^{-1} \cdot
    \begin{bmatrix}
        \zeta_{0}(x) \\
        \zeta_{1}(x) \\
        \vdots         \\
        \zeta_{d}(x) \\
    \end{bmatrix}.
\end{gather}
Denoting $C_{5,d,l} \coloneqq C_{4,d}\cdot \mathcal{v^d}_l$ and applying~\cref{eq:estimate-for-zeta} we have,
for $l=0,\ldots,d$:
\begin{gather}
    \begin{split}
        \abs{\widetilde{\alpha_{l}}(x) - \alpha_{l}(x)} &\leq
        C_{5,d,l} \left( 1 + \frac{A_{x}}{\abs{A_{0}(x)}} \right) \rgamgam N_{1}^{-l-1} \\
        &\overset{A_0(x)\leq A_{x}}{\leq} 2C_{5,d,l} \frac{A_{x}}{\abs{A_{0}(x)}} \rgamgam N_{1}^{-l-1}
    \end{split}
\end{gather}
and by going back to $\alpha_{l}(x) = {\imath}^{l}A_{d-l}(x)$ (see~\cref{alpha-l-x}) we have that:
\begin{equation*}
    \begin{split}
        \abs{\widetilde{\alpha_{l}}(x) - \alpha_{l}(x)} &=\abs{\widetilde{A_{d-l}}(x) - A_{d-l}(x)}.
    \end{split}
\end{equation*}
Finally by denoting
\begin{gather*}
    C_{6,d-l} \coloneqq C_{5,d,l} \cdot 2^{l-d} \cdot \left( d+1 \right)^{l-d-1}
\end{gather*}
and applying~\cref{bound-on-N} we conclude that:
\begin{gather}\label{Ax-tilde-bound}
    \abs{\widetilde{A_{l}}(x) - A_{l}(x)} \leq C_{6,d-l} \frac{A_{x}}{\abs{A_{0}(x)}} \rgamgam N^{l-d-1}.
\end{gather}

\subsubsection{Pointwise values}
Here we prove the last part of~\cref{thm:general-theorem-for-fx}, i.e. the accuracy of recovering $\Fx$.

We begin by denoting
\begin{equation} \label{eq:definition-for-Fx-M}
    F_{x, N}(y) \coloneqq \sum_{\abs{\oy}\leq N}\left( c_{\oy}\left( \Fx \right) - c_{\oy}\left(
    \Phidx \right)\right)e^{\imath\oy y} + \Phidx(y).
\end{equation}
\begin{proposition} \label{bound-for-error-of-fx}
    Let $\Fx:\mathbb{T}\rightarrow \mathbb{R}$ as described in~\cref{definition-of-Fx-and-jump-location-and-jump-magnitudes-and-assumptions} and
    let $r>0$.
    Then for a large enough $M$ and for every $\abs{\oy}\leq N$ for which~\cref{alg:full-order-algorithm} recovers
    $\left\{ \widetilde{\psiOy}(x) \right\}$, as $\oy=-N,\ldots,N$
    (where also $N^{2} \leq M$) with the accuracy as in~\cref{thm:main-result-for-psi},
    there exist $\rgamgam$, $R_{\Gamma_{\mathbb{T}}}$ and $C_{10,d}$ (see~\cref{r-gamma-gamma,eq:absolut-bound-for-gamma,final-constant-point-wise}) such that:
    \begin{equation*}
        \abs{\widetilde{\Fx}(y) - \Fx(y)} \leq C_{10,d}\frac{A_{x}(1+A_x)}{\abs{A_{0}(x)}} \rgamgam N^{-d-1},\quad y\in\mathbb{T}\setminus B_{r}(\yx)
    \end{equation*}
    where $B_{r}(\yx)$ is a ball of radius $r>0$ centered at $\yx$.
\end{proposition}

\begin{proof}
    Consider
    \begin{equation*}
        \abs{\widetilde{\Fx}(y) - \Fx(y)} \leq \abs{\widetilde{\Fx}(y) - F_{x,N}(y)}+\abs{F_{x,N}(y)-\Fx(y)}.
    \end{equation*}
    We begin evaluating the second term on the right side of the above inequality.  From
    $\gammax = \Fx - \Phidx$, and applying~\cref{decomposition-of-Fx,eq:absolut-bound-for-gamma}, we have:
    \begin{align} \label{second-term-bound}
        \begin{split}
            \abs{F_{x,N}(y) - \Fx(y)} &= \abs{\sum_{\abs{\oy} \leq N} c_{\oy}\left( \gammax \right)e^{\imath\oy y} - \gammax(y) } \\
            &= \abs{\sum_{\abs{\oy} > N} c_{\oy}\left( \gammax \right)e^{\imath\oy y} } \leq C^{**} R_{\Gamma_{\mathbb{T}}} N^{-d-1},
        \end{split}
    \end{align}
    where $C^{**} \in \mathbb{R}^{+}$.

    Now we consider the first term. Letting
    $\Theta_{x}\coloneqq \widetilde{\Phidx} - \Phidx$ we claim that:
    \begin{align*}
        \abs{\widetilde{\Fx}(y) - F_{x,N}(y)} &= \abs{\sum_{\abs{\oy}\leq N}\left( c_{\oy}\left( \Phidx \right) -
        c_{\oy}\left(\widetilde{\Phidx} \right)\right)e^{\imath\oy y} + \widetilde{\Phidx}(y) - \Phidx(y)} \\
        &= \abs{\sum_{\abs{\oy}\leq N} c_{\oy}\left( \Theta_{x} \right)e^{\imath\oy y} - \Theta_{x}(y)}.
    \end{align*}
    Now using~\cref{bound-on-yx,Ax-tilde-bound} we claim that there exist $\beta_x(N)$ and $\beta_{l,x}(N)$ for
    each $0\leq l \leq d$ s.t.:
    \begin{eqnarray} \label{y-tiled-x-and-A-tiled-x}
        \begin{split}
            \widetilde{\yx} &= \yx + \beta_x(N),&\quad \abs{\beta_x(N)}
            &\leq C_{2,d}\frac{A_x}{\abs{A_{0}(x)}} \rgamgam N^{-d-2}, \\
            \widetilde{A_{l}}(x) &= A_{l}(x) + \beta_{l,x}(N),&\quad \abs{\beta_{l,x}(N)}
            &\leq C_{6,d-l} \frac{A_{x}}{\abs{A_{0}(x)}} \rgamgam N^{l-d-1}.
        \end{split}
    \end{eqnarray}
    Applying~\cref{eq:phidx} we get:
    \begin{align*}
        \Theta_{x}(y) &= \sum_{l=0}^{d} \left(\widetilde{A_{l}}(x)V_{l}(y;\widetilde{\yx})
        - A_{l}(x) V_{l}(y;\yx) \right) \\
        &= \sum_{l=0}^{d} \left(\widetilde{A_{l}}(x)V_{l}(y;\widetilde{\yx})
        - \left( \widetilde{A_{l}}(x) - \beta_{l}(N) \right) V_{l}(y;\yx) \right)\\
        &= \sum_{l=0}^{d} \widetilde{A_{l}}(x) \Bigl( V_{l}(y;\widetilde{\yx}) - V_{l}(y;\yx) \Bigr)
        + \sum_{l=0}^{d} \beta_{l,x}(N) V_{l}(y;\yx) \\
        &= \sum_{l=0}^{d} \widetilde{A_{l}}(x) \Bigl( V_{l}(y;\yx)+\beta_{l,x}(N)) - V_{l}(y;\yx)
        \Bigr)  + \sum_{l=0}^{d} \beta_{l,x}(N) V_{l}(y;\yx).
    \end{align*}
    
    Now let $r>0$ be given.
    Since $\inf_{x\in\torus[1]}|A_0(x)|=A_L>0$, for large enough $N'$ we will have
    \begin{gather}\label{condition-for-r}
        \beta_x(N)<r \text{ for all } N > N^{'}, \text{ uniformly in } x.
    \end{gather}
    
    For every $\epsilon<r$, we define:
    \begin{gather*}\label{U-l-eps-func}
        U_{l,\epsilon}(y) \coloneqq V_{l}\left( y;\yx+\epsilon \right) - V_{l}\left( y;\yx \right)
    \end{gather*}
    and take $\epsilon = \beta_x(N)$ to get:
    \begin{gather*}
        \Theta_{x}(y) = \sum_{l=0}^{d} \beta_{l,x}(N)V_{l}(y;\yx) + \sum_{l=0}^{d} \widetilde{A_{l}}(x)
        U_{l,\beta_x(N)}(y).
    \end{gather*}
    Once again we denote:
    \begin{gather*}\label{z-func-at-x}
        \mathcal{Z}_x(y) = \sum_{l=0}^{d} \beta_{l,x}(N) V_{l}(y;\yx) \\
        \mathcal{W}_{x}(y) = \sum_{l=0}^{d} \widetilde{A_{l}}(x) U_{l,\beta_x(N)}(y)
    \end{gather*}
    and write:
    \begin{gather*}
        \Theta_{x}(y) = \mathcal{Z}_x(y) + \mathcal{W}_{x}(y).
    \end{gather*}
    Therefore
    \begin{align*}
        \abs{\widetilde{\Fx}(y) - F_{x,N}(y)} &= \abs{ \sum_{\abs{\oy} \leq N} c_{\oy}(\Theta_{x})
            e^{\imath\oy y} - \Theta_{x}(y) } \\
        &= \abs{ \sum_{\abs{\oy} \leq N} \Bigl( c_{\oy}(\mathcal{Z}_x) + c_{\oy}(\mathcal{W}_{x}) \Bigr)
            e^{\imath\oy y} - \Theta_{x}(y)} \\
        &= \abs{ \sum_{\abs{\oy} \leq N} \Bigl( c_{\oy}(\mathcal{Z}_x) + c_{\oy}(\mathcal{W}_{x}) \Bigr)
            e^{\imath\oy y} - \mathcal{Z}_x(y) - \mathcal{W}_{x}(y)} \\
        &\leq \abs{ \sum_{\abs{\oy} \leq N} c_{\oy}(\mathcal{Z}_x) e^{\imath\oy y} - \mathcal{Z}_x(y)}
        + \abs{ \sum_{\abs{\oy} \leq N} c_{\oy}(\mathcal{W}_{x}) e^{\imath\oy y} - \mathcal{W}_{x}(y)} \\
        &= \abs{ \sum_{\abs{\oy} > N} c_{\oy}(\mathcal{Z}_x) e^{\imath\oy y} } + \abs{ \sum_{\abs{\oy}
        \leq N} c_{\oy}(\mathcal{W}_{x}) e^{\imath\oy y} - \mathcal{W}_{x}(y)}.
    \end{align*}
    Considering $\mathcal{Z}_x(y) = \sum_{l=0}^{d} \beta_{l,x}(N) V_{l}(y;\yx)$, we have that:
    \begin{gather*}
        c_{\oy}\left( \mathcal{Z}_x\right) = c_{\oy}\left(\sum_{l=0}^{d} \beta_{l,x}(N) V_{l} \right)
        = \sum_{l=0}^{d} \beta_{l,x}(N) c_{\oy} \left( V_{l} \right),
    \end{gather*}
    and since the function $V_{l}$ is in $C_{\mathbb{T}}^{l}$ there exists $S_{l}\in\mathbb{R}^{+}$
    such that:
    \begin{gather*}
        \abs{ \sum_{\abs{\oy} > N} c_{\oy} \left( V_{l} \right) e^{\imath\oy y}} \leq S_{l}.
        N^{-l}
    \end{gather*}
    Consequently, we have that:
    \begin{align}
        \begin{split}
            \notag
            \abs{ \sum_{\abs{\oy} > N} c_{\oy}(\mathcal{Z}_x) e^{\imath\oy y} } &=
            \abs{\sum_{\abs{\oy} > N} \Bigl( \sum_{l=0}^{d} \beta_{l,x}(N) c_{\oy} \left( V_{l} \right)
            \Bigr) e^{\imath\oy y} } \\
            &\leq \sum_{l=0}^{d} \abs{\beta_{l,x}(N)} \abs{ \sum_{\abs{\oy} > N} c_{\oy} \left( V_{l}
            \right) e^{\imath\oy y}} \\
            &\leq (d+1) C_{6,d-l} \frac{A_{x}}{\abs{A_{0}(x)}} \rgamgam N^{l-d-1} \sum_{\abs{\oy} >
            N}  \abs{ c_{\oy} \left( V_{l} \right)} \\
            &\leq (d+1)C_{6,d-l} \frac{A_{x}}{\abs{A_{0}(x)}} \rgamgam N^{l-d-1} S_{l} N^{-l} \\
            &= (d+1) C_{6,d-l} S_{l} \frac{A_{x}}{\abs{A_{0}(x)}} \rgamgam N^{-d-1}.
        \end{split}
    \end{align}
    Denoting:
    \begin{align}
        C_{7,d} &\coloneqq (d+1)\cdot \max\limits_{0\leq l\leq d}\left\{C_{6,d-l} \cdot S_l\right\}
    \end{align}
    we get:
    \begin{align}
        \abs{ \sum_{\abs{\oy} > N} c_{\oy}(\mathcal{Z}_x) e^{\imath\oy y} } &\leq C_{7,d} \frac{A_{x}}{\abs{A_{0}(x)}}
        \rgamgam N^{-d-1} \label{bound-on-Z}.
    \end{align}

    Now we move on to $\mathcal{W}_{x}(y) = \sum_{l=0}^{d} \widetilde{A_{l}(x)} U_{l,\beta_x(N)}(y)$.
    The function $U_{l,\beta_x(N)}(y)$ is defined at the ``\textit{jump-free}'' region so there exists $C_{4} \in \mathbb{R}$
    (uniformly in $x$) such that $\abs{U_{l,\beta_x(N)}(y)} \leq C_{4}\epsilon$.
    This bound can be obtained by Taylor-expanding the function $V_{l}(y;\yx+\epsilon)$ at $\epsilon = 0$.
    In particular, for $\epsilon = \beta_x(N)$ we have:
    \begin{align*}
        \abs{ \mathcal{W}_{x}(y) } &= \abs{\sum_{l=0}^{d} \widetilde{A_{l}}(x) U_{l,\beta_x(N)}(y)} \leq \sum_{l=0}^{d} \abs{ \widetilde{A_{l}}(x)} \abs{ U_{l,\beta_x(N)}(y)} \leq \sum_{l=0}^{d} A_{x} \cdot C_{4} \cdot \abs{\beta_x(N)} \\
        &\leq \sum_{l=0}^{d} \left( C_{4} \cdot C_{2,d}\frac{A_{x}^2}{\abs{A_{0}(x)}} \rgamgam N^{-d-2} \right)\\
        &= (d+1)\cdot C_{4} \cdot C_{2,d}\frac{A_{x}^2}{\abs{A_{0}(x)}} \rgamgam N^{-d-2}.
    \end{align*}
    \indent Denoting $C_{8,d} = (d+1)\cdot C_{2,d} \cdot C_{4}$ we get:
    \begin{gather}
        \label{bound-on-W}
        \abs{ \mathcal{W}_{x}(y) } \leq C_{8,d} \frac{A_{x}^2}{\abs{A_{0}(x)}} \rgamgam N^{-d-2}.
    \end{gather}
    Continuing our analysis we look at $U_{l,\epsilon}(y)$, for $\epsilon = \beta_x(N)$:
    \begin{align*}
        U_{l,\epsilon}(y) &= V_{l}\left( y;\yx+\epsilon \right) - V_{l}\left( y;\yx \right) \\
        \\
        &=\frac{(2\pi)^{l}}{(l+1)!} \left( B_{l+1} \left( \frac{y-( \yx + \beta_x(N))}{2\pi} \right) - B_{l+1}
        \left( \frac{y-\yx}{2\pi} \right) \right) \\
        \\
        &=\frac{(2\pi)^{l}}{(l+1)!} \left( B_{l+1} \left( \frac{y-\widetilde{\yx}}{2\pi} \right) - B_{l+1}
        \left( \frac{y-\yx}{2\pi} \right) \right).
    \end{align*}
    Therefore, there exists a constant
    $C_{5}$ such that in the region of length $\beta_x(N)$ between $\yx$ and $\widetilde{\yx}$ we have:
    \begin{gather}
        \label{bound-on-U-between-xi-xi-tiled}
        \abs{U_{l,\beta_x(N)}(y)} \leq C_{5} \cdot A_{x}.
    \end{gather}
    As mentioned before, at the ``\textit{jump-free}'' area we have:
    \begin{gather} \label{bound-on-U-in-jump-free}
        \abs{U_{l,\beta(N)}(y)} \leq C_{4} C_{2,d}\frac{A_x^2}{\abs{A_{0}(x)}} \rgamgam N^{-d-2}.
    \end{gather}
    We conclude that therefore exists a constant $C_{6}$ such that the Fourier coefficients of $\mathcal{W}_{x}$ are bounded by:
    \begin{gather*}
        \abs{ c_{\oy}\left( \mathcal{W}_{x} \right) } \leq C_{6} C_{2,d}\frac{A_x^2}{\abs{A_{0}(x)}} \rgamgam N^{-d-2}
    \end{gather*}
    which in turn allows us to conclude that exists another constant $C_{7}$ such that:
    \begin{gather} \label{bound-on-trunc-sum-of-W}
        \abs{ \sum_{\abs{\oy}\leq N} c_{\oy}(\mathcal{W}_{x}) e^{\imath\oy y} } \leq
        C_{7} C_{2,d}\frac{A_{x}^2}{\abs{A_{0}(x)}} \rgamgam N^{-d-1}.
    \end{gather}

    Now by using~\cref{bound-on-Z,bound-on-W,bound-on-U-between-xi-xi-tiled,bound-on-U-in-jump-free,bound-on-trunc-sum-of-W}
    we get:
    \begin{align*}
        \abs{\widetilde{\Fx}(y) - F_{x,N}(y)} &\leq \abs{ \sum_{\abs{\oy} > N} c_{\oy}(\mathcal{Z}_x) e^{\imath\oy y} } + \abs{ \sum_{\abs{\oy}
        \leq N} c_{\oy}(\mathcal{W}_{x}) e^{\imath\oy y} - \mathcal{W}_{x}(y)} \\
        &\leq \abs{ \sum_{\abs{\oy} > N} c_{\oy}(\mathcal{Z}_x) e^{\imath\oy y} } + \abs{ \sum_{\abs{\oy}
        \leq N} c_{\oy}(\mathcal{W}_{x}) e^{\imath\oy y}} + \abs{ \mathcal{W}_{x}(y) } \\
        &\leq C_{7,d} \frac{A_{x}^2}{\abs{A_{0}(x)}} \rgamgam N^{-d-1} +
        C_{7} C_{2,d}\frac{A_{x}^2}{\abs{A_{0}(x)}} \rgamgam N^{-d-1} +
        C_{8,d} \frac{A_{x}^2}{\abs{A_{0}(x)}} \rgamgam N^{-d-2} \\
        &\leq \left(C_{7,d} + C_{7} C_{2,d} + C_{8,d} \right) \frac{A_{x}^2}{\abs{A_{0}(x)}} \rgamgam N^{-d-1}.
    \end{align*}
    \noindent Denoting:
    \begin{gather} \label{H-x}
        C_{9,d} \coloneqq C_{7,d} + C_{7} C_{2,d} + C_{8,d}
    \end{gather}
    we have:
    \begin{gather}\label{first-term-bound}
        \abs{\widetilde{\Fx}(y) - F_{x,N}(y)} \leq C_{9,d}\frac{A_{x}^2}{\abs{A_{0}(x)}} \rgamgam N^{-d-1},
    \end{gather}
    and by applying~\cref{first-term-bound,second-term-bound} we have:
    \begin{align}
        \begin{split}
            \abs{\widetilde{\Fx}(y) - \Fx(y)} &\leq \abs{\widetilde{\Fx}(y) - F_{x,N}(y)}+\abs{F_{x,N}(y)-\Fx(y)} \\
            &\leq \left( C_{9,d}\frac{A_{x}^2}{\abs{A_{0}(x)}} \rgamgam + C^{**} R_{\Gamma_{\mathbb{T}}} \right) N^{-d-1}.
        \end{split}
    \end{align}
    By the definition of $\rgamgam$ (see~\cref{r-gamma-gamma}) we claim that $R_{\Gamma_{\mathbb{T}}} \leq \rgamgam$ which implies that
    \begin{gather*}
        \left( C_{9,d}\frac{A_{x}^2}{\abs{A_{0}(x)}} \rgamgam + C^{**} R_{\Gamma_{\mathbb{T}}} \right) N^{-d-1} \leq
        \left( C_{9,d}\frac{A_{x}^2}{\abs{A_{0}(x)}} + C^{**} \right) \rgamgam N^{-d-1}.
    \end{gather*}
    Finally, by denoting
    \begin{gather}\label{final-constant-point-wise}
        C_{10,d} = \max \left\{ C_{9,d},\; C^{**} \right\}
    \end{gather}
    we obtain
    \begin{gather*}\label{final-pointwise-bound}
        \abs{\widetilde{\Fx}(y) - \Fx(y)} \leq C_{10,d}\frac{A_{x}(1+A_x)}{\abs{A_{0}(x)}} \rgamgam N^{-d-1},
    \end{gather*}
    which completes the proof.
\end{proof}

\subsection{The full algorithm}\label{subsec:algorithm}
%

In this section we present the entire algorithm for reconstructing the slices $F_x$ of $F$ from its Fourier coefficients.
    
The 1D reconstruction procedure from \cite{batenkov2012algebraic,batenkov2015complete}, is summarized in \cref{alg:full-order-algorithm} below, while our 2D algorithm is presented in \cref{alg:2D reconstruction}.

 \begin{algorithm}
    \caption{Half-order reconstruction algorithm,~\cite{batenkov2012algebraic}}
    Let $f\in \PC$ and assume that $f=\gamma + \varphi$ where $\gamma\in\C$ and $\phi$ is a piecewise polynomial of degree
    $d$ absorbing all of $f$'s discontinuities.
    \begin{algorithmic}[1]
        \Require Fourier coefficients of $\left\{ c_{k}(f) \right\}_{\abs{k}\leq M}$ s.t. $M\gg 1$.
        \State Fix a reconstruction order $d_1\leq \lfloor \frac{d}{2} \rfloor$
        \If{jump location unknown}
            \State Solve the polynomial in ~\cite[eq.~3.3]{batenkov2012algebraic} with $k=M$ and $d=d_1$
            \State Take $\widetilde{\omega}$ to be the closest root to the unit circle
            \State Take $\widetilde{\xi} = \arg\left(\widetilde{\omega}\right)$ as the approximation of the actual discontinuity point
        \Else{ }
            \State Take $\widetilde{\xi} = x_0$
        \EndIf
        \State Obtain the jump magnitudes $\left\{ \widetilde{A_{l}} \right\}_{l=0}^{d_1}$ by using $\widetilde{\omega}$ and by solving the
        linear system in~\cite[eq.~3.5]{batenkov2012algebraic} with $k=M,\ldots,M+d_1$
        \State Obtain coefficients of $\widetilde{\gamma}$, $\left\{ c_k(\widetilde{\gamma}) \right\}_{|k|\leq M}$
        by~\cite[Algorithm,~2.2]{batenkov2012algebraic}
        \State Take the final approximation
        \begin{align*}
            \widetilde{f}(x)&=\widetilde{\gamma}(x)+\widetilde{\phi}(x)=\sum_{|k|\leq M}c_{k}(\widetilde{\gamma})e^{\imath k x} + \sum_{l=0}^{d_1}\widetilde{A_{l}}V_{l}(x;\widetilde{\xi}).
        \end{align*}
    \end{algorithmic}\label{alg:half-order-algorithm}
 \end{algorithm}

 \begin{algorithm}
    \caption{Full-order reconstruction algorithm,~\cite{batenkov2015complete}}
    Let $f:\PC$ and assume that $f=\gamma + \varphi$ where $\gamma\in\C$ and $\phi$ is a piecewise polynomial of degree
    $d$ absorbing all of $f$'s discontinuities.
    \begin{algorithmic}[1]
        \Require Fourier coefficients of $\left\{ c_{k}(f) \right\}_{\abs{k}\leq M}$ s.t. $M\gg 1$.
        \State Use~\cref{alg:half-order-algorithm} to approximate
        $\widetilde{\omega}_h = e^{-\imath\widetilde{\xi}_h}$
        \State Take $N=\lfloor \frac{M}{d+2} \rfloor$
        \State Construct the polynomial $q_{N}^{d}(u)=\sum_{j=0}^{d+1} (-1)^j \binom{d+1}{j}\widetilde{m}_{(j+1)N}u^{d+1-j}$ where
        $\widetilde{m}_{k}\coloneqq 2\pi (\imath k)^{d+1}c_{k}(f)$, and find it's roots.
        \State Take $\widetilde{z}$ to be the closest root to the unit circle and denote $\sqrt[N]{Z_{q_N^d}}$ as the set of $N$ possible
        values of $\sqrt[N]{\widetilde{z}}$
        \State Take $\widetilde{\omega}_f\coloneqq \arg\min\limits_{z\in\sqrt[N]{Z_{q_N^d}}} \Big\{ \big|z - \widetilde{\omega}_h \big| \Big\}$
        and set $\widetilde{\xi}_f = -\arg(\widetilde{\omega}_f)$
        \State Recover the jump magnitudes by solving the perturbed linear system
            \[\begin{bmatrix}
                \widetilde{m}_N\widetilde{\omega}_f^{-N}   \\
                \widetilde{m}_{2N}\widetilde{\omega}_f^{-2N} \\
                \vdots                                      \\
                \widetilde{m}_{(d+1)N}\widetilde{\omega}_f^{-(d+1)N}
            \end{bmatrix}
            = V_{N}^{d} \cdot
            \begin{bmatrix}
                \widetilde{\alpha}_{0} \\
                \widetilde{\alpha}_{1} \\
                \vdots          \\
                \widetilde{\alpha}_{d}
            \end{bmatrix}
            \]
        and extracting $\left\{ \widetilde{A_l} \right\}_{l=0}^{d}$ by using $\widetilde{A}_{l} = (-\imath)^{d-l}\widetilde{\alpha}_{d-l}$
        \State Obtain coefficients of $\widetilde{\gamma}$ using $c_k(\widetilde{\gamma}) = c_k(f)-c_k(\widetilde{\phi})$,
        where \[c_k(\widetilde{\phi})=\frac{\widetilde{\omega}_f^k}{2\pi} \sum_{l=0}^{d}\frac{\widetilde{A_l}}{(\imath k)^{l+1}}\]
        \State Take the final approximation
        \begin{align*}
            \widetilde{f}(x)&=\widetilde{\gamma}(x)+\widetilde{\phi}(x)=\sum_{|k|\leq M}c_{k}(\widetilde{\gamma})e^{\imath k x} + \sum_{l=0}^{d_1}\widetilde{A_{l}}V_{l}(\widetilde{\xi}).
        \end{align*}
    \end{algorithmic}\label{alg:full-order-algorithm}
 \end{algorithm}

 \begin{algorithm}
    \caption{2D reconstruction}
    Let $F:\torus \to \mathbb{R}$ and $\Fx:\torus[1]\to\mathbb{R}$ as defined in~\cref{description-of-F-and-xi} and
    $\psiOy\in\PCoy$ as defined in~\cref{def-of-psi-omega-y,psi-oy-properties}.
    \begin{algorithmic}[1]
        \Require $(2M+1) \times (2N+1)$ Fourier coefficients of $F$ s.t. $N^2 \leq M$
        and a sets of samples $X=\left\{ x_0,\ldots,x_{n} \right\}\subset\torus[1]$,
        $Y = \left\{ y_0,\ldots,y_{m} \right\}\subset\torus[1]$
        \For{$x\in X$}
            \For{$\oy \gets -N \to N$}
                \State $\widetilde{\psiOy}(x)\gets$ apply~\cref{alg:full-order-algorithm} using 
                $\left\{ c_{\ox}(\psiOy) \right\}_{\abs{\ox}\leq M}$ and known jump location $x_0=-\pi$
            \EndFor
            \For{$y\in Y$}
                \State $\widetilde{F_x}(y)\gets$ apply~\cref{alg:full-order-algorithm} using
                $\left\{ \widetilde{\psiOy}(x) \right\}_{|\oy|\leq N}$
            \EndFor
        \EndFor
    \end{algorithmic}\label{alg:2D reconstruction}
 \end{algorithm}

\section{Numerical Experiments}\label{sec:numerical-experiments}
In this section we provide simulations with the primary goal of validating the asymptotic accuracy predictions
developed in this work.
We provide the actual Fourier coefficients, which were calculated analytically, for the function
$F:\torus[2]\to\mathbb{R}$ which is defined as follows:
\begin{gather}
    U_n(y)\coloneqq
    \begin{cases}
        B_{n+1}(\frac{y+2\pi}{2\pi}),\quad y\in [-2\pi,0) \\
        B_n(\frac{y}{2\pi}),\quad y\in [0, 2\pi)
    \end{cases}
\end{gather}
as $B_n(y)$ is the Bernoulli function~\cite{sun2004identities}
\begin{align*}
    V_n(x,y) &\coloneqq -\frac{(2\pi)^n}{(n+1)!}U_n(y-\xi(x)) \\
    \Phi_d(x,y) &\coloneqq \sum_{l=0}^d A_{l}V_{l}(x,y)
\end{align*}
where $\xi(x)$ is the discontinuity curve and $A_l(x)$ is the jump magnitude at $x$.
Setting $\xix=x$ and $\Al = 1$ as $l=0,\ldots,d$ we simulate $F:\torus[2]\to\mathbb{R}$ which is defined by
\begin{gather}
    F(x,y)\coloneqq \Phi_{11}(x,y)
\end{gather}
In our simulation we fix a point $x\in\torus[1]$ to get the slice $F_x$. We set the reconstruction orders of $\psiOy$ and of $\Fx$ to 9, i.e. $d=9$. We further set number of the given Fourier coefficients for the unknown function $F:\torus[2]\to\mathbb{R}$ as different sets
of $\left( M, N \right)$ where $ \left( M, N \right) = \left( N^2, N \right)$. In the figures below we calculate the following approximation errors:
\begin{enumerate}
    \item The error in approximating $\Fx$ over a set $Y\subset\torus[1]$
    \begin{gather*}
        \triangle \Fx \coloneqq \max \left\{ \abs{\Fx(y) - \widetilde{\Fx}(y)}\colon y\in\torus[1]\right\},\quad x\in \torus[1]
    \end{gather*}
    which is shown in~\cref{fig:fx-tf-vs-moy}.
    \item The error in approximating the discontinuity curve $\xi(x)$
    \begin{gather*}
        \triangle\xi\coloneqq \max\limits_{x\in\torus [1]} \left\{ \abs{\xi(x) - \widetilde{\xi}(x)} \right\}
    \end{gather*}
    which is shown in~\cref{fig:xi-vs-moy}.
    \item The error in approximating the jump magnitudes:
    \begin{gather*}
        \triangle A_{l} \coloneqq \max\limits_{x\in\torus[1]} \left\{ \abs{A_l(x) - \widetilde{A_{l}}(x)}\right\},\: l=0,\ldots,d
    \end{gather*}
    which is shown in~\cref{fig:al-vs-moy}.
\end{enumerate}
The simulations were performed in a standard Python environment using the \href{https://mpmath.org}{mpmath} library
which in turn allowed us to perform our calculations with a varying precision up to 100 decimal places.
\begin{figure}[H]
    \centering
    \subfloat[\centering Exact $F$ and $\xi$\label{fig:exact-F}]{{\includegraphics[width=0.49\linewidth]{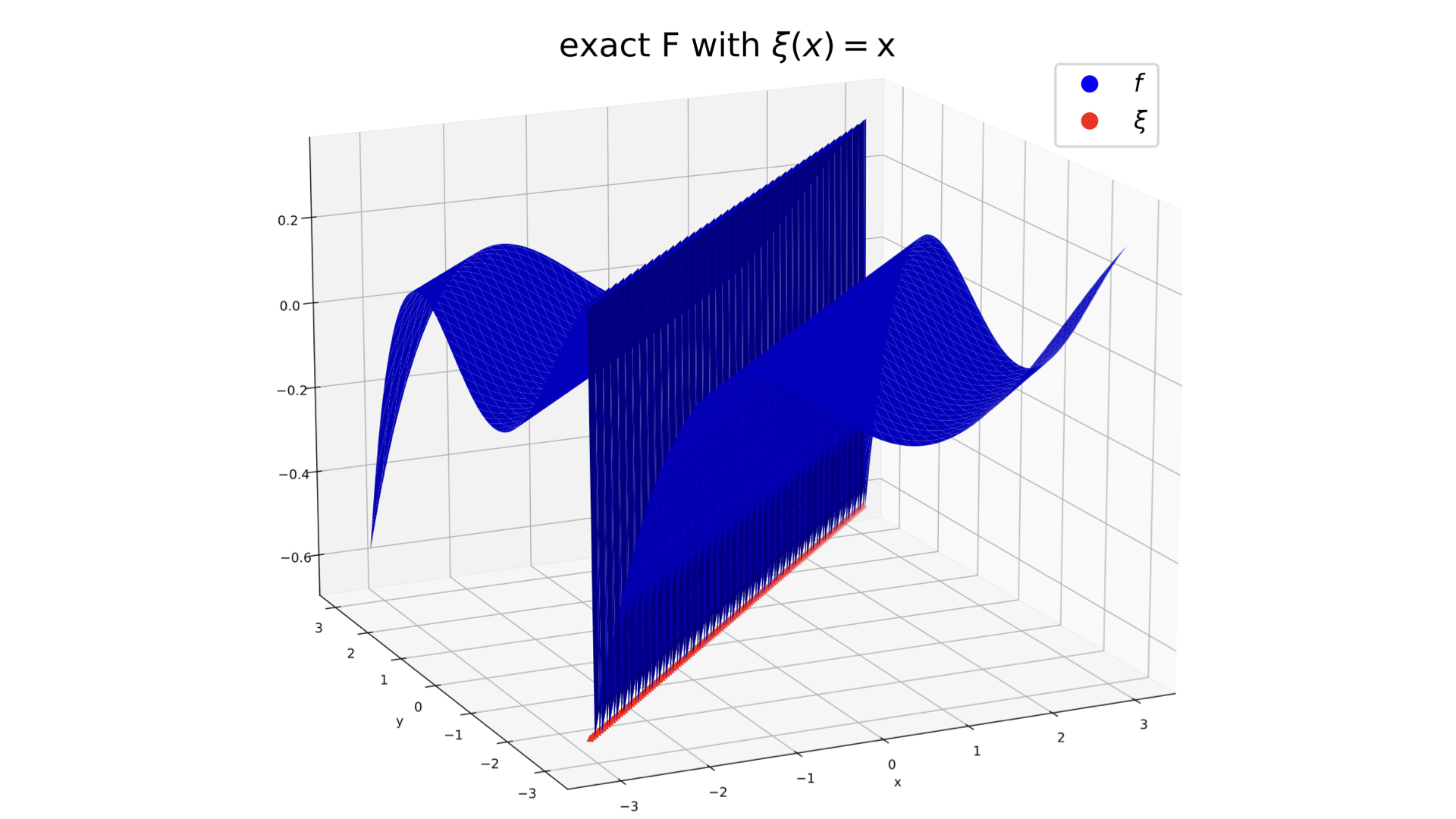}}}%
    \hfill
    \subfloat[\centering $\tilde{F}$ and $\tilde{\xi}$ \label{fig:approximate-F}]{{\includegraphics[width=0.49\linewidth]{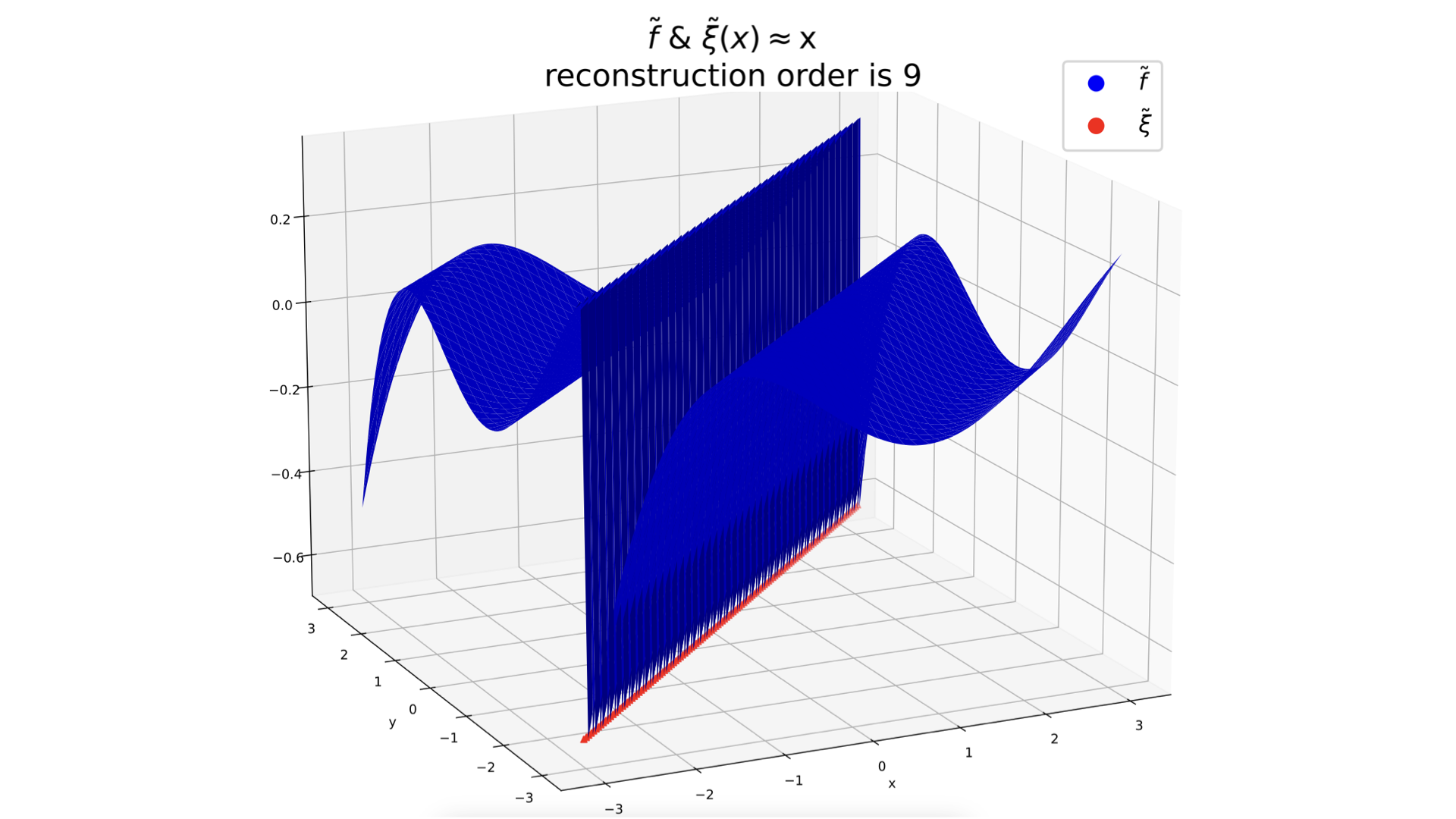}}}%
    \caption{Exact function with its discontinuity curve and approximated functions with its approximated discontinuity curve}
    \label{fig:f-and-xi}
\end{figure}

\begin{figure}[H]
    \centering
    \subfloat[\centering $\Delta \Fx$ vs $\Delta \mathcal{T}_x$ vs $N$ \label{fig:fx-tf-vs-moy}]{{
        \includegraphics[width=0.32\linewidth]{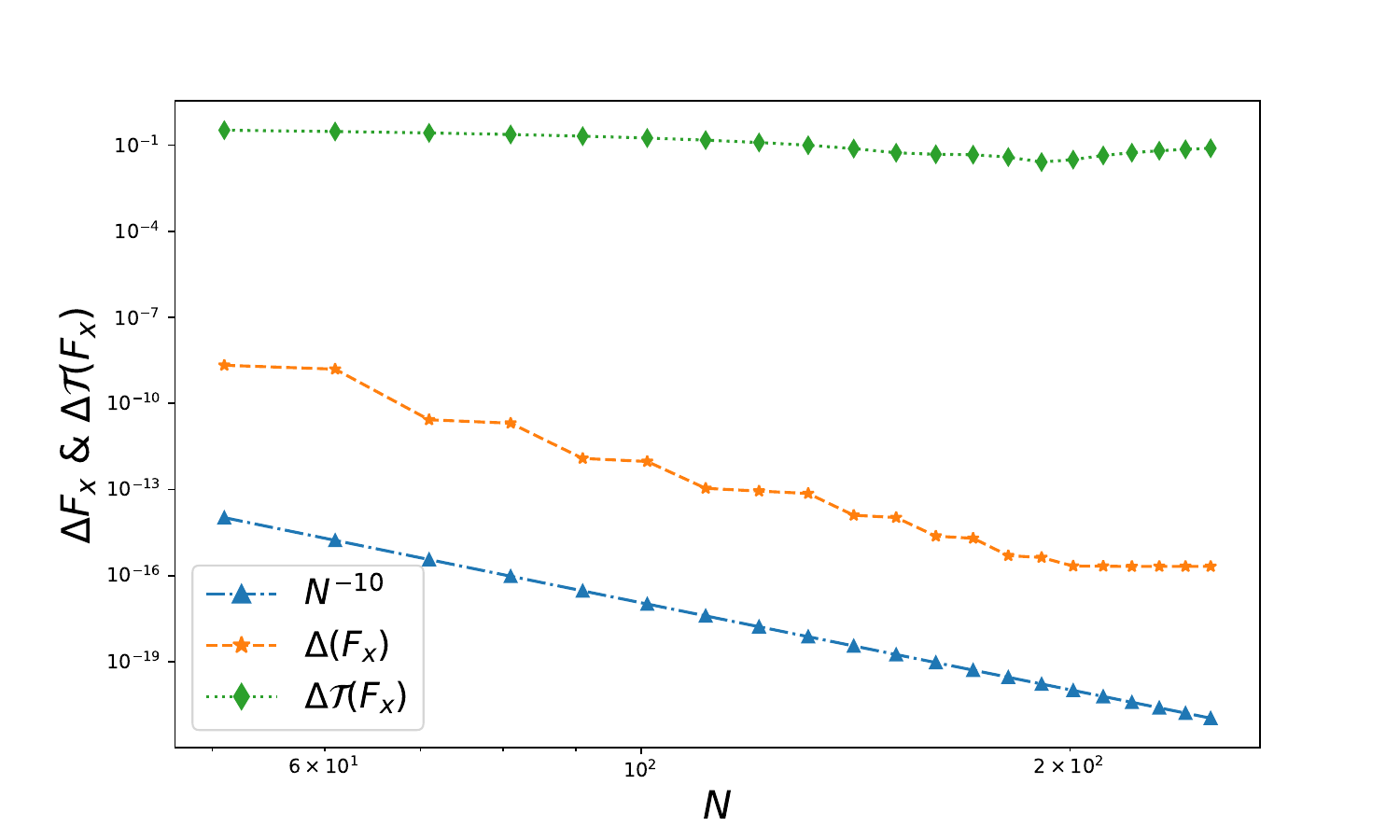}}}%
    \hfill
    \subfloat[\centering $\Delta \xix$ vs $N$ \label{fig:xi-vs-moy}]{{
        \includegraphics[width=0.32\linewidth]{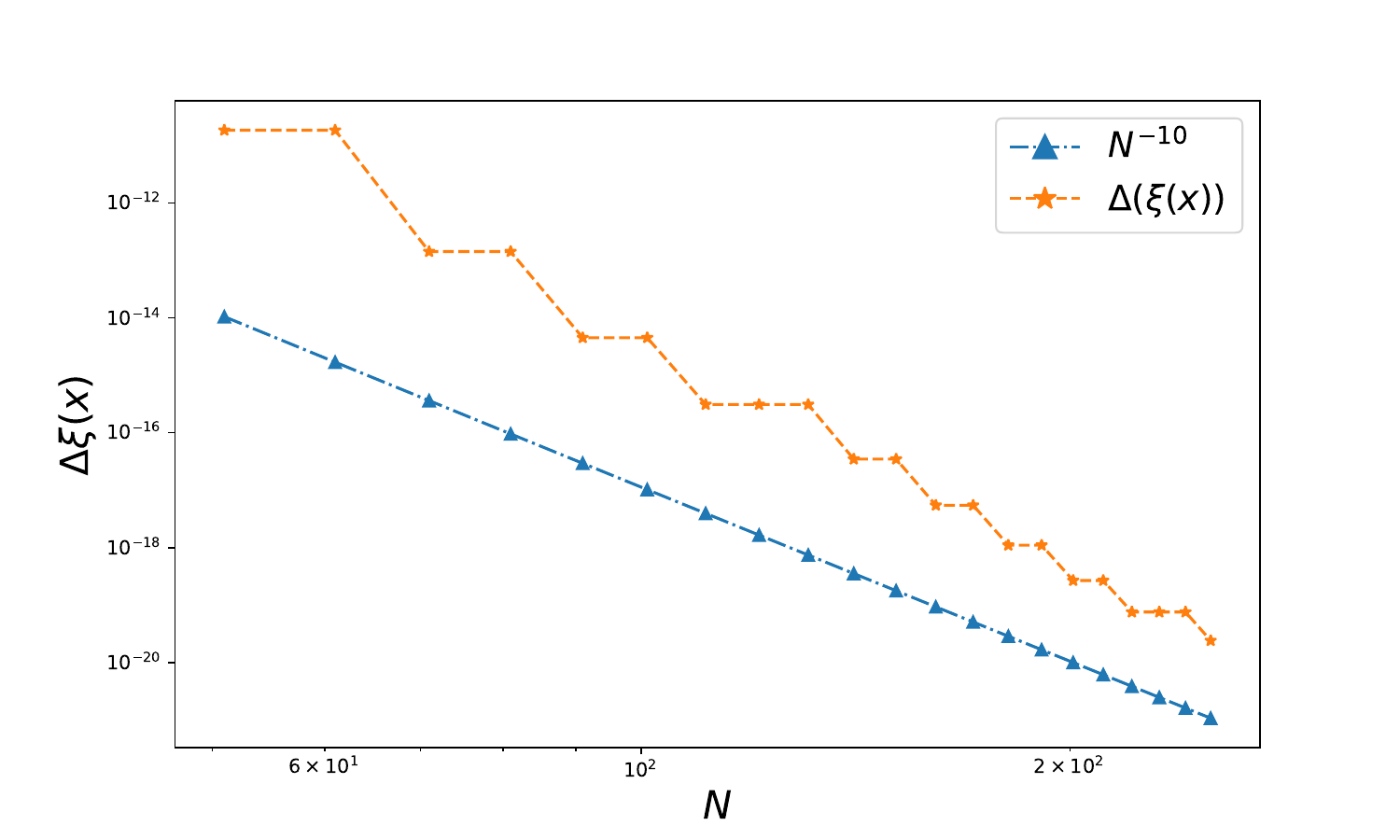}}}%
    \hfill
    \subfloat[\centering $\Delta A_l(x)$ vs $N$ for $l=0\ldots, 9$ \label{fig:al-vs-moy}]{{
        \includegraphics[width=0.32\linewidth]{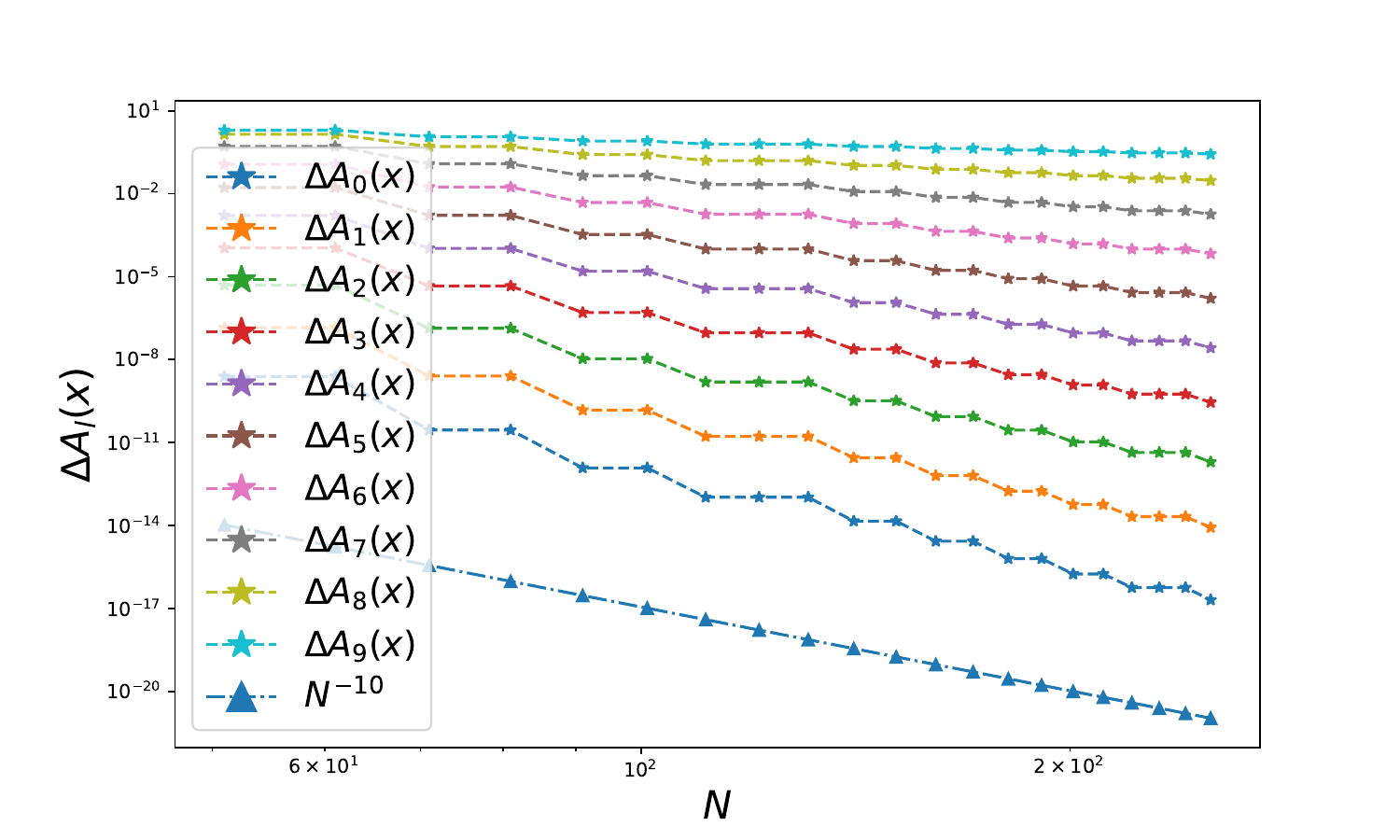}}}%
    \caption{Approximation errors computing by our algorithm for the function in \cref{fig:approximate-F}.}
    \label{fig:fx_xi_al-vs-moy}
\end{figure}

In~\cref{fig:fx-tf-vs-moy} we present the maximal error in approximating $F_x$ at $x=1.1$ using our method vs the maximal
error in approximating $F_x$ using a truncated Fourier sum.
In~\cref{fig:xi-vs-moy} we present the maximal error in approximating the jump location of $F_x$ at $x=1.1$ ($\tilde{\xi}(1.1)$)
using our method and in~\cref{fig:al-vs-moy} we present maximal error in approximating $\frac{d^i}{dx^i}(F_x)$ as $0\leq i \leq 9$
jump magnitudes at $x=1.1$.
All plots include a graph of $N^{-10}$ for comparison with the theoretical bounds presented in~\cref{thm:general-theorem-for-fx}.

The code implementing the algorithms described in~\cref{subsec:algorithm} and reproducing the results of this section is available at \url{https://github.com/mlevinov/algebraic-fourier-2d}.

\section{Discussion and future work}\label{sec:future-work}

In this paper we proposed a method for reconstructing a 2D function $F:\torus[2]\to\mathbb{R}$ from its Fourier coefficients. Our model of the function was restricted to only two continuity pieces. We have demonstrated algebraic convergence of the method, with rates commensurate with the smoothness of the pieces.  There exist several natural extensions of our work in multiple directions.

\begin{enumerate}
    \item In order to recover the entire function, one would need to perform interpolation from the approximated slices, while taking into account the discontinuity curve as well.
    \item Our method can be extended to reconstructing functions with multiple discontinuity pieces, supported on domains with smooth boundaries $\left\{ \Sigma_{1},\ldots,\Sigma
    _K \right\}$. First, the discontinuity structure can be approximated by a low-order method (as done in the one-dimensional case in~\cite{batenkov2012algebraic,batenkov2015complete}). Then, each piece can be further localized by applying a mollifier to create a region containing a single
    discontinuity curve. An advantage of this approach is that the boundary discontinuity will disappear, enabling to apply fast algorithms in the first stage.
    \item The algebraic structure of the method can be extended to handle other orthonormal basis, such as Chebyshev polynomials. Related work in this direction can be found in \cite{eckhoff1993accurate,eckhoff1995accurate}. This would allow us to handle non-periodic functions with variable discontinuities more efficiently. 
    \item The fundamental relationship between $\widehat{F}$ and $\psiOy$ can be naturally extended to multiple dimensions by recursion. While this introduces additional symbolic complexity, we believe these challenges can be efficiently handled by modern computer algebra and automatic differentiation techniques.  
\end{enumerate}

\printbibliography

@book{stewart2016calculus,
  title={Calculus: early transcendentals 8th edition},
  author={Stewart, James},
  year={2016},
  publisher={Cengage Learning}
}

@article{wasserman2015,
  title = {Image {{Reconstruction}} from {{Fourier Data Using Sparsity}} of {{Edges}}},
  author = {Wasserman, Gabriel and Archibald, Rick and Gelb, Anne},
  date = {2015-11-01},
  journal = {Journal of Scientific Computing},
  shortjournal = {J Sci Comput},
  volume = {65},
  number = {2},
  pages = {533--552},
  issn = {0885-7474, 1573-7691},
  doi = {10.1007/s10915-014-9973-3}
}

@book{daubechies1992,
  title = {Ten {{Lectures}} on {{Wavelets}}},
  author = {Daubechies, I.},
  date = {1992-01-01},
  series = {{{CBMS-NSF Regional Conference Series}} in {{Applied Mathematics}}},
  publisher = {{Society for Industrial and Applied Mathematics}},
  url = {http://epubs.siam.org/doi/book/10.1137/1.9781611970104},
  urldate = {2016-02-01},
  isbn = {978-0-89871-274-2},
  pagetotal = {369}
}

@article{adcock2017a,
  title = {Weighted Frames of Exponentials and Stable Recovery of Multidimensional Functions from Nonuniform {{Fourier}} Samples},
  author = {Adcock, Ben and Gataric, Milana and Hansen, Anders C.},
  date = {2017-05-01},
  journal = {Applied and Computational Harmonic Analysis},
  shortjournal = {Applied and Computational Harmonic Analysis},
  volume = {42},
  number = {3},
  pages = {508--535},
  issn = {1063-5203},
  doi = {10.1016/j.acha.2015.09.006}
}

@article{amir2018,
  title = {High Order Approximation to Non-Smooth Multivariate Functions},
  author = {Amir, Anat and Levin, David},
  year = {2018},
  journal = {Computer Aided Geometric Design},
  shortjournal = {Computer Aided Geometric Design},
  volume = {63},
  pages = {31--65},
  issn = {0167-8396},
  doi = {10.1016/j.cagd.2018.02.004}
}

@article{lipman2010,
  title = {Approximating Piecewise-Smooth Functions},
  author = {Lipman, Yaron and Levin, David},
  year = {2010},
  journal = {IMA Journal of Numerical Analysis},
  volume = {30},
  number = {4},
  pages = {1159--1183}
}

@article{amat2023,
  title = {Global and Explicit Approximation of Piecewise-Smooth Two-Dimensional Functions from Cell-Average Data},
  author = {Amat, Sergio and Levin, David and Ruiz-Alvárez, Juan and Yáñez, Dionisio F},
  year = {2023},
  journal = {IMA Journal of Numerical Analysis},
  shortjournal = {IMA Journal of Numerical Analysis},
  volume = {43},
  number = {4},
  pages = {2299--2319},
  issn = {0272-4979},
  doi = {10.1093/imanum/drac042}
}

@article{donoho1999,
  title = {Wedgelets: Nearly Minimax Estimation of Edges},
  shorttitle = {Wedgelets},
  author = {Donoho, David L.},
  date = {1999-06},
  journal = {The Annals of Statistics},
  volume = {27},
  number = {3},
  pages = {859--897},
  publisher = {{Institute of Mathematical Statistics}},
  issn = {0090-5364, 2168-8966},
  doi = {10.1214/aos/1018031261}
}

@article{blakely2007,
  title = {An automated method for recovering piecewise smooth functions on spheres free from Gibbs oscillations},
  author = {Blakely, Chris and Gelb, Anne and Navarra, Antonio},
  date = {2007-09},
  year = {2007},
  journal = {Sampling Theory in Signal and Image Processing},
  volume = {6},
  number = {3},
  pages = {323--346},
  issn = {1530-6429},
  url = {https://asu.pure.elsevier.com/en/publications/an-automated-method-for-recovering-piecewise-smooth-functions-on-},
  urldate = {2017-03-03}
}

@article{gelb2016,
  title = {Detecting {{Edges}} from {{Non-uniform Fourier Data Using Fourier Frames}}},
  author = {Gelb, Anne and Song, Guohui},
  date = {2016-11-18},
  year = {2016},
  journal = {Journal of Scientific Computing},
  shortjournal = {J Sci Comput},
  pages = {1--22},
  issn = {0885-7474, 1573-7691},
  doi = {10.1007/s10915-016-0320-8}
}

@article{levin2020,
  title = {Reconstruction of {{Piecewise Smooth Multivariate Functions}} from {{Fourier Data}}},
  author = {Levin, David},
  date = {2020-07-24},
  year = {2020},
  journal = {Axioms},
  volume = {9},
  number = {3},
  pages = {88},
  publisher = {{MDPI}},
  issn = {20751680},
  doi = {10.3390/axioms9030088}
}

@article{gibbs1898fourier,
  title={Fourier's series},
  author={Gibbs, J Willard},
  journal={Nature},
  volume={59},
  number={1522},
  pages={200--200},
  year={1898},
  publisher={Nature Publishing Group UK London}
}

@book{boole1872treatise,
  title={A treatise on the calculus of finite differences},
  author={Boole, George and Moulton, John Fletcher},
  year={1872},
  publisher={Macmillan and company}
}

@article{batenkov2012algebraic,
  title   = {Algebraic Fourier reconstruction of piecewise smooth functions},
  author  = {Batenkov, Dmitry and Yomdin, Yosef},
  journal = {Mathematics of Computation},
  volume  = {81},
  number  = {277},
  pages   = {277--318},
  year    = {2012}
}

@article{batenkov2015complete,
  title   = {Complete algebraic reconstruction of piecewise-smooth functions from Fourier data},
  author  = {Batenkov, Dmitry},
  journal = {Mathematics of Computation},
  volume  = {84},
  number  = {295},
  pages   = {2329--2350},
  year    = {2015}
}

@article{eckhoff1993accurate,
  title   = {Accurate and efficient reconstruction of discontinuous functions from truncated series expansions},
  author  = {Eckhoff, Knut S},
  journal = {mathematics of computation},
  volume  = {61},
  number  = {204},
  pages   = {745--763},
  year    = {1993}
}

@article{eckhoff1995accurate,
  title   = {Accurate reconstructions of functions of finite regularity from truncated Fourier series expansions},
  author  = {Eckhoff, Knut S},
  journal = {Mathematics of Computation},
  volume  = {64},
  number  = {210},
  pages   = {671--690},
  year    = {1995}
}

@article{eckhoff1998high,
  title   = {On a high order numerical method for functions with singularities},
  author  = {Eckhoff, Knut},
  journal = {Mathematics of Computation},
  volume  = {67},
  number  = {223},
  pages   = {1063--1087},
  year    = {1998}
}

@article{fubini1907sugli,
  title   = {Sugli integrali multipli},
  author  = {Fubini, Guido},
  journal = {Rend. Acc. Naz. Lincei},
  volume  = {16},
  pages   = {608--614},
  year    = {1907}
}

@book{gottlieb1977numerical,
  title     = {Numerical analysis of spectral methods: theory and applications},
  author    = {Gottlieb, David and Orszag, Steven A},
  year      = {1977},
  publisher = {SIAM}
}

@book{olver1993applications,
  title     = {Applications of Lie groups to differential equations},
  author    = {Olver, Peter J},
  volume    = {107},
  year      = {1993},
  publisher = {Springer Science \& Business Media}
}

@book{olver2010nist,
  title     = {NIST handbook of mathematical functions hardback and CD-ROM},
  author    = {Olver, Frank WJ and Lozier, Daniel W and Boisvert, Ronald F and Clark, Charles W},
  year      = {2010},
  publisher = {Cambridge university press}
}

@book{zygmund2002trigonometric,
  title={Trigonometric series},
  author={Zygmund, Antoni},
  volume={1},
  year={2002},
  publisher={Cambridge university press}
}

@article{sun2004identities,
  title   = {Identities concerning Bernoulli and Euler polynomials},
  author  = {Sun, Zhi-Wei and Pan, Hao},
  journal = {arXiv preprint math/0409035},
  year    = {2004}
}

@article{veraart2016gibbs,
  title={Gibbs ringing in diffusion MRI},
  author={Veraart, Jelle and Fieremans, Els and Jelescu, Ileana O and Knoll, Florian and Novikov, Dmitry S},
  journal={Magnetic resonance in medicine},
  volume={76},
  number={1},
  pages={301--314},
  year={2016},
  publisher={Wiley Online Library}
}

@book{mallat1999wavelet,
  title={A wavelet tour of signal processing},
  author={Mallat, St{\'e}phane},
  year={1999},
  publisher={Elsevier}
}

@article{mhaskar2013applications,
  title={Applications of classical approximation theory to periodic basis function networks and computational harmonic analysis},
  author={Mhaskar, Hrushikesh N and Nevai, Paul and Shvarts, Eugene},
  journal={Bulletin of Mathematical Sciences},
  volume={3},
  number={3},
  pages={485--549},
  year={2013},
  publisher={Springer}
}

@article{candes2004new,
  title={New tight frames of curvelets and optimal representations of objects with piecewise C2 singularities},
  author={Cand{\`e}s, Emmanuel J and Donoho, David L},
  journal={Communications on Pure and Applied Mathematics: A Journal Issued by the Courant Institute of Mathematical Sciences},
  volume={57},
  number={2},
  pages={219--266},
  year={2004},
  publisher={Wiley Online Library}
}

@article{guo2012optimally,
  title={Optimally Sparse Representations of 3D Data with C\^{}2 Surface Singularities Using Parseval Frames of Shearlets},
  author={Guo, Kanghui and Labate, Demetrio},
  journal={SIAM Journal on Mathematical Analysis},
  volume={44},
  number={2},
  pages={851--886},
  year={2012},
  publisher={SIAM}
}

@online{cohen2024,
  title = {High Order Recovery of Geometric Interfaces from Cell-Average Data},
  author = {Cohen, Albert and Mula, Olga and Somacal, Agustín},
  date = {2024-02-01},
  eprint = {2402.00946},
  eprinttype = {arxiv},
  eprintclass = {cs, math},
  doi = {10.48550/arXiv.2402.00946},
  pubstate = {preprint}
}

@book{kutyniok2012,
  title = {Shearlets},
  editor = {Kutyniok, Gitta and Labate, Demetrio},
  date = {2012},
  publisher = {Birkhäuser Boston},
  location = {Boston},
  doi = {10.1007/978-0-8176-8316-0},
  isbn = {978-0-8176-8315-3 978-0-8176-8316-0}
}

@incollection{batenkov2013,
  title = {Reconstruction of Planar Domains from Partial Integral Measurements},
  booktitle = {Complex Analysis and Dynamical Systems {{V}}},
  author = {Batenkov, D. and Golubyatnikov, V. and Yomdin, Y.},
  year = {2013},
  series = {Contemp. {{Math}}.},
  volume = {591},
  pages = {51--66},
  publisher = {Amer. Math. Soc., Providence, RI},
  urldate = {2017-12-23},
  mrnumber = {3155677}
}

@article{akinshin2021,
  title = {Geometry of Error Amplification in Solving the {{Prony}} System with Near-Colliding Nodes},
  author = {Akinshin, Andrey and Goldman, Gil and Yomdin, Yosef},
  year = {2021},
  journal = {Mathematics of Computation},
  volume = {90},
  number = {327},
  pages = {267--302},
  issn = {0025-5718, 1088-6842},
  doi = {10.1090/mcom/3571},
  urldate = {2020-12-20},
  langid = {english}
}

@article{batenkov2013c,
  title = {On the Accuracy of Solving Confluent {{Prony}} Systems},
  author = {Batenkov, D. and Yomdin, Y.},
  year = {2013},
  journal = {SIAM J. Appl. Math.},
  volume = {73},
  number = {1},
  pages = {134--154},
  doi = {10.1137/110836584}
}

@article{batenkov2013b,
  title = {Decimated Generalized {{Prony}} Systems},
  author = {Batenkov, Dmitry},
  year = {2013},
  month = aug,
  journal = {arXiv:1308.0753 [math]},
  eprint = {1308.0753},
  primaryclass = {math},
  urldate = {2014-06-01},
  archiveprefix = {arXiv}
}

@article{batenkov2014b,
  title = {Geometry and {{Singularities}} of the {{Prony}} Mapping},
  author = {Batenkov, Dmitry and Yomdin, Yosef},
  year = {2014},
  journal = {Journal of Singularities},
  volume = {10},
  pages = {1--25},
  issn = {19492006},
  doi = {10.5427/jsing.2014.10a},
  urldate = {2014-12-25}
}

@article{batenkov2017c,
  title = {Accurate Solution of Near-Colliding {{Prony}} Systems via Decimation and Homotopy Continuation},
  author = {Batenkov, Dmitry},
  year = {2017},
  month = jun,
  journal = {Theoretical Computer Science},
  series = {Symbolic {{Numeric Computation}}},
  volume = {681},
  pages = {27--40},
  issn = {0304-3975},
  doi = {10.1016/j.tcs.2017.03.026},
  urldate = {2017-06-13}
}

@article{batenkov2018,
  title = {Stability and Super-Resolution of Generalized Spike Recovery},
  author = {Batenkov, Dmitry},
  year = {2018},
  month = sep,
  journal = {Applied and Computational Harmonic Analysis},
  volume = {45},
  number = {2},
  pages = {299--323},
  issn = {1063-5203},
  doi = {10.1016/j.acha.2016.09.004},
  urldate = {2018-07-10}
}

@article{batenkov2020,
  title = {Conditioning of {{Partial Nonuniform Fourier Matrices}} with {{Clustered Nodes}}},
  author = {Batenkov, Dmitry and Demanet, Laurent and Goldman, Gil and Yomdin, Yosef},
  year = {2020},
  month = jan,
  journal = {SIAM Journal on Matrix Analysis and Applications},
  volume = {44},
  number = {1},
  pages = {199--220},
  issn = {0895-4798},
  doi = {10/ggjwzb},
  urldate = {2020-02-02}
}

@article{batenkov2021,
  title = {The Spectral Properties of {{Vandermonde}} Matrices with Clustered Nodes},
  author = {Batenkov, Dmitry and Diederichs, Benedikt and Goldman, Gil and Yomdin, Yosef},
  year = {2021},
  month = jan,
  journal = {Linear Algebra and its Applications},
  volume = {609},
  pages = {37--72},
  issn = {0024-3795},
  doi = {10.1016/j.laa.2020.08.034},
  urldate = {2020-09-07},
  langid = {english}
}

@article{batenkov2021b,
  title = {Super-Resolution of near-Colliding Point Sources},
  author = {Batenkov, Dmitry and Goldman, Gil and Yomdin, Yosef},
  year = {2021},
  month = jun,
  journal = {Information and Inference: A Journal of the IMA},
  volume = {10},
  number = {2},
  pages = {515--572},
  publisher = {Oxford Academic},
  doi = {10.1093/imaiai/iaaa005},
  urldate = {2021-06-29},
  langid = {english}
}

@article{batenkov2023,
  title = {Super-Resolution of Generalized Spikes and Spectra of Confluent {{Vandermonde}} Matrices},
  author = {Batenkov, Dmitry and Diab, Nuha},
  year = {2023},
  month = jul,
  journal = {Applied and Computational Harmonic Analysis},
  volume = {65},
  pages = {181--208},
  issn = {1063-5203},
  doi = {10.1016/j.acha.2023.03.002},
  urldate = {2023-03-20},
  langid = {english}
}

@article{katz2023,
  title = {Decimated {{Prony}}'s {{Method}} for {{Stable Super-Resolution}}},
  author = {Katz, Rami and Diab, Nuha and Batenkov, Dmitry},
  year = {2023},
  journal = {IEEE Signal Processing Letters},
  volume = {30},
  pages = {1467--1471},
  issn = {1558-2361},
  doi = {10.1109/LSP.2023.3324553},
  urldate = {2023-10-25}
}

@article{katz2024b,
  title = {On the Accuracy of {{Prony}}'s Method for Recovery of Exponential Sums with Closely Spaced Exponents},
  author = {Katz, Rami and Diab, Nuha and Batenkov, Dmitry},
  year = {2024},
  month = nov,
  journal = {Applied and Computational Harmonic Analysis},
  volume = {73},
  pages = {101687},
  issn = {1063-5203},
  doi = {10.1016/j.acha.2024.101687},
  urldate = {2024-07-17}
}

@article{katz2024c,
  title = {Data-Driven {{Delay Estimation}} in {{Reaction-Diffusion Systems}} via {{Exponential Fitting}}},
  author = {Katz, Rami and Giordano, Giulia and Batenkov, Dmitry},
  year = {2024},
  month = jan,
  journal = {IFAC-PapersOnLine},
  series = {18th {{IFAC Workshop}} on {{Time Delay Systems TDS}} 2024},
  volume = {58},
  number = {27},
  pages = {102--107},
  issn = {2405-8963},
  doi = {10.1016/j.ifacol.2024.10.307},
  urldate = {2024-11-28}
}

@article{mhaskar2019,
  title = {Super-Resolution Meets Machine Learning: Approximation of Measures},
  shorttitle = {Super-Resolution Meets Machine Learning},
  author = {Mhaskar, H. N.},
  year = {2019},
  journal = {Journal of Fourier Analysis and Applications},
  pages = {1--19}
}

@article{poon2019,
  title = {{{MultiDimensional Sparse Super-Resolution}}},
  author = {Poon, Clarice and Peyr{\'e}, Gabriel},
  year = {2019},
  month = jan,
  journal = {SIAM Journal on Mathematical Analysis},
  volume = {51},
  number = {1},
  pages = {1--44},
  issn = {0036-1410, 1095-7154},
  doi = {10.1137/17M1147822},
  urldate = {2019-04-16},
  langid = {english}
}

@misc{diab2024,
  title = {Spectral {{Properties}} of {{Infinitely Smooth Kernel Matrices}} in the {{Single Cluster Limit}}, with {{Applications}} to {{Multivariate Super-Resolution}}},
  author = {Diab, Nuha and Batenkov, Dmitry},
  year = {2024},
  month = jul,
  number = {arXiv:2407.10600},
  eprint = {2407.10600},
  primaryclass = {cs, math},
  publisher = {arXiv},
  urldate = {2024-07-16},
  archiveprefix = {arXiv}
}

@article{candes2014,
  title = {Towards a {{Mathematical Theory}} of {{Super-resolution}}},
  author = {Cand{\`e}s, Emmanuel J. and {Fernandez-Granda}, Carlos},
  year = {2014},
  month = jun,
  journal = {Communications on Pure and Applied Mathematics},
  volume = {67},
  number = {6},
  pages = {906--956},
  issn = {1097-0312},
  doi = {10.1002/cpa.21455},
  urldate = {2014-05-18},
  copyright = {{\copyright} 2014 Wiley Periodicals, Inc.},
  langid = {english}
}

@article{cuyt2018,
  title = {Multivariate Exponential Analysis from the Minimal Number of Samples},
  author = {Cuyt, Annie and Lee, Wen-shin},
  year = {2018},
  month = aug,
  journal = {Advances in Computational Mathematics},
  volume = {44},
  number = {4},
  pages = {987--1002},
  issn = {1572-9044},
  doi = {10.1007/s10444-017-9570-8},
  urldate = {2019-07-04},
  langid = {english}
}

@article{diederichs2022,
  title = {How Many {{Fourier}} Coefficients Are Needed?},
  author = {Diederichs, Benedikt and Kolountzakis, Mihail N. and Papageorgiou, Effie},
  year = {2022},
  month = oct,
  journal = {Monatshefte f{\"u}r Mathematik},
  issn = {0026-9255, 1436-5081},
  doi = {10.1007/s00605-022-01792-0},
  urldate = {2022-11-08},
  langid = {english}
}

@article{filbir2012,
  title = {On the Problem of Parameter Estimation in Exponential Sums},
  author = {Filbir, F. and Mhaskar, H. N. and Prestin, J.},
  year = {2012},
  journal = {Constructive Approximation},
  volume = {35},
  number = {3},
  pages = {323--343},
  urldate = {2015-05-10}
}

@article{mhaskar2000a,
  title = {On the Detection of Singularities of a Periodic Function},
  author = {Mhaskar, Hrushikesh Narhar and Prestin, J{\"u}rgen},
  year = {2000},
  journal = {Advances in Computational Mathematics},
  volume = {12},
  number = {2-3},
  pages = {95--131},
  urldate = {2015-05-24}
}

@article{schober2021,
  title = {Edge Detection with Trigonometric Polynomial Shearlets},
  author = {Schober, Kevin and Prestin, J{\"u}rgen and Stasyuk, Serhii A.},
  year = {2021},
  month = feb,
  journal = {Advances in Computational Mathematics},
  volume = {47},
  number = {1},
  pages = {17},
  issn = {1019-7168, 1572-9044},
  doi = {10.1007/s10444-020-09838-3},
  urldate = {2022-11-21},
  langid = {english}
}

@article{sauer2017,
  title = {Prony's Method in Several Variables},
  author = {Sauer, Tomas},
  year = {2017},
  month = jun,
  journal = {Numerische Mathematik},
  volume = {136},
  number = {2},
  pages = {411--438},
  issn = {0029-599X, 0945-3245},
  doi = {10.1007/s00211-016-0844-8},
  urldate = {2019-03-04},
  langid = {english}
}

@article{sauer2018,
  title = {Prony's Method in Several Variables: {{Symbolic}} Solutions by Universal Interpolation},
  shorttitle = {Prony's Method in Several Variables},
  author = {Sauer, Tomas},
  year = {2018},
  month = jan,
  journal = {Journal of Symbolic Computation},
  volume = {84},
  pages = {95--112},
  issn = {07477171},
  doi = {10.1016/j.jsc.2017.03.006},
  urldate = {2019-03-04},
  langid = {english}
}

@article{chen2014d,
  title = {Recovering Exponential Accuracy from Collocation Point Values of Smooth Functions with End-Point Singularities},
  author = {Chen, Zheng and Shu, Chi-Wang},
  year = {2014},
  month = aug,
  journal = {Journal of Computational and Applied Mathematics},
  volume = {265},
  pages = {83--95},
  issn = {03770427},
  doi = {10.1016/j.cam.2013.09.029},
  urldate = {2024-12-15},
  langid = {english}
}

@article{gottlieb1996,
  title = {On the {{Gibbs}} Phenomenon {{III}}: Recovering Exponential Accuracy in a Sub-Interval from a Spectral Partial Sum of a Piecewise Analytic Function},
  author = {Gottlieb, D. and Shu, C. W.},
  year = {1996},
  journal = {SIAM Journal on Numerical Analysis},
  pages = {280--290}
}

@article{adcock2014b,
  title = {On the {{Numerical Stability}} of {{Fourier Extensions}}},
  author = {Adcock, Ben and Huybrechs, Daan and {Mart{\'i}n-Vaquero}, Jes{\'u}s},
  year = {2014},
  month = aug,
  journal = {Foundations of Computational Mathematics},
  volume = {14},
  number = {4},
  pages = {635--687},
  issn = {1615-3383},
  doi = {10.1007/s10208-013-9158-8},
  urldate = {2021-05-16},
  langid = {english}
}

@article{kvernadze2004,
  title = {Approximating the Jump Discontinuities of a Function by Its {{Fourier-Jacobi}} Coefficients},
  author = {Kvernadze, G.},
  year = {2004},
  journal = {Mathematics of Computation},
  volume = {73},
  number = {246},
  pages = {731--752}
}

@article{kvernadze2010,
  title = {Approximation of the Discontinuities of a Function by Its Classical Orthogonal Polynomial {{Fourier}} Coefficients},
  author = {Kvernadze, G.},
  year = {2010},
  journal = {Mathematics of Computation},
  volume = {79},
  pages = {2265--2285}
}

@article{barkhudaryan2007,
  title = {Asymptotic Behavior of {{Eckhoff}}'s Method for {{Fourier}} Series Convergence Acceleration},
  author = {Barkhudaryan, A. and Barkhudaryan, R. and Poghosyan, A.},
  year = {2007},
  journal = {Analysis in Theory and Applications},
  volume = {23},
  number = {3},
  pages = {228--242}
}

@article{poghosyan2021,
  title = {On the Convergence of the Quasi-Periodic Approximations on a Finite Interval},
  author = {Poghosyan, Arnak and Poghosyan, Lusine and Barkhudaryan, Rafayel},
  year = {2021},
  month = dec,
  journal = {Armenian Journal of Mathematics},
  volume = {13},
  number = {10},
  pages = {1--44},
  issn = {1829-1163},
  doi = {10.52737/18291163-2021.13.10-1-44},
  urldate = {2024-08-11},
  copyright = {Copyright (c) 2021 Armenian Journal of Mathematics},
  langid = {english}
}

@article{kunis2016,
  title = {A Multivariate Generalization of {{Prony}}'s Method},
  author = {Kunis, Stefan and Peter, Thomas and R{\"o}mer, Tim and {von der Ohe}, Ulrich},
  year = {2016},
  month = feb,
  journal = {Linear Algebra and its Applications},
  volume = {490},
  pages = {31--47},
  issn = {0024-3795},
  doi = {10/f797mv},
  urldate = {2019-10-28},
  langid = {english}
}

@article{gelb2007a,
  title = {The Resolution of the {{Gibbs}} Phenomenon for {{Fourier}} Spectral Methods},
  author = {Gelb, Anne and Gottlieb, Sigal},
  year = {2007},
  journal = {Advances in The Gibbs Phenomenon. Sampling Publishing, Potsdam, New York},
  urldate = {2014-02-26}
}

@book{gottlieb1977,
  title = {Numerical {{Analysis}} of {{Spectral Methods}}: {{Theory}} and {{Applications}}},
  author = {Gottlieb, David and Orszag, Steven A},
  year = {1977},
  volume = {26},
  publisher = {SIAM}
}

@article{gottlieb1997,
  title = {On the {{Gibbs}} Phenomenon and Its Resolution},
  author = {Gottlieb, D. and Shu, C. W.},
  year = {1997},
  journal = {SIAM Review},
  pages = {644--668}
}

@article{tadmor2007,
  title = {Filters, Mollifiers and the Computation of the {{Gibbs}} Phenomenon},
  author = {Tadmor, E.},
  year = {2007},
  journal = {Acta Numerica},
  volume = {16},
  pages = {305--378}
}

\appendix

\section{APPENDIX}\label{A}

\subsection{Auxiliary results}\label{subsec:auxiliary-results}
%
We first provide several auxiliary results.
\begin{lemma}\label{lem:derivative-of-gc-gs}
    Let $\oyInZ$, and $x\in\torus[1]$, $\xix $ as defined in~\cref{description-of-F-and-xi} and $A_{m}(x)$ as defined in~\cref{Fx-jump-mag-def}.
    Then by using the following notations
    \begin{itemize}
        \item $\xi_{i}(x) \coloneq \ddxk[i]\xix$
        \item $f_{c}(x) \coloneq \cos(\oy\xix)$
        \item $f_{s}(x) \coloneq \sin(\oy\xix)$
        \item $g_{c}(x) \coloneq f_{c}(x)\cdot\xi_{1}(x)$
        \item $g_{s}(x) \coloneq f_{s}(x)\cdot\xi_{1}(x)$
    \end{itemize}
    we have:
    \begin{gather*}
        \ddxk[i] g_{c}(x) = \ddxk[i]\left( f_{c}(x)\cdot\xi_{1}(x) \right) = \sum_{j=0}^{i}{a_{j}\cdot f_{c}^{(j)}(x)\cdot\xi_{i+1-j}}(x)\\
        \ddxk[i] g_{s}(x) = \ddxk[i]\left( f_{s}(x)\cdot\xi_{1}(x) \right) = \sum_{j=0}^{i}{b_{j}\cdot f_{s}^{(j)}(x)\cdot\xi_{i+1-j}(x)}
    \end{gather*}
    where $a_{j}, b_{j} \in\mathbb{R}$ for each $1\leq j\leq i$.
    \end{lemma}
    \begin{proof}
        We will provide a proof only for $f_{c}(x)$, the proof for $f_{s}(x)$ is identical.
        First, by using the general Leibnitz rule ~\cite[Section 3.3]{stewart2016calculus} we have
        \begin{align}
            \begin{split}\label{eq:my_lebnitz}
                \ddxk[l]\left( g_{c}(x)\cdot A_{m}(x) \right)&=\sum_{i=0}^{l}{a_{i}\cdot g_{c}^{(i)}(x)\cdot A_{m}^{(l-i)}(x)}\\
                \ddxk[l]\left( g_{s}(x)\cdot A_{m}(x) \right)&=\sum_{i=0}^{l}{a_{i}\cdot g_{s}^{(i)}(x)\cdot A_{m}^{(l-i)}(x)}
            \end{split}
        \end{align}
        where $a_{i}\in\mathbb{R}$ for each $1\leq i\leq l$.
        Now, for $i=1$ we have
        \begin{equation*}
            g_{c}^{(1)}(x) = \left( f_{c}(x) \cdot \xi_{1}(x) \right)^{(1)} = f_{c}^{(1)}(x)\cdot f_{c}(x)\cdot\xi_{2}(x) =
            \sum_{j=0}^{1}{f_{c}^{(j)}(x)\cdot\xi_{2-j}(x)}.
        \end{equation*}
        Assuming that
        \begin{equation*}
            \ddxk[i] g_{c}(x) = \ddxk[i]\left( f_{c}(x)\cdot\xi_{1}(x) \right) = \sum_{j=0}^{i}{a_{j}\cdot f_{c}^{(j)}(x)\cdot\xi_{i+1-j}(x)},
        \end{equation*}
        we continue to
        \begin{gather*}
            \ddxk[i+1] g_{c}(x) = \ddx\left( \ddxk[i] g_{c}(x) \right) = \ddx\left( \sum_{j=0}^{i}{a_{j}\cdot f_{c}^{(j)}(x)\cdot\xi_{i+1-j}(x)} \right) =\\
            \sum_{j=0}^{i}{a_{j}\cdot\ddx\left( f_{c}^{(j)}(x)\cdot\xi_{i+1-j}(x) \right)} =
            \sum_{j=0}^{i}{a_{j}\cdot f_{c}^{(j+1)}(x)\cdot\xi_{i+1-j}(x)} + \sum_{j=0}^{i}{a_{j}\cdot f_{c}^{(j)}(x)\cdot\xi_{i+2-j}(x)}.
        \end{gather*}
        For the first sum, we denote $k=j+1$ and get
        \begin{equation*}
            \sum_{j=0}^{i}{a_{j}\cdot f_{c}^{(j+1)}(x)\cdot\xi_{i+1-j}(x)} = \sum_{k=1}^{i+1}{a_{k-1}\cdot f_{c}^{(k)}(x)\cdot\xi_{i+2-k}(x)},
        \end{equation*}
        so reverting back to $j$ we get
        \begin{gather*}
            \sum_{j=0}^{i}{f_{c}^{(j+1)}(x)\cdot\xi_{i+1-j}(x)} + \sum_{j=0}^{i}{f_{c}^{(j)}(x)\cdot\xi_{i+2-j}(x)} = \\
            \sum_{j=1}^{i+1}{a_{j-1}\cdot f_{c}^{(j)}(x)\cdot\xi_{i+2-j}(x)} + \sum_{j=0}^{i}{a_{j}\cdot f_{c}^{(j)}(x)\cdot\xi_{i+2-j}(x)} = \\
            a_{0}\cdot f_{c}^{(0)}(x)\cdot\xi_{i+2}(x) + \sum_{j=1}^{i}{a_{j}\cdot f_{c}^{(j)}(x)\cdot\xi_{i+2-j}(x)} +
            \sum_{j=1}^{i}{a_{j-1}\cdot f_{c}^{(j)}(x) \cdot \xi_{i+2-j}(x)} + a_{i}\cdot f_{c}^{(i+1)}(x) \cdot \xi_{1}(x)\\
            = a_{0}\cdot f_{c}^{(0)}(x) \cdot \xi_{i+2}(x) + \sum_{j=1}^{i}{\left( a_{j-1} + a_{j} \right) \cdot f_{c}^{(j)}(x) \cdot \xi_{i+2-j}(x)} +
            a_{i}\cdot f_{c}^{(i+1)}(x) \cdot \xi_{1}(x).
        \end{gather*}
        Denoting $c_{j} = a_{j-1} + a_{j}$ we get $c_{j} \in \mathbb{R}$ and
        \begin{gather*}
            \ddxk[i+1] g_{c}(x) = \sum_{j=0}^{i+1}{c_{j}\cdot f_{c}^{(j)}(x) \cdot \xi_{i+2-j}(x)}.\qedhere
        \end{gather*}
    \end{proof}
    \begin{lemma}\label{lem:fcs-derivatives}
        Let $x\in\torus[1]$, $\oyInZ$ and let $\xi_{i}(x)$, $f_c(x)$, $f_s(x)$ be as defined in~\cref{lem:derivative-of-gc-gs}. Furthermore:
        \begin{itemize}
            \item $\oyapm \coloneqq -\oy^{a}\text{ or } \oy^{a}$, where $a\in\mathbb{N}$.
            \item $\fcsk[k] \coloneqq \ddxk[k] \cosoyxi$ or $\ddxk[k]\sinoyxi$.
            \item $\gkma{m_l}{a_l} = \prod\limits_{j=1}^{k}{a_{l,j}\cdot\xi_{j}^{m_{l,j}}}$, where
            $a_{l,j}\in\mathbb{R}$ and $m_{l,j}\in\mathbb{N}$ as $1\leq j\leq k$.
        \end{itemize}
        Then there exists $1\leq s\in\mathbb{N}$ s.t.
        \begin{equation*}
            \ddxk[k]\fcs = \sum\limits_{l=1}^{s}{\oypm^{m_{l,0}}\cdot \fcs \cdot \gkma[k]{m_l}{a_l}}
        \end{equation*}
        where $m_{1,0}=\max\limits_{1\leq l\leq s}\left\{ m_{l,0} \right\}=k$ and
        $m_{1,1}=\max\limits_{1\leq l\leq s}\left\{ m_{l,j}\;\big|\; 1\leq j\leq k \right\}=k$.
    \end{lemma}
    \begin{proof}
        First we see that
        \begin{equation}\label{eq:basis}
            \ddx\fcs = \oypm\cdot\fcs\cdot\xi_{1}
        \end{equation}
        Now we assume that for $k\in\mathbb{N}$ s.t. $k\leq d$ we have
        \begin{equation}\label{eq:assumption}
            \ddxk[k]\fcs = \sum\limits_{l=1}^{s}{\oypm^{m_{l,0}}\cdot \fcs \cdot \gkma[k]{m_l}{a_l}},
        \end{equation}
        where $m_{1,0}=\max\limits_{1\leq l\leq s}\left\{ m_{l,0} \right\}=k$ and
        $m_{1,1}=\max\limits_{1\leq l\leq s}\left\{ m_{l,j}\;\big|\; 1\leq j\leq k \right\}=k$.  Presently we show that there exists
        $s_1\in\mathbb{N}$ s.t.
        $\ddxk[k+1]\fcs = \sum\limits_{l=1}^{s_1}{\oypm^{m_{l,0}}\cdot \fcs \cdot \gkma[k]{m_l}{a_l}}$:
        \begin{align*}
            \ddxk[k+1]\fcs &= \ddx\left( \ddxk[k]\fcs \right) \overset{\cref{eq:assumption}}{=}
            \ddx\left( \sum\limits_{l=1}^{s}{\oypm^{m_{l,0}}\cdot \fcs \cdot \gkma[k]{m_l}{a_l}} \right) \\
            &=\sum\limits_{l=1}^{s}{\oypm^{m_{l,0}}\cdot\left( \ddx\left( \fcs \right)\cdot\gkma[k]{m_l}{a_l} +
            \fcs\cdot\ddx\left( \gkma{m_l}{a_l} \right) \right)}\\
            &=\sum\limits_{l=1}^{s}{\oypm^{m_{l,0}}\cdot\ddx\left( \fcs \right)\cdot\gkma[k]{m_l}{a_l}} +
            \sum\limits_{l=1}^{s}{\oypm^{m_{l,0}}\cdot\fcs\cdot\ddx\left( \gkma{m_l}{a_l} \right)}.
        \end{align*}
        Now, by denoting
        \begin{gather}
            \label{sig-1-def}
            \Sigma_{1}(x) \coloneq \sum\limits_{l=1}^{s}{\oypm^{m_{l,0}}\cdot\ddx\left( \fcs \right)\cdot\gkma[k]{m_l}{a_l}}\\
            \label{sig-2-def}
            \Sigma_{2}(x) \coloneq \sum\limits_{l=1}^{s}{\oypm^{m_{l,0}}\cdot\fcs\cdot\ddx\left( \gkma{m_l}{a_l} \right)}
        \end{gather}
        we have that $\ddxk[k+1]\fcs = \Sigma_{1}(x)+\Sigma_{2}(x)$.
        
        Next we analyze $\Sigma_{1}(x)$ and $\Sigma_{2}(x)$. First, by~\cref{eq:basis} we see that
        \begin{equation*}
            \Sigma_{1}(x) = \sum\limits_{l=1}^{s}{\oypm^{m_{l,0}}\cdot\fcs\cdot\xi_{1}(x)\cdot\gkma[k]{m_l}{a_l}}
            = \sum\limits_{l=1}^{s}{\oypm^{m_{l,0}}\cdot\fcs\cdot\xi_{1}(x) \cdot\prod\limits_{j=1}^{k}{a_{l,j}\cdot\xi_{j}^{m_{l,j}}(x)}}
        \end{equation*}
        and if we denote $\tilde{m}_{l,1} = m_{l,1}+1$ and $\forall_{2\leq j\leq k} \; \tilde{m}_{l,j}=m_{l,j}$ we have
        \begin{equation}\label{eq:sig-1-result}
            \Sigma_{1}(x) = \sum\limits_{l=1}^{s}{\oypm^{m_{l,0}}\cdot\fcs\cdot\xi_{1}\cdot\prod\limits_{j=1}^{k}{a_{l,j}\cdot\xi_{j}^{\tilde{m}_{l,j}}(x)}}=
            \sum\limits_{l=1}^{s}{\oypm^{m_{l,0}}\cdot\fcs\cdot\xi_{1}(x)\cdot\gkma[k]{\tilde{m}_{l}}{a_l}}.
        \end{equation}
        
        Moving on to $\Sigma_{2}(x)$, we analyze $\ddx\left( \gkma{m_l}{a_l} \right)$ using General Leibniz rule~\cite{olver1993applications}:
        \begin{align*}
            \ddx\left( \gkma{m_l}{a_l} \right) &= \ddx\left( \prod\limits_{j=1}^{k}{a_{l,j}\cdot\xi_{j}^{m_{l,j}}(x)} \right)\\
            &=\sum\limits_{t_1+t_2+\cdots + t_k = 1}\left( \prod\limits_{1\leq s\leq k}{\ddxk[t_s]\left( a_{l,s}\cdot\xi_{s}^{m_{l,s}}(x) \right)} \right).
        \end{align*}
        Since $t_s\in\mathbb{N}$ for each $1\leq s\leq k$ the equation $t_1+t_2+\cdots + t_k = 1$ implies that there are only $k$ possibilities, where for some $i=1,\dots,k$ we have 
        $t_i=1$ and for $s\neq i$ we have $t_s=0$. This leads to
        \begin{gather*}
            \sum\limits_{t_1+t_2+\cdots + t_k = 1}\left( \prod\limits_{1\leq s\leq k}{\ddxk[t_s]\left( a_{l,s}\cdot\xi_{s}^{m_{l,s}}(x) \right)} \right)= \\
            m_{l,1}\cdot\xi_{1}^{-1}(x)\cdot\xi_{2}(x)\cdot \prod\limits_{j=3}^{k}{a_{l,j}\cdot\xi_{j}^{m_{l,j}}(x)} +
            m_{l,2}\cdot\xi_{2}^{-1}(x)\cdot\xi_{3}(x)\cdot \prod\limits_{\overset{j=1}{j\neq 2,3}}^{k}{a_{l,j}\cdot\xi_{j}^{m_{l,j}}(x)} +\cdots \\
            + m_{l,k-1}\cdot\xi_{k-1}^{-1}(x) \cdot \xi_{k}(x) \cdot \prod\limits_{\overset{j=1}{j\neq k, k+1}}^{k}{a_{l,j}\cdot\xi_{j}^{m_{l,j}}(x)} +
            m_{l,k}\cdot\xi_{k}^{-1}(x)\cdot\xi_{k+1}(x) \cdot \prod\limits_{j=1}^{k}{a_{l,j}\cdot \xi_{j}^{m_{l,j}}(x)}.
        \end{gather*}
        In each of the first $k-1$ elements in the sum above the power of $\xi_{i}$ is reduced by $1$ and the power of $\xi_{i+1}$ is increased by $1$
        where $1\leq i\leq k-1$ and in the last element the only power of $\xi_{k}$ is reduced by $1$ and there's a new element in the multiplication
        which is $\xi_{k+1}$ with power and coefficient being both $1$z, so to continue we will denote the powers of $\xi_{i}$ in each
        element $1\leq j\leq k$ by $\tilde{m}_{l_{j},i}$.
        We also have a similar change in the coefficients $a_{l,j}$, where each element $j$ of the sum above is multiplied by the corresponding power of $\xi_{i}$
        so we denote the new coefficients by $\tilde{a}_{l_{j},i}$ and write:
        \begin{gather*}
            m_{l,1}\cdot\xi_{1}^{-1}(x)\cdot\xi_{2}(x)\cdot \prod\limits_{j=3}^{k}{a_{l,j}\cdot\xi_{j}^{m_{l,j}}(x)} +
            m_{l,2}\cdot\xi_{2}^{-1}(x)\cdot\xi_{3}(x)\cdot \prod\limits_{\overset{j=1}{j\neq 2,3}}^{k}{a_{l,j}\cdot\xi_{j}^{m_{l,j}}(x)} +\cdots +\\
            m_{l,k-1}\cdot\xi_{k-1}^{-1}(x)\cdot\xi_{k}(x)\cdot \prod\limits_{\overset{j=1}{j\neq k, k+1}}^{k}{a_{l,j}\cdot\xi_{j}^{m_{l,j}}(x)} +
            m_{l,k}\cdot\xi_{k}^{-1}(x)\cdot\xi_{k}(x)\cdot \prod\limits_{j=1}^{k}{a_{l,j}\cdot\xi_{j}^{m_{l,j}}(x)}=\\
            \prod\limits_{j=1}^{k}{\tilde{a}_{l_{1},j}\cdot\xi_{j}^{\tilde{m}_{l_{1},j}}(x)} +
            \prod\limits_{j=1}^{k}{\tilde{a}_{l_{2},j}\cdot\xi_{j}^{\tilde{m}_{l_{2},j}}(x)} + \cdots +
            \prod\limits_{j=1}^{k}{\tilde{a}_{l_{k-1},j}\cdot\xi_{j}^{\tilde{m}_{l_{k-1},j}}(x)} +
            \prod\limits_{j=1}^{k+1}{\tilde{a}_{l_{k},j}\cdot\xi_{j}^{\tilde{m}_{l_{k},j}}(x)} = \\
            \gkma{\tilde{a}_{l_{1},j}}{\tilde{m}_{l_{1},j}} + \gkma{\tilde{a}_{l_{2},j}}{\tilde{m}_{l_{2},j}} +\cdots +
            \gkma{\tilde{a}_{l_{k-1},j}}{\tilde{m}_{l_{k-1},j}} + \gkma[k+1]{\tilde{a}_{l_{k},j}}{\tilde{m}_{l_{k},j}} =\\
            \sum\limits_{i=1}^{k-1}{\gkma[k]{\tilde{a}_{l_{i},j}}{\tilde{m}_{l_{i},j}}} + \gkma[k+1]{\tilde{a}_{l_{k},j}}{\tilde{m}_{l_{k},j}}.
        \end{gather*}
        Going back to~\cref{sig-2-def} we conclude that
        \begin{equation}\label{eq:sig-2-result}
            \begin{split}
                \Sigma_{2}(x) &= \sum\limits_{l=1}^{s}{\oypm^{m_{l,0}}\cdot\fcs\cdot
                \left( \sum\limits_{i=1}^{k-1}{\gkma[k]{\tilde{a}_{l_{i},j}}{\tilde{m}_{l_{i},j}}} + \gkma[k+1]{\tilde{a}_{l_{k},j}}{\tilde{m}_{l_{k},j}} \right)}=\\
                &\sum\limits_{l=1}^{s}{\oypm^{m_{l,0}}\cdot\fcs\cdot\left( \sum\limits_{i=1}^{k-1}{\gkma[k]{\tilde{a}_{l_{i},j}}{\tilde{m}_{l_{i},j}}} \right)} +
                \sum\limits_{l=1}^{s}{\oypm^{m_{l,0}}\cdot\fcs\cdot\gkma[k+1]{\tilde{a}_{l_{k},j}}{\tilde{m}_{l_{k},j}}}=\\
                &\sum\limits_{l=1}^{s}{\sum\limits_{i=1}^{k-1}{\oypm^{m_{l,0}}\cdot\fcs\cdot\gkma[k]{\tilde{a}_{l_{i},j}}{\tilde{m}_{l_{i},j}}}} +
                \sum\limits_{l=1}^{s}{\oypm^{m_{l,0}}\cdot\fcs\cdot\gkma[k+1]{\tilde{a}_{l_{k},j}}{\tilde{m}_{l_{k},j}}}..
            \end{split}
        \end{equation}
        Finally, using \cref{eq:sig-1-result,eq:sig-2-result} we have:
        \begin{align*}
            \ddxk[k+1]\fcs &= \Sigma_{1}(x) + \Sigma_{2}(x)\\
            &=\sum\limits_{l=1}^{s}{\oypm^{m_{l,0}}\cdot\fcs\cdot\xi_{1}\cdot\gkma[k]{\tilde{m}_{l}}{a_l}}\\
            &\quad + \sum\limits_{l=1}^{s}{\sum\limits_{i=1}^{k-1}{\oypm^{m_{l,0}}\cdot\fcs\cdot\gkma[k]{\tilde{a}_{l_{i},j}}{\tilde{m}_{l_{i},j}}}}\\
            &\quad +\sum\limits_{l=1}^{s}{\oypm^{m_{l,0}}\cdot\fcs\cdot\gkma[k+1]{\tilde{a}_{l_{k},j}}{\tilde{m}_{l_{k},j}}}.
        \end{align*}
        So we can conclude that there exists $s_{1}\in\mathbb{N}$ s.t.:
        \begin{gather*}
            \ddxk[k+1]\fcs = \sum\limits_{l=1}^{s_1}{\oypm^{m_{l,0}}\cdot \fcs \cdot \gkma[k]{m_l}{a_l}}.
        \end{gather*}
        Since
        \begin{gather*}
            \ddxk[k+1]\fcs=\ddx\left( \ddxk[k]\fcs \right) =
            \ddx\left( \sum\limits_{l=1}^{s}{\oypm^{m_{l,0}}\cdot \fcs \cdot \gkma[k]{m_l}{a_l}} \right) =\\
            \ddx\left( \oypm^{m_{1,0}}\cdot \fcs \cdot \gkma[k]{m_l}{a_l} + \sum\limits_{l=2}^{s}{\oypm^{m_{l,0}}\cdot \fcs \cdot \gkma[k]{m_l}{a_l}} \right) =\\
            \ddx\left( \oypm^{m_{1,0}}\cdot \fcs \cdot \gkma[k]{m_l}{a_l} \right) +
            \ddx\left( \sum\limits_{l=2}^{s}{\oypm^{m_{l,0}}\cdot \fcs \cdot \gkma[k]{m_l}{a_l}} \right)=\\
            \oypm^{m_{1,0}}\cdot \fcs^{(1)} \cdot \gkma[k]{m_l}{a_l} + \oypm^{m_{1,0}}\cdot \fcs \cdot \ddx\left( \gkma[k]{m_l}{a_l} \right) +
            \ddx\left( \sum\limits_{l=2}^{s}{\oypm^{m_{l,0}}\cdot \fcs \cdot \gkma[k]{m_l}{a_l}} \right)=\\
            \oypm^{m_{1,0}+1} \cdot a_{1,1}\cdot\xi_{1}^{m_{1,1}+1}(x)\cdot \prod\limits_{j=2}^{k}{a_{l,j}\cdot\xi_{j}^{m_{l,j}}(x)} +
            \oypm^{m_{1,0}}\cdot \fcs \cdot \ddx\left( \gkma[k]{m_l}{a_l} \right) + \\
            \ddx\left( \sum\limits_{l=2}^{s}{\oypm^{m_{l,0}}\cdot \fcs \cdot \gkma[k]{m_l}{a_l}} \right),
        \end{gather*}
        adding~\cref{eq:assumption} we have $m_{1,0}+1=k+1$ and $m_{1,1}+1 = k+1$ which implies that the highest power of $\oypm$ in
        $\ddxk[k+1]\fcs$ is $k+1$ and the highest power of $\xi_{i}$ in $\ddxk[k+1]\fcs$ is also $k+1$.
    \end{proof}

    \begin{proof}[Proof of~\cref{psi-oy-properties}]
        First, since $F_1,F_2,\xi$ are smooth in $x$, we can differentiate under the integral sign and conclude that $\psiOy$ is smooth
        in $x$, with possibly a jump discontinuity at the endpoint $x=-\pi$.
        For an example of a situation with nonzero jump magnitudes, see \cref{case-for-discontinuity-for-psi} below.
        \newpar
        Now let $k\in\mathbb{N}$ s.t. $1\leq k \leq d$ and by using~\cref{derivative-of-psi} we get:
        \begin{equation}\label{eq:doy-derivative-of-psi}
            \ddxk[k]\psiOy(x)  = \frac{1}{2\pi} \left( \ik[k] + \sum_{l=0}^{k-1}{\ddxk[l]\left( e^{-\imath\oy\xi(x)}\xi^{(1)}(x)\Al[k-1-l] \right)} \right)
        \end{equation}
        where (see~\cref{integral-notations})
        \begin{equation*}
            I_{\oy, k}(x) \coloneqq \int_{-\pi}^{\xi(x)} e^{-\imath \oy y} \frac{\partial^{k+1}}{\partial x^{k+1}} F(x,y) \, dy +
            \int_{\xi(x)}^{\pi} e^{-\imath \oy y} \frac{\partial^{k+1}}{\partial x^{k+1}} F(x,y) \, dy
        \end{equation*}
        and since for each $x\in\torus[1]$ we have $\Fx\in\PC$ (recall \cref{description-of-F-and-xi}), this implies that
        $I_{\oy, k}$ is continuous in $(-\pi,\pi)$ with one singularity at most at $x=-\pi$.
        \newpar
        Now we proceed to analyze the sum on the right-hand side of~\cref{eq:doy-derivative-of-psi}.
        First we begin by denoting $\xii(x)\coloneq\ddxk[i]\xix$, $f_{c}(x)\coloneq\cos(\oy\xix)$, $f_{s}(x)\coloneq\sin(\oy\xix)$,
        $g_{c}(x)\coloneq f_{c}(x)\xi_{1}(x)$ and $g_{s}(x)\coloneq f_{s}(x)\xi_{1}(x)$ so that we can claim that
        \begin{align*}
            \sum_{l=0}^{k-1}{\ddxk[l]\left( e^{-\imath\oy\xix}\xi_{1}(x) A_{k-1-l}(x) \right)} &=
            \sum_{l=0}^{k-1}{\ddxk[l]\left( \left( \cos(\oy\xix)-\imath\sin(\oy\xix) \right)\xi_{1}(x) A_{k-1-l}(x) \right)}\\
            &=\sum_{l=0}^{k-1}{\ddxk[l]\left(\cos(\oy\xix)\xi_{1}(x) A_{k-1-l}(x) \right)}\\
            &\qquad -\imath\sum_{l=0}^{k-1}{\ddxk[l]\left( \imath\sin(\oy\xix)\xi_{1}(x) A_{k-1-l}(x) \right)}\\
            &=\sum_{l=0}^{k-1}{\ddxk[l]\left(g_{c}(x) A_{k-1-l}(x) \right)}-
            \imath\sum_{l=0}^{k-1}{\ddxk[l]\left( g_{s}(x) A_{k-1-l}(x) \right)}.
        \end{align*}
        By~\cref{eq:my_lebnitz} we claim that
        \begin{gather*}
            \sum_{l=0}^{k-1}{\ddxk[l]\left(g_{c}(x) A_{k-1-l}(x) \right)}-
            \imath\sum_{l=0}^{k-1}{\ddxk[l]\left( g_{s}(x) A_{k-1-l}(x) \right)} = \\
            \sum_{l=0}^{k-1}{\sum_{i=0}^{l}{a_{i} \cdot g_{c}^{(i)}(x) \cdot A_{k-1-l}^{(l-i)}}(x)} -
            \imath\sum_{l=0}^{k-1}{\sum_{i=0}^{l}{b_{i} \cdot g_{s}^{(i)}(x) \cdot A_{k-1-l}^{(l-i)}}(x)} = \\
            \sum_{l=0}^{k-1}{\sum_{i=0}^{l}{\left( a_{i} \cdot g_{c}^{(i)}(x) - i \cdot b_{i} \cdot g_{s}^{(i)}(x) \right)\cdot A_{k-1-l}^{(l-i)}}(x)}
        \end{gather*}
        and furthermore by~\cref{lem:derivative-of-gc-gs}
        \begin{gather*}
            \sum_{l=0}^{k-1}{\sum_{i=0}^{l}{\left( a_{i} \cdot g_{c}(x)^{(i)} - i \cdot b_{i} \cdot g_{s}^{(i)}(x) \right)\cdot A_{k-1-l}^{(l-i)}(x)}}=\\
            \sum_{l=0}^{k-1}{\sum_{i=0}^{l}{\sum_{j=0}^{i}{\left( a_{i} \cdot f_{c}^{(j)}(x) - i \cdot b_{i} \cdot f_{s}^{(j)}(x) \right)\cdot \xi_{i+1-j}(x)
            \cdot A_{k-1-l}^{(l-i)}(x)}}}.
        \end{gather*}
        Adding~\cref{lem:fcs-derivatives} we conclude that for $k\in\mathbb{N}$ each of the following derivatives $\fcsk[k]$ is a sum of
        the functions of the form
        \begin{equation}\label{eq:sum-of-fcs}
            a_{i}\cdot\oyapm[i]\cdot\fcs\cdot\xi_{1}^{m_1}(x) \cdots \xi_{k}^{m_n}(x)
        \end{equation}
        where $a_{i}\in\mathbb{R}$ and $m_{1},\dots,m_{n}\in\mathbb{N}$.
        Then by combining all of our conclusions above, the following sum
        \begin{gather}
            \sum_{l=0}^{k-1}{\ddxk[l]\left( e^{-\imath\oy\xi(x)}\xi^{(1)}(x)\Al[k-1-l] \right)} = \\
            \sum_{l=0}^{k-1}{\sum_{i=0}^{l}{\sum_{j=0}^{i}{\left( a_{i} \cdot f_{c}^{(j)}(x) - i \cdot b_{i} \cdot f_{s}^{(j)}(x) \right) \cdot \xi_{i+1-j}(x) \cdot
            A_{k-1-l}^{(l-i)}(x)}}}
        \end{gather}
        is seen to be a sum of functions described in~\cref{eq:sum-of-fcs} where the derivative order of $\xix$ is at most $\doy$ which is
        continuous in $(-\pi, \pi)$ with one singularity at most at $x=-\pi$. Then in turn, we have
        \begin{equation*}
            \ddxk[k]\psiOy(x)  = \frac{1}{2\pi} \left( \ik[k] + \sum_{l=0}^{k-1}{\ddxk[l]\left( e^{-\imath\oy\xi(x)}\xi^{(1)}(x)\Al[k-1-l] \right)} \right)
        \end{equation*}
        as a sum of continuous functions in $(-\pi,\pi)$ with one singularity at most at $x=-\pi$ for each $1\leq k \leq d$.
    \end{proof}

    \begin{proof}[Proof of \cref{prop:phi-oy-coefficients}]\label{proof-phi-oy-coefficients}
        The Bernoulli polynomials satisfy~\cite{sun2004identities}:
        \begin{gather*}
            \begin{cases}
                B_{n}(0) = B_{n}(1) &\text{if } n \neq 1 \\
                B_{n}(0) = -B_{n}(1) = -\frac{1}{2} &\text{if } n = 1
            \end{cases}, \qquad \text{for } n\in\mathbb{N}
        \end{gather*}
        implying that:
        \begin{gather}
            \label{eq:2.17}
            \begin{cases}
                V_{\psiOy,l}(-\pi) - V_{\psiOy,l}(\pi) = 1 &\text{if } l=0 \\
                V_{\psiOy,l}(-\pi) - V_{\psiOy,l}(\pi) = 0 &\text{if } 1 \leq l \leq\doy
            \end{cases}.
        \end{gather}
        Adding~\cref{lem:phi-oy-derivatives} we have:
        \begin{gather}
            \label{eq:2.18}
            \ddxk[m]\phiOy(x)|_{x=-\pi} - \ddxk[m]\phiOy(x)|_{x=\pi} =
            \begin{cases}
                \Aoylarg[m] &\text{if } 0 \leq m \leq \doy \\
                0 &\text{if } m > \doy
            \end{cases}
        \end{gather}
        so now we go back to $c_{k}(\phiOy) \coloneqq \normalizer\int_{-\pi}^{\pi} e^{-\imath kx} \phiOy(x)\,dx$ and use integration by parts
        with~\eqref{eq:2.18}:
        \begin{align*}
            2\pi c_{k}(\phiOy) &= \left[ \frac{e^{-\imath kx} \phiOy(x)}{-ik} \right]_{x=-\pi}^{x=\pi}+\frac{1}{\imath k}
            \int_{-\pi}^{\pi} e^{-\imath kx} \phi_{\oy}^{(1)}(x)\,dx \\
            &= \frac{(-1)^k \left( \phiOy(-\pi) - \phiOy(\pi) \right)}{\imath k}+\frac{1}{\imath k}\int_{-\pi}^{\pi} e^{-\imath kx}
            \phi_{\oy}^{(1)}(x)\,dx \\
            &= \frac{(-1)^k \Aoylarg[0]}{\imath k} + \frac{1}{\imath k}\int_{-\pi}^{\pi} e^{-\imath kx} \phi_{\oy}^{(1)}(x)\,dx
            = \ldots\\
            &\ldots = \sum_{l=0}^{d} \frac{(-1)^k \Aoyl}{(\imath k)^{l+1}}.
        \end{align*}
        As we defined $c_{0}(\phiOy) \equiv 0$, we conclude:
        \begin{gather*}
            c_{k}(\phiOy) =
            \begin{cases}
                0 &\text{if } k=0 \\
                \frac{\left( -1 \right)^k}{2\pi} \sum_{l=0}^{d_{\omega_{y}}} \frac{\Aoyl}{\left( \imath k \right)
                ^{\left( l+1 \right)}} &\text{otherwise }
            \end{cases}. \qedhere
        \end{gather*}
    \end{proof}

    \begin{lemma}
        \label{lem:phi-oy-derivatives}
        Let $\oyInZ$ and let $\phiOy(x)$ be as defined in~\cref{phi-omega-y} a pieceiwse polynomial of degree $\doy$. Then:
        \begin{gather}
            \label{eq:2.13}
            \ddxk[m] \phiOy(x) =
            \begin{cases}
                -\frac{\Aoylarg[m-1]}{2\pi} + \sum_{l=m}^{\doy} A_{l}(\oy) V_{\psi. l-m}(x) &\text{if } 1 \leq m \leq \doy \\
                -\frac{\Aoylarg[\doy]}{2\pi} &\text{if } m = \doy + 1 \\
                0 &\text{if } m > \doy + 1
            \end{cases}
        \end{gather}
    \end{lemma}
    \begin{proof}
        First we remark that for Bernoulli polynomials, $B_{n}(x)$, we have $\ddx B_{n}(x) = n\cdot B_{n-1}(x)$ which in turn for $m\leq n$
        gives us:
        \begin{gather*}
            \ddxk[m] B_{l+1}(x) = \prod_{k=0}^{m-1} (l-k+1) \cdot B_{n-m}(x).
        \end{gather*}
        Therefore we have:
        \begin{gather}
            \label{eq:2.14}
            \ddxk[m] B_{l+1}\left( \frac{x+\pi}{2\pi} \right) =
            \begin{cases}
                \frac{\prod_{k=0}^{m-1} \left( l-k+1 \right) }{(2\pi)^{m}} \cdot B_{l-m+1}\left( \frac{x +\pi}{2\pi} \right) &\text{if }
                m\leq l+1\\
                0 &\text{if } m > l + 1
            \end{cases}
        \end{gather}
        Now we have:
        \begin{gather*}
            \ddxk[m] \Vpsi(x) = \ddxk[m] \left( -\frac{\left(2\pi\right)^l}{\left( l+1 \right)!}B_{l+1}\left( \frac{x+\pi}{2\pi} \right)
            \right)
            = -\frac{\left(2\pi\right)^l}{\left( l+1 \right)!} \ddxk[m] \left( B_{l+1}\left( \frac{x+\pi}{2\pi} \right) \right) = \\
            = -\frac{\left(2\pi\right)^l}{\left( l+1 \right)!}
            \begin{cases}
                \frac{\prod_{k=0}^{m-1} \left( l-k+1 \right) }{(2\pi)^{m}} \cdot B_{l-m+1}\left( \frac{x +\pi}{2\pi} \right) &\text{if }
                m < l+1\\
                -\normalizer &\text{if } m = l+1 \\
                0 &\text{if } m > l+1
            \end{cases} \\
            =\begin{cases}
                 \frac{\left( 2\pi \right)^{l-m}}{\left( l-m+1 \right)!} \cdot B_{l-m+1}\left( \frac{x +\pi}{2\pi} \right) &\text{if }
                 m < l+1\\
                 -\normalizer &\text{if } m = l+1 \\
                 0 &\text{if } m > l+1
            \end{cases}
            =\begin{cases}
                 V_{\psi, l-m}(x) &\text{if } m < l+1 \\
                 -\normalizer &\text{if } m = l+1 \\
                 0 &\text{if } m > l+1
            \end{cases}
        \end{gather*}
        \indent and if we denote $k=l-m$ then
        \begin{gather}
            \label{eq:derivative-of-v}
            \ddxk[m] \Vpsi(x) =
            \begin{cases}
                0 &\text{if } k<-1 \\
                -\normalizer &\text{if } k = -1 \\
                V_{\psi, k}(x) &\text{if } k > -1
            \end{cases}.
        \end{gather}
        Going back to $\ddxk[m] \phiOy(x)$ using~\eqref{eq:derivative-of-v}  and keeping in mind that $m\in\mathbb{N}$ we
         get:
        \begin{align*}
            \ddxk[m]\phiOy(x) &= \ddxk[m]\left( \sum_{l=0}^{\doy} \Aoyl\Vpsi(x) \right)= \sum_{l=0}^{\doy} \Aoyl\ddxk[m]\Vpsi(x)= \\
            &=\sum_{l=0}^{\doy} \Aoyl
            \begin{cases}
                V_{\psi, l-m}(x) &\text{if } m < l+1 \\
                -\normalizer &\text{if } m = l+1 \\
                0 &\text{if } m > l+1
            \end{cases} \\
            &=\sum_{k=-m}^{\doy - m} \Aoylarg[k+m]
            \begin{cases}
                V_{\psi, k}(x) &\text{if } k > -1 \\
                -\normalizer &\text{if } k = -1 \\
                0 &\text{if } k < -1
            \end{cases} \\
            &=\sum_{k=-1}^{\doy - m} \Aoylarg[k+m]
            \begin{cases}
                V_{\psi, k}(x) &\text{if } k > -1 \\
                -\normalizer &\text{if } k = -1
            \end{cases} \\
            &= -\frac{\Aoylarg[m-1](\oy)}{2\pi}+\sum_{k=0}^{\doy - m} \Aoylarg[k+m] V_{\psi, k}(x)
        \end{align*}
        Swtiching back from $k=l-m$ to $l = m+k$ we conclude the proof:
        \begin{gather*}
            \ddxk[m] \phiOy(x) =
            \begin{cases}
                -\frac{\Aoylarg[m-1]}{2\pi} + \sum_{l=m}^{\doy} \Aoyl V_{\psi. l-m}(x) &\text{if } 1 \leq m \leq \doy \\
                -\frac{\Aoylarg[\doy]}{2\pi} &\text{if } m = \doy + 1 \\
                0 &\text{if } m > \doy + 1
            \end{cases}. \qedhere
        \end{gather*}
    \end{proof}

\subsection{A bound on the derivatives}\label{subsec:proof-of-prop-15}
%
Here we prove a key technical estimate, used in the proof of \cref{main-result-for-coefficients-of-gamma-omega-y}.

\begin{proposition} \label{bound-for-gamma-omega-y-derivatives}
    Assume $\oyInZ$ and let $\gammaOy$ be as in~\cref{eq:decomposition-psi-omega-y} so that $\gammaOy\in C_{\mathbb{T}}^{d+1}$,
$d + 1 \in\mathbb{N}$ and let $k\leq d$, $x\in\torus[1]$.
Then exists constants $B_{F},\;A_x \geq 0$ (see~\cref{definition-of-Fx-and-jump-location-and-jump-magnitudes-and-assumptions})
and $\Admxi$ as defined in~\cref{Admxi} such that:
\begin{gather*}
        \abs{\ddxk[k+1]\left( \gammaOy \right)} \leq \left(2d + 5\right) \left(B_{F} + \Admxi A_x \abs{\oy}^d \right).
    \end{gather*}
\end{proposition}
Before we prove~\cref{bound-for-gamma-omega-y-derivatives} we develop some auxiliary statements.

\begin{lemma}
    \label{bound-on-composition-exp-xi}
    Let $\xix$ be as described in~\cref{description-of-F-and-xi,definition-of-Fx-and-jump-location-and-jump-magnitudes-and-assumptions}
    and let $\oyInZ$, and  $k\in\mathbb{N}$ such that $k\leq d$. Then there  exists
    $c_k\in\mathbb{N}\setminus\left\{0\right\}$ such that:
    \begin{gather}
        \abs{\ddxk[k]\left( e^{-\imath\oy\xi(x)} \right)} \leq c_k\abs{\oy}^{k}M^{k^2}_{\xi}.
    \end{gather}
\end{lemma}

 \begin{proof}
    First we notice that
    \begin{gather}\label{eioyxi-cos-sin}
        \begin{split}
            \abs{\ddxk[k]\left( e^{-\imath\oy\xi(x)} \right)} &= \abs{\ddxk[k]\left( \cosoyxi-
            \imath\sinoyxi \right)} \\
            &\leq \abs{\ddxk[k]\cosoyxi}+\abs{\ddxk[k]\sinoyxi}
        \end{split}
    \end{gather}
    and now we will look for a bound over $\abs{\cos(\oy\xix)}$ and $\abs{\sin(\oy\xix)}$.
    Before we begin the proof let us note that since we are only interested in finding a bound over the two
    above functions, we will allow ourself to avoid finding the explicit formula for $\ddxk[k]\cos(\oy\xix)$ and $\ddxk[k]\sin(\oy\xix)$,
    and because of the relation between $\ddx\cos(x)$ and $\ddx\sin(x)$ we will be using a more loose notations
    and assume that the bound over $\abs{\cos(\oy\xix)}$ will apply to $\ddxk[k]\sin(\oy\xix)$ as well.
    \newpar
    We begin borrowing specific notations from~\cref{lem:derivative-of-gc-gs}:
    \begin{itemize}
        \item $\xii(x) \coloneqq \ddxk[i] \xix$
        \item $\oyapm \coloneqq -\oy^{a}\text{ or } \oy^{a}$, where $a\in\mathbb{N}$
        \item $\fcsk[k] \coloneqq \ddxk[k]\left( \cosoyxi \right)$ or $\ddxk[k]\left( \sinoyxi \right)$
    \end{itemize}
    Before adding another notations we observe the result of~\cref{lem:fcs-derivatives}
    and conclude that for each derivative $\ddxk[l]\fcs$ we will have a sum of equations such as
    $\oyapm[i]\fcs\left( \xii[1]^{m_1}(x)\cdots\xii[k]^{m_k}(x) \right)$,
    but if we are interested in finding a bound for $\abs{\ddxk[s]\left( e^{-\imath\oy\xi(x)} \right)}$ we will evaluate the
    absolute value of every equation of the type presented above.
    So together with the assumption that for every $k\leq d_{\xi}$ we have a constant $M_{\xi}$ such that
    $\abs{\xix^{k}}\leq M_{\xi}$ and $\abs{\fcs}\leq 1$ we can claim that we can ignore the derivative order of $\xix$ and
    focus only on its power.
    Now we denote $\xi_{*}(x)\coloneqq \xi_{k}(x)$ for every $0\leq k\leq s$ and observe that
    \begin{gather*}
        \begin{split}
            \abs{\ddx\left( \xi_{(i)}(x) \right)^a} &= \abs{\ddx\left( \xi_{*}^a(x) \right)} \\
            &= \abs{a\left( \xi_{(i-1)}(x) \right)^{a-1}\xi_{(1)}(x)} = \abs{a\xi_{*}^{(a-1)}(x) \xi_{*}(x)} \\
            &= a\abs{\xi_{*}(x)}^a \leq aM_{\xi}^a
        \end{split}
    \end{gather*}
    Using~\cref{lem:fcs-derivatives} we see that there exists $s\in \mathbb{N}$ s.t.
    \begin{gather*}
        \ddxk[k]\fcs = \sum\limits_{l}^{s}{\oypm^{m_{l,0}}\cdot \fcs \cdot \gkma[k]{m_l}{a_l}}.
    \end{gather*}
    Now we can proceed:
    \begin{align*}
        \abs{\ddxk[k]\fcs} &= \abs{\sum\limits_{l=1}^{s}{\oypm^{m_{l,0}}\cdot \fcs \cdot \gkma[k]{m_l}{a_l}}}\\
        &=\abs{\sum\limits_{l=1}^{s}{\oypm^{m_{l,0}}\cdot \fcs \cdot
        \left( \prod\limits_{j=1}^{k}{a_{l,j}\cdot\xi_{j}^{m_{l,j}}(x)} \right)}} \\
        &\leq\sum\limits_{l=1}^{s}{\abs{\oy}^{m_{l,0}} \cdot \abs{\fcs} \cdot
        \abs{\prod\limits_{j=1}^{k}{a_{l,j}\cdot\xi_{j}^{m_{l,j}}(x)}}}\\
        &\overset{\abs{\fcs}\leq 1}{\leq}\sum\limits_{l=1}^{s}{\abs{\oy}^{m_{l,0}} \cdot
        \prod\limits_{j=1}^{k}{\abs{a_{l,j}}\cdot \abs{\xi_{j}(x)}^{m_{l,j}}}}\\
        &\leq\abs{\oy}^{k}\cdot\sum\limits_{l=1}^{s}{\prod\limits_{j=1}^{k}{\abs{a_{l,j}}\cdot \abs{\xi_{*}(x)}^{m_{l,j}}}}\\
        &\leq\abs{\oy}^{k}\cdot\sum\limits_{l=1}^{s}{\prod\limits_{j=1}^{k}{\abs{a_{l,j}}\cdot \abs{\xi_{*}(x)}^{k}}}\\
        &\overset{\abs{\xi_{*}(x)}^{m_{l,j}} \leq M_{\xi}^{k}}{\leq}\abs{\oy}^{k}\cdot s \cdot M_{\xi}^{k^2}\cdot
        \sum\limits_{l=1}^{s}{\prod\limits_{j=1}^{k}{\abs{a_{l,j}}}}.
    \end{align*}

    Denoting $c_{k}\coloneq 2\cdot s\cdot\sum\limits_{l=1}^{s}{\prod\limits_{j=1}^{k}{\abs{a_{l,j}}}}$ together
    with~\cref{eioyxi-cos-sin} implies
    \begin{gather*}
        \abs{\ddxk[k]\left( e^{-\imath\oy\xi(x)} \right)} \leq 2\cdot \abs{\ddxk[k]\fcs}
        \leq c_{k}\cdot\abs{\oy}^{k}\cdot M_{\xi}^{k^2}
    \end{gather*}
\end{proof}

\begin{lemma}
    \label{bound-for-xi-times-A}
    Let x $\in\torus[1]$ and $s\in\mathbb{N}$, take $A_{l}(x)$ as defined
    in~\cref{definition-of-Fx-and-jump-location-and-jump-magnitudes-and-assumptions}, then
    \begin{equation} \label{eq:bound-over-derivatives-xi-Al}
        \abs{\ddxk[s] \left( \xi^{(1)}(x) A_{s-l}(x) \right)} \leq  2^{s} M_{\xi} A_{x}.
    \end{equation}
\end{lemma}

\begin{proof}
    Using General Leibniz rule~\cite{olver1993applications} we have:
    \begin{gather*}
        \abs{\ddxk[s]\left( \xi^{(1)}(x) A_{s-l}(x) \right)} = \abs{\sum_{m=0}^{s}{\binom{s}{m} \left( \xi^{(1)}(x)
        \right)^{(s-m)} \left( A_{s-l}(x) \right)^{(m)}}}\\
        \leq \sum_{m=0}^{s}{\binom{s}{m}\abs{\left( \xi(x) \right)^{(s-m+1)}} \abs{\left( A_{s-l}(x)\right)^{(m)}}}
        \leq M_{\xi} A_{x}\sum_{m=0}^{s}{\binom{s}{m}} \leq 2^{s} M_{\xi} A_{x}. \qedhere
    \end{gather*}
\end{proof}

\noindent Now we denote
\begin{gather} \label{integral-notations}
    I_{\oy, k}(x) \coloneqq \int_{-\pi}^{\xi(x)} e^{-\imath \oy y} \frac{\partial^{k+1}}{\partial x^{k+1}} F(x,y) \, dy +
    \int_{\xi(x)}^{\pi} e^{-\imath \oy y} \frac{\partial^{k+1}}{\partial x^{k+1}} F(x,y) \, dy
\end{gather}

\noindent and formulate a useful result by applying Leibniz integral rule on $\ik$:
\begin{gather} \label{result-from-Leibnitz-general-rule}
    \ddx\left( \ik \right) = e^{-\imath\oy\xi(x)}\xi(x)^{(1)}\Al[k] + \ik[k+1].
\end{gather}

\begin{proposition} \label{bound-for-derivative-of-psi-omega-y}
    Let $x\in\torus[1]$, $\oyInZ$ and let $\psiOy$ as described in~\cref{def:psi-oy}, then exists $B_{F}\geq 0$ and $\Admxi$ (as defined in~\cref{Admxi})
    s.t. for every $k=0,1,\dots,d$ we have
    \begin{gather}
        \abs{\ddxk[k+1]\psiOy(x)} \leq B_{F} + \Admxi \cdot A_x \cdot \abs{\oy}^{d}
    \end{gather}
\end{proposition}

\begin{proof}
    \begin{align*}
        2\pi\ddxk[k+1]\psiOy(x) &= \ddxk[k+1] \left( \ik[0]  \right) = \ddxk[k] \left( e^{-\imath\oy\xi(x)}\xi^{(1)}(x)\Al[0] + \ik[1] \right) \\
        &= \ddxk[k] \left( \ik[1] \right) + \ddxk[k] \left( e^{-\imath\oy\xi(x)}\xi^{(1)}(x)\Al[0] \right) \\
        &= \ddxk[k-1] \left(  e^{-\imath\oy\xi(x)}\xi^{(1)}(x)\Al[1] + \ik[2] \right) + \ddxk[k] \left( e^{-\imath\oy\xi(x)}\xi^{(1)}(x)\Al[0] \right) \\
        &= \ddxk[k-1] \left( \ik[2] \right) + \sum_{l=k-1}^{k}  \left( e^{-\imath\oy\xi(x)}\xi^{(1)}(x)\Al[k-l] \right) \\
        &= \ddxk[k-2] \left( \ik[3] \right) + \sum_{l=k-2}^{k}  \left( e^{-\imath\oy\xi(x)}\xi^{(1)}(x)\Al[k-l] \right) \\
        &= \ddxk[k-3]\left( \ik[4] \right) + \sum_{l=k-3}^{k}{\ddxk[l]\left( e^{-\imath\oy\xi(x)}\xi^{(1)}(x)\Al[k-l] \right)}=\ldots \\
        &= \ddxk[k-j]\left( \ik[j+1] \right) + \sum_{l=k-j}^{k}{\ddxk[l]\left( e^{-\imath\oy\xi(x)}\xi^{(1)}(x)\Al[k-l] \right)}=\ldots \\
        &= \ik[k+1] + \sum_{l=0}^{k}{\ddxk[l]\left( e^{-\imath\oy\xi(x)}\xi^{(1)}(x)\Al[k-l] \right)}.
    \end{align*}

    Therefore,
    \begin{gather}
        \label{derivative-of-psi}
        \ddxk[k+1]\psiOy(x)  = \frac{1}{2\pi} \left( \ik[k+1] + \sum_{l=0}^{k}{\ddxk[l]\left( e^{-\imath\oy\xi(x)}\xi^{(1)}(x)\Al[k-l] \right)} \right).
    \end{gather}

    Proceeding further, we have:
    \begin{align*}
        &\abs{\ddxk[k+1]  \left(\psiOy(x) \right)} = \frac{1}{2\pi} \abs{\left( \ik[k+1] + \sum_{l=0}^{k}{\ddxk[l]
        \left( e^{-\imath\oy\xi(x)}\xi^{(1)}(x)\Al[k-l] \right)} \right)} \\
        &\leq\normalizer\left( \abs{\ik[k+1]} + \abs{\sum_{l=0}^{k}{\ddxk[l]\left( e^{-\imath\oy\xi(x)}\xi^{(1)}(x)\Al[k-l] \right)}} \right) \\
        &\leq\normalizer\left( \abs{\ik[k+1]} + \sum_{l=0}^{k}{\abs{\ddxk[l]\left( e^{-\imath\oy\xi(x)}\xi^{(1)}(x)\Al[k-l] \right)}} \right).
    \end{align*}

    \noindent By applying Leibniz General rule~\cite{olver1993applications} on the sum above we further have
    \begin{gather*}
        \abs{\ddxk[k+1]  \left(\psiOy(x) \right)} \leq \normalizer\abs{\ik[k+1]} + \normalizer\sum_{l=0}^{k} \sum_{m=0}^{l} \binom{l}{m} \abs{\ddxk[l-m]\left( e^{-\imath\oy\xi(x)} \right)}
        \abs{\ddxk[m]\left( \xi^{(1)}(x) A_{k-l}(x) \right)}
    \end{gather*}
    \noindent and now using~\cref{bound-on-composition-exp-xi,bound-for-xi-times-A}
    and denoting $\mathcal{C}_{d}\coloneqq \max\limits_{0\leq k\leq d} \left\{ \abs{c_k} \right\}$, where $c_{k}$ is defined in~\cref{bound-on-composition-exp-xi}
    we get:
    \begin{align*}
        \abs{\ddxk[k+1]  \left(\psiOy(x) \right)}&\leq\normalizer\abs{\ik[k+1]} +
        \frac{A_{x}M_{\xi}}{2\pi}\sum_{l=0}^{k} \sum_{m=0}^{l} \binom{l}{m} \mathcal{C}_{d} 2^{m}\abs{\oy}^{l-m}M_{\xi}^{(l-m)^2}\\
        &\leq \normalizer\abs{\ik[k+1]} +
        \frac{A_{x} \mathcal{C}_{d} M_{\xi}}{2\pi} \sum_{l=0}^{k} 2^{l} M_{\xi}^{l} \sum_{m=0}^{l}
        \binom{l}{m} \left( \abs{\oy}M_{\xi} \right)^{l-m}
    \end{align*}

    Using the known identities
    \begin{gather*}
        \sum_{k=0}^{n}{\binom{n}{k}a^{n-k}} = \left( 1 + a \right)^{n},\quad a\in\mathbb{R}\setminus
        \left\{ 0 \right\} \\
        \sum_{k=0}^{n} 2^{k} \left( 1 + a \right)^{k} = \frac{2^{n+1} \left( 1 + a \right)^{n+1} - 1}{2a+1}
    \end{gather*}
    we proceed to obtain:
    \begin{align*}
        \abs{\ddxk[k+1]  \left(\psiOy(x) \right)}&\leq \normalizer\abs{\ik[k+1]} +
        \frac{A_{x} \mathcal{C}_{d} M_{\xi}}{2\pi} \sum_{l=0}^{k} \left( 2M_{\xi} \right)^{l} \left( 1+ \abs{\oy}M_{\xi} \right)^l\\
        &= \normalizer\abs{\ik[k+1]} +
        \frac{A_{x} \mathcal{C}_{d} M_{\xi}}{2\pi} \frac{\left( 2M_{\xi} \right)^{k+1}(1 + \abs{\oy}M_{\xi})^{k+1} - 1}{2\abs{\oy}M_{\xi} + 1}\\
        &\leq \normalizer\abs{\ik[k+1]} +
        \frac{A_{x} \mathcal{C}_{d} M_{\xi}^{k+2}}{2\pi} \frac{2^{k+1}(1 + \abs{\oy}M_{\xi})^{k+1}}{1+\abs{\oy}M_{\xi}}\\
        &= \normalizer\abs{\ik[k+1]} +
        \frac{A_{x} \mathcal{C}_{d} M_{\xi}^{k+2}}{\pi} 2^{k}(1 + \abs{\oy}M_{\xi})^{k}\\
        &\leq \normalizer\abs{\ik[k+1]} +
        \frac{A_{x} \mathcal{C}_{d} M_{\xi}^{k+2}}{\pi} 2^{k}(2 \abs{\oy}M_{\xi})^{k}\\
        &= \normalizer\abs{\ik[k+1]} +
        \frac{A_{x} \mathcal{C}_{d}}{\pi} 4^{k} M_{\xi}^{2k+2} \abs{\oy}^{k}.
    \end{align*}

    From~\cref{definition-of-Fx-and-jump-location-and-jump-magnitudes-and-assumptions} we have a constant $B_F$ s.t.:
    \begin{gather} \label{bf-bound}
        \abs{\frac{\partial^{k+1}}{\partial x^{k+1}}\left( F(x,y) \right)} \leq B_{F}
    \end{gather}
    which allows us to continue:
    \begin{gather*}
        \abs{I_{\oy, k+1}} \leq \int_{-\pi}^{\xi(x)} \abs{e^{-\imath y \oy}} \abs{\frac{\partial^{k+1}}{\partial x^{k+1}}F(x,y)} \, dy +
        \int_{\xi(x)}^{\pi} \abs{e^{-\imath y \oy}} \abs{\frac{\partial^{k+1}}{\partial x^{k+1}}F(x,y)} \, dy
    \end{gather*}
    \begin{gather*}
        \leq \int_{-\pi}^{\xi(x)} \abs{\frac{\partial^{k+1}}{\partial x^{k+1}}F(x,y)} \, dy +
        \int_{\xi(x)}^{\pi} \abs{\frac{\partial^{k+1}}{\partial x^{k+1}}F(x,y)} \, dy
        \leq \int_{-\pi}^{\xi(x)} B_{F} \, dy + \int_{\xi(x)}^{\pi} B_{F} \, dy = 2\pi B_{F}.
    \end{gather*}
    Denoting
    \begin{gather} \label{Admxi}
         \Admxi \coloneqq \frac{4^{d} \mathcal{C}_d M_{\xi}^{2d+2} }{\pi}
    \end{gather}
    we finally have
    \begin{align*}
        \abs{\ddxk[k+1]\left( \psiOy(x) \right)} &\leq \frac{2\pi B_{F}}{2\pi} +
        \frac{A_{x} \mathcal{C}_{d}}{\pi} 4^{k} M_{\xi}^{2k+2} \abs{\oy}^{k}\\
        &= B_{F} + \frac{A_{x} \mathcal{C}_{d}}{\pi} 4^{k} M_{\xi}^{2k+2} \abs{\oy}^{k} \\
        &\overset{k\leq d}{\leq} B_{F} + \frac{A_{x} \mathcal{C}_{d}}{\pi} 4^{d} M_{\xi}^{2d+2} \abs{\oy}^{d} \\
        &= B_{F} + \Admxi \cdot A_x \cdot \abs{\oy}^{d}.\qedhere
    \end{align*}
\end{proof}

Before presenting~\cref{bound-for-phi-omega-y-derivatives} we need an additional technical result.

\begin{lemma}\label{bound-on-jump-mag-of-psi-for-omega-y}
    Let $\oyInZ$ and $\psiOy$ as in~\cref{def:psi-oy}.
    Then there exist constants $B_F$ (as in~\cref{definition-of-Fx-and-jump-location-and-jump-magnitudes-and-assumptions}) and $\Admxi$ such that:
    \begin{gather}
        \abs{\Aoyl} \leq 2 B_F + 2\Admxi A_x \abs{\oy}^d.
    \end{gather}
\end{lemma}

\begin{proof}
    From~\cref{psi-oy-jump-mag-def} we get:
    \begin{gather*}
        \abs{\Aoyl} = \abs{\psiOy^{(l)}(x)\Big|_{x=-\pi} - \psiOy^{(l)}(x)\Big|_{x=\pi}} \leq \abs{\left( \psiOy^{(l)}(x)\Big|_{x=-\pi}
        \right)} + \abs{\left( \psiOy^{(l)}(x)\Big|_{x=\pi} \right)}.
    \end{gather*}

    Using~\cref{bound-for-derivative-of-psi-omega-y} and~\cref{Admxi}  we have:
    \begin{align*}
        \abs{\Aoyl} &\leq 2\Big( B_F + \frac{A_x\mathcal{C}_d}{\pi} 4^d M_{\xi}^{2d+2}\abs{\oy}^d \Big) \\
        &= 2\Big( B_{F} + \Admxi A_x \abs{\oy}^d \Big),
    \end{align*}
    completing the proof.
\end{proof}

\begin{proposition}
    \label{bound-for-phi-omega-y-derivatives}
    Let $\omega\in\mathbb{Z}$ and $\phiOy$ as defined in~\cref{phi-omega-y} and $k\in\mathbb{N}$, then
    \begin{gather}
        \abs{\dkdxk\phiOy(x)} \leq 2\left( d+2 \right) \Big( B_{F} + \Admxi A_x \abs{\oy}^d \Big),
    \end{gather}
    where $\Admxi$ is defined in~\cref{Admxi}
\end{proposition}

\begin{proof}
    \begin{equation*}
        \begin{split}
            \ddx\phiOy(x) &=\ddx\left( \sum_{l=0}^{d}{\Aoyl \Vpsi(x)} \right) = \sum_{l=0}^{d}{\ddx\left( \Aoyl \Vpsi(x) \right)} \\
            &= \sum_{l=0}^{d}{\Aoyl\ddx\left(\Vpsi(x) \right)} \\
            &= \sum_{l=0}^{d}{\Aoyl\ddx\left( -\frac{(2\pi)^l}{(l+1)!}B_{l+1} \left( \frac{x+\pi}{2\pi} \right) \right)} \\
            &= \sum_{l=0}^{d}{\Aoyl\left( -\frac{(2\pi)^{l-1}}{(l)!}B_{l} \left( \frac{x+\pi}{2\pi} \right) \right)} \\
            &= \sum_{l=0}^{d}{\Aoyl V_{\psi,l-1}(x)} = -\frac{\Aoylarg[0]}{2\pi} + \sum_{l=1}^{d}{\Aoyl V_{\psi,l-1}(x)} \\
            &= -\frac{\Aoylarg[0]}{2\pi} + \sum_{l=0}^{d-1}{\Aoylarg[l+1] V_{\psi,l}(x)}.
        \end{split}
    \end{equation*}

    Computing the second derivative we have:
    \begin{equation*}
        \begin{split}
            \frac{d^2}{dx^2}\phiOy(x) &=\ddx\left( -\frac{\Aoylarg[0]}{2\pi} + \sum_{l=0}^{d-1}{\Aoylarg[l+1] V_{\psi,l}(x)} \right) \\
            &=-\frac{\Aoylarg[1]}{2\pi} + \sum_{l=1}^{d-1}{\Aoylarg[l+1] V_{\psi,l-1}(x)} \\
            &= -\frac{\Aoylarg[1]}{2\pi} + \sum_{l=0}^{d-2}{\Aoylarg[l+2] V_{\psi,l}(x)}.
        \end{split}
    \end{equation*}

    By induction we can show that:

    \begin{equation}
        \label{eq:a.15}
        \dkdxk\phiOy(x) =
        \begin{dcases}
            -\frac{\Aoylarg[k-1]}{2\pi} + \sum_{l=0}^{d-k}{\Aoylarg[l+k] V_{\psi,l}(x)} ,&\text{if } k \leq d \\
            -\frac{\Aoylarg[d]}{2\pi} ,&\text{if } k = d +1 \\
            0,              &\text{otherwise}
        \end{dcases}
    \end{equation}

    With the above we now can estimate
    \begin{align*}
        \begin{split}
            \abs{\dkdxk\phiOy(x)} &\leq \abs{-\frac{\Aoylarg[k-1]}{2\pi} + \sum_{l=0}^{d-k}{\Aoylarg[l+k] V_{\psi,l}(x)}} \\
            &\leq \frac{\abs{\Aoylarg[k-1]}}{2\pi} + \sum_{l=0}^{d-k}{\abs{\Aoylarg[l+k]} \abs{V_{\psi,l}(x)}} \\
            &\leq \frac{\abs{\Aoylarg[k-1]}}{2\pi} + \sum_{l=0}^{d-k}{\abs{\Aoylarg[l+k]}} \\
            &\leq \abs{\Aoylarg[k-1]} + \sum_{l=0}^{d-k}{\abs{\Aoylarg[l+k]}} = \sum_{l=k-1}^{d}{\abs{\Aoylarg[l]}}.
        \end{split}
    \end{align*}
    Using~\cref{bound-on-jump-mag-of-psi-for-omega-y} we conclude the proof:
    \begin{gather*}
        \abs{\dkdxk\phiOy(x)} \leq 2\left( d+2 \right) \Big( B_{F} + \Admxi A_x \abs{\oy}^d \Big).\qedhere
    \end{gather*}
\end{proof}

\begin{proof}[Proof of ~\cref{bound-for-gamma-omega-y-derivatives}]
From~\cref{eq:decomposition-psi-omega-y} we have:
\begin{equation*}
    \ddxk[k+1]\left( \gamma_{\omega_y} \right) =\frac{d^{k+1}}{dx^{k+1}} \left( \psiOy \right) - \ddxk[k+1]\left( \phiOy \right)
\end{equation*}

From~\cref{bound-for-derivative-of-psi-omega-y,bound-for-phi-omega-y-derivatives} we have:
\begin{equation*}
    \begin{split}
    &\abs{\ddxk[k+1]\left( \gammaOy \right)} \leq \abs{\frac{d^{k+1}}{dx^{k+1}}\left( \psiOy \right)} + \abs{\frac{d^{k+1}}{dx^{k+1}}
    \left( \phiOy \right)}\leq\\
    &\leq B_{F} + \Admxi A_x \abs{\oy}^d + 2(d+2)\Big( B_{F} + \Admxi A_x \abs{\oy}^d \Big) \\
    &\leq \left(2d + 5\right) \left(B_{F} + \Admxi A_x \abs{\oy}^d \right)
    \end{split}.
\end{equation*}
Therefore,
\begin{equation}
    \label{bound-for-gamma-omega-y-for-psi}
    \abs{\ddxk[k+1]\left( \gamma_{\oy} \right)} \leq \left(2d + 5\right) \left(B_{F} + \Admxi A_x \abs{\oy}^d \right),
\end{equation}
completing the proof of ~\cref{bound-for-gamma-omega-y-derivatives}.
\end{proof}

\subsection{Condition for  discontinuous $\psiOy$}\label{subsec:proof-of-prop-18}
Now we demonstrate that $\psiOy$ may generically have a discontinuity at the domain boundary  $(x=\pm\pi)$. We focus on a simple example where the discontinuity curve is $\xi(x)=x$. Substitution into \cref{derivative-of-psi} entails:
\begin{gather*}
    \begin{split}
        2\pi \frac{d^{k+1}}{dx^{k+1}}\psiOy(x) &= \sum_{l=0}^{k}{\ddxk[l]\left( e^{-\imath\oy x} \Al[k-l] \right)} + \\
        &\int_{-\pi}^{x} e^{-\imath \oy y} \frac{\partial^{k+1}}{\partial x^{k+1}} F(x,y) \, dy +
        \int_{x}^{\pi} e^{-\imath \oy y} \frac{\partial^{k+1}}{\partial x^{k+1}} F(x,y) \, dy.
    \end{split}
\end{gather*}
Using Leibniz general rule~\cite{olver1993applications} and $e^{-\imath \oy \pi} = e^{\imath \oy \pi} = (-1)^{\oy},\; \oyInZ $ we have:
\begin{gather*}
    \begin{split}
        2\pi \frac{d^{k+1}}{dx^{k+1}}\psiOy(x) &= (-1)^{\oy} \sum_{l=0}^{k} \sum_{m=0}^{l} \binom{l}{m} (-\imath\oy)^{l-m}A_{k-l}^{(m)}(x)+ \\
        &\int_{-\pi}^{x} e^{-\imath \oy y} \frac{\partial^{k+1}}{\partial x^{k+1}} F(x,y) \, dy +
        \int_{x}^{\pi} e^{-\imath \oy y} \frac{\partial^{k+1}}{\partial x^{k+1}} F(x,y) \, dy.
    \end{split}
\end{gather*}

Using~\cref{psi-oy-jump-mag-def} we have:
\begin{gather*}
    \begin{split}
        2\pi A_{k+1}^{\psi}(\oy) &= (-1)^{\oy} \sum_{l=0}^{k} \sum_{m=0}^{l} \binom{l}{m} (-\imath\oy)^{l-m} A_{k-l}^{(m)}(-\pi) \\
        &\qquad\qquad + \int_{-\pi}^{-\pi} e^{-\imath \oy y} \frac{\partial^{k+1}}{\partial x^{k+1}} F(x,y)\Big\vert_{x=-\pi} \, dy \\
        &\qquad\qquad - (-1)^{\oy}\sum_{l=0}^{k} \sum_{m=0}^{l} \binom{l}{m} (-\imath\oy)^{l-m}A_{k-l}^{(m)}(\pi) \\
        &\qquad\qquad - \int_{-\pi}^{\pi} e^{-\imath \oy y} \frac{\partial^{k+1}}{\partial x^{k+1}} F(x,y)\Big\vert_{x=\pi} \, dy\\
        &=(-1)^{\oy} \sum_{l=0}^{k} \sum_{m=0}^{l} \binom{l}{m} (-\imath\oy)^{l-m}
        \left(A_{l}^{(m)}(-\pi) - A_{l}^{(m)}(\pi)\right)+ \\
        &\int_{-\pi}^{\pi} e^{-\imath \oy y} \left( \frac{\partial^{k+1}}{\partial x^{k+1}} F(x,y)\Big\vert_{x=-\pi} -
        \frac{\partial^{k+1}}{\partial x^{k+1}} F(x,y)\Big\vert_{x=\pi} \right) \, dy.
    \end{split}
\end{gather*}
Denoting:
\begin{gather*}
    \frac{\partial^{k+1}}{\partial x^{k+1}} F(x,y) \coloneqq F_{x}^{(k+1)}(x,y)
\end{gather*}
we get:
\begin{gather*}
    \begin{split}
        2\pi A_{k+1}^{\psi}(\oy) &= (-1)^{\oy} \sum_{l=0}^{k} \sum_{m=0}^{l} \binom{l}{m} (-\imath\oy)^{l-m}
        \left(A_{l}(-\pi) - A_{l}(\pi)\right)+ \\
        &\int_{-\pi}^{\pi} e^{-\imath \oy y} \left(F_{x}^{(k+1)}(-\pi,y) - F_{x}^{(k+1)}(\pi,y) \right) \, dy.
    \end{split}
\end{gather*}
Now if $\psiOy$ has $ \doy $ derivatives then in order to have $A_{k+1}^{\psi}(\omega_y)=0$ for each $0\leq k\leq \doy-1$ we must have:
\begin{gather*}
    \begin{split}
        (-1)^{\oy} \sum_{l=0}^{k} &\sum_{m=0}^{l} \binom{l}{m} (-\imath\oy)^{l-m} \left(A_{l}^{(m)}(-\pi) - A_{l}^{(m)}(\pi)\right)+ \\
        &\int_{-\pi}^{\pi} e^{-\imath \oy y} \left(F_{x}^{(k+1)}(-\pi,y) - F_{x}^{(k+1)}(\pi,y) \right) \, dy = 0.
    \end{split}
\end{gather*}
If we assume that all derivatives $F_x^{l}$, $l=0,\dots,l$ are periodic (in $x$) for all $y$, then we have:
\begin{gather*}
    \int_{-\pi}^{\pi} e^{-\imath \oy y} \left(F_{x}^{(k+1)}(-\pi,y) - F_{x}^{(k+1)}(\pi,y) \right) \, dy = 0
\end{gather*}
which leaves us with:
\begin{gather*}
   A_{k}^{\psi}(\oy) = 0 \Leftrightarrow  \sum_{l=0}^{k} \sum_{m=0}^{l} \binom{l}{m} (-\imath\oy)^{l-m} \left(A_{l}^{(m)}(-\pi) - A_{l}^{(m)}(\pi)\right) = 0.
\end{gather*}
Denoting:
\begin{gather*}
    \trilm[l]{m}\coloneqq A_{l}^{(m)}(-\pi)-A_{l}^{(m)}(\pi)
\end{gather*}
we claim:
\begin{gather*}        A_{k+1}^{\psi}(\oy) = 0 \Leftrightarrow \sum_{l=0}^{k} \sum_{m=0}^{l} \binom{l}{m} \trilm[l]{m} (-\oy)^{l-m} = 0.
\end{gather*}
Also notice that:
\begin{align*}
    A_{1}^{\psi}(\oy) &= \trilm[0]{0} \\
    A_{2}^{\psi}(\oy) &= \trilm[0]{0} + \trilm[1]{0}(-\imath\oy) + \trilm[1]{1} = A_{1}^{\psi}(\oy) + \trilm[1]{0}(-\imath\oy) + \trilm[1]{1}\\
    A_{3}^{\psi}\oy) &= \trilm[0]{0} + \trilm[1]{0}(-\imath\oy) + \trilm[1]{1} + \trilm[2]{0}(-\imath\oy)^{2} + \trilm[2]{1}(-\imath\oy) + \trilm[2]{2} \\
    &= A_{2}^{\psi}(\oy) + \trilm[2]{0}(-\imath\oy)^{2} + \trilm[2]{1}(-\imath\oy) + \trilm[2]{2} \\
    &\;\; \vdots \\
    A_{k+1}^{\psi}(\oy) &= A_{k}^{\psi}(\oy) + \sum_{m=0}^{k} \trilm[k]{m}\binom{k}{m}(-\imath\oy)^{k-m}
\end{align*}

So in order for $A_{k}^{\psi}(\oy) = 0$ to take place we would have to demand that $A_{k}^{\psi}(\oy) = -\sum_{m=0}^{k}
\trilm[k]{m}\binom{k}{m}(-\imath\oy)^{k-m} = \sum_{m=0}^{k} \trilm[k]{m}(-1)^{k-m+1}\binom{k}{m}(-\imath\oy)^{k-m}$ and by denoting:
\begin{align*}
    \trivec &\coloneqq \left( \trilm[k]{0},\trilm[k]{1},\ldots,\trilm[k]{k-1},\trilm[k]{k} \right) \\
    \omegavec &\coloneqq \left( (-1)^{k+1}\binom{k}{0}(\imath\oy)^{k}, (-1)^{k}\binom{k}{1}(\imath\oy)^{k-1}, \ldots,
    (-1)^{2}\binom{k}{k-1}(\imath\oy)^{k}, -1 \right)
\end{align*}
we finally write a necessary and sufficient condition for the absence of the boundary discontinuity of $\psiOy$:
\begin{gather*}
    A_{k+1}^{\psi}(\oy) = 0 \iff A_{k}^{\psi}(\oy) = \trivec[k] \cdot \omegavec[k].
\end{gather*}
Suppose $A_{0}^{\psi}(\oy) = \ldots = A_{k}^{\psi}(\oy) = 0$ then for $A_{k+1}^{\psi}(\oy) = 0$ we need
$\trivec[k] \perp \omegavec[k]$. In more detail $A_{0}^{\psi}(\oy) = 0$ implies
$A_{1}^{\psi}(\oy) = 0 \iff \trivec[0] \perp \omegavec[0]$ and then $A_{2}^{\psi}(\oy) = 0 \iff \trivec[1] \perp \omegavec[1]$ and so on.

The above is summarized in the following proposition:
\begin{proposition}
    \label{case-for-discontinuity-for-psi}
    Let $F(x,y)$ as in~\cref{description-of-F-and-xi} and $\psiOy$ as in~\cref{def-of-psi-omega-y,psi-oy-properties}.
    Assume that $\frac{d^{k}}{dx^{k}}F(x,y)\Big\vert_{x=x_0+2\pi n} = \frac{d^{k}}{dx^{k}}F(x,y)\Big\vert_{x=x_0},\quad n\in \mathbb{Z}$, and
    $A_{0}(\oy)=0$. Then:
    \begin{gather*}
        A_{k}^{\psi}(\oy) = 0 \iff \trivec[k-1] \perp \omegavec[k-1],\quad k = 1,\ldots, \doy.
    \end{gather*}
\end{proposition}

\end{document}